\documentclass{amsart}
\usepackage{amsfonts,amsthm,amsmath}
\usepackage{amssymb}
\usepackage{amscd}
\usepackage[utf8]{inputenc}
\usepackage[a4paper]{geometry}
\usepackage{enumerate}
\usepackage[usenames,dvipsnames]{color}
\usepackage{array}
\usepackage{array,tabularx}

\numberwithin{equation}{section}
\numberwithin{equation}{subsection}

\theoremstyle{plain}

\newtheorem{theorem}[equation]{Theorem}
\newtheorem{lemma}[equation]{Lemma}
\newtheorem{proposition}[equation]{Proposition}

\newtheorem{corollary}[equation]{Corollary}

\theoremstyle{definition}
\newtheorem{example}[equation]{Example}
\newtheorem{remark}[equation]{Remark}

\newtheorem{definition}[equation]{Definition}

\newtheorem{bekezdes}[equation]{}

\def\C{\mathbb C}
\def\Q{\mathbb Q}

\def\Z{\mathbb Z}

\def\im{{\rm im}}

\newcommand{\cale}{{\mathcal E}}

\newcommand{\calv}{{\mathcal V}}

\newcommand{\calm}{{\mathcal M}}

\newcommand{\calG}{{\mathcal G}}
\newcommand{\cali}{{\mathcal I}}

\newcommand{\calO}{{\mathcal O}}

\newcommand{\calS}{{\mathcal S}}

\newcommand{\calL}{\mathcal{L}}

\newcommand{\tX}{\widetilde{X}}

\newcommand{\reg}{{\rm reg}}

\newcommand{\cO}{{\mathcal O}}
\newcommand{\bP}{{\mathbb P}}

\newcommand{\bC}{{\mathbb C}}

\newcommand{\eca}{{\rm ECa}}
\newcommand{\pic}{{\rm Pic}}

\newcommand{\m}{\mathfrak{m}}\newcommand{\fr}{\mathfrak{r}}
\newcommand{\mfl}{\mathfrak{L}}

\newcommand{\frII}{\mathfrak{I}}

\newcommand{\bZ}{{\mathbb{Z}}}
\newcommand{\bQ}{{\mathbb{Q}}}

\title{The multiplicity of \\ generic normal surface singularities}

\author{J\'anos Nagy}
\address{Alfr\'ed R\'enyi Institute of Mathematics,
Re\'altanoda utca 13-15, H-1053,  Budapest, Hungary}
\email{janomo@renyi.hu}

\author{Andr\'as N\'emethi}
\address{Alfr\'ed R\'enyi Institute of Mathematics,
Re\'altanoda utca 13-15, H-1053, Budapest, Hungary \newline
 \hspace*{4mm} ELTE - University of Budapest, Dept. of Geometry, Budapest, Hungary \newline \hspace*{4mm}
BCAM - Basque Center for Applied Math.,
Mazarredo, 14 E48009 Bilbao, Basque Country – Spain}
\email{nemethi.andras@renyi.hu }

\thanks{The authors are supported by the NKFIH Grant ``\'Elvonal (Frontier)'' KKP 126683.}

\begin{document}

\keywords{normal surface singularities, generic analytic structure, links of singularities, Picard group,
natural line bundles,
 rational homology spheres, geometric genus, multiplicity.}

\subjclass[2010]{Primary. 32S05, 32S25, 32S50,
Secondary. 14Bxx, 14J80}

\begin{abstract}
We provide combinatorial/topological formula for the multiplicity of a complex analytic normal surface singularity
whenever the analytic structure on the fixed topological type is generic.
\end{abstract}

\maketitle

\linespread{1.2}


\pagestyle{myheadings} \markboth{{\normalsize  J. Nagy, A. N\'emethi}} {{\normalsize Multiplicity}}


\section{Introduction}\label{s:intr}

\subsection{The `multiplicity problem'. }
Probably the most fundamental numerical invariant of a projective variety $X$ embedded in some
projective space $\bP^N$ is its degree. Its local analogue, defined for local (algebraic or analytic) germs
$(X,o)$ is the multiplicity ${\rm mult}(X,o)$ of $(X,o)$. If $(X,o)$ is embedded in some $(\C^N,o)$
then it is the smallest intersection multiplicity of $(X,o)$ with a linear subspace germ $(L,o)$
of dimension $N-\dim (X,o)$. It is independent of the embedding $(X,o)\subset (\C^N,o)$, it can also be defined
via the Hilbert--Samuel function of the maximal ideal $\m_{X,o}\subset \calO_{X,o}$, cf.  \ref{bek:HSF}.
 By definition it is an analytic invariant, and it guides several central geometric problems.

 E.g., besides its geometric significance as the `local degree' of $(X,o)$, which obstructs (guides) the structure of
 analytic functions defined on $(X,o)$, it is the key numerical invariant  of
  several objects associated canonically to $(X,o)$.
See e.g. the significance  of the multiplicity of the polar curve or of the discriminant in the case of
 hypersurface singularities   \cite{Te0,Te},
or the multiplicities of the $\delta$--constant (Severi) strata of the deformation of a plane curve singularity \cite{FGS,Shende}.

In this note we focus on the multiplicity of a complex analytic normal surface singularity.
The guiding  question is whether the multiplicity is computable from the topology of the link.
The topology of the link (as an oriented 3--manifold with usually `large' fundamental group) contains a huge amount of
information, however the problem is still difficult. E.g., there are examples of local, topologically constant
deformations  when the multiplicity jumps (see e.g. the examples from section \ref{s:examples}, when any
analytic type can be deformed into a generic one). Moreover,  there are `easy' pairs of
examples, even hypersurface singularities,
with the same topology but different multiplicity (e.g. $\{x^2+y^7+z^{14}=0\}$ and $\{x^3+y^4+z^{12}=0\}$.)
In such cases of  pairs of hypersurface singularities  the common link is not a rational homology sphere.
Therefore, it is natural to impose for the link
to be  a rational homology sphere (that is, in a resolution of $(X,o)$ all the exceptional curves are
rational and the dual graph is a tree).

The problem can be compared with the famous Zariski's Conjecture \cite{ZaOpen},
 which asks whether the multiplicity of an isolated hypersurface singularity $(X,o)\subset (\C^{n+1},o)$ can be
 recovered from the embedded topological type, that is, from the smooth embedding $link(X)\subset S^{2n+1}$. Except
 for some particular families the answer is not known yet, it is open even for surface singularities.
 For a survey see \cite{E} (and the references therein).
 Note that our projects wishes to connect the multiplicity  merely with  the abstract link
 (but under the assumption that
 the link is a rational homology sphere).
 In fact, in \cite{MeNe} it is conjectured that for isolated hypersurface surface singularities with
 rational homology sphere link the abstract link determines the multiplicity (it was verified  in the
 suspension case in \cite{MeNe} and for germs with non--degenerate Newton principal parts in \cite{BN}).
 Note that for hypersurfaces the multiplicity is the smallest degree of the monomials from its equation, still,
 to recover this number from the topology can be hard.

 If the normal surface singularity $(X,o)$ is not a hypersurface then the situation is even harder: it might happen
 that the topological type carries many rather different families of analytic structures.

 On the other hand, there are some `positive example/families' as well. Artin in \cite{Artin62,Artin66}
characterized rational singularities topologically and  determined the multiplicity explicitly from the graph.
This was extended by  Laufer  in \cite{Laufer77} to minimally elliptic singularities, and
extended further for
Gorenstein elliptic singularities in \cite{weakly}. For splice quotient singularities (cf. \cite{NWsq}),
a family which includes weighted homogeneous germs as well,
the multiplicity was determined topologically in \cite{NCL};
for abelian covers of  splice quotient singularities in \cite{O15}.
Otherwise the literature is rather restrictive about any kind of
multiplicity formulae. (Here we might mention recent connections with the bi--Lipschitz geometry,
however bi--Lipschitz property is an analytic property, stronger than the abstract topological type).

In \cite{Wa70} Wagreich proved that in the presence of a resolution $\tX\to X$,
if $Z_{max} $ is the `maximal ideal cycle' (of Yau \cite{Yau1}),  and $\calO_{\tX}(-Z_{max})$ has no base points,
then ${\rm mult}(X,o)=-Z_{max}^2$. Here there are two difficulties: to determine $Z_{max}$, and to characterize the
base points of $\calO_{\tX}(-Z_{max})$.

\subsection{} In the present note, instead of certain peculiar families,  we focus on the `generic  analytic structures'.
 We fix a topological type, say a dual graph $\Gamma$,
and we determine the multiplicity of a singularity $(X,o)$, which has a resolution
$\tX$ with dual graph $\Gamma$,  and $\tX$ carries a generic analytic structure.
(It turns out that the expression is independent of the choice of $\Gamma$ up to the natural blow up  of the graph.)
Note that the moduli space of analytic structures supported on $\Gamma$ are not known, we will
use the parameter space of local deformations of Laufer \cite{LaDef1} to define the `generic analytic structure'.

For generic analytic structures in \cite{NNII} we already determined several analytic invariant topologically.
That package of results basically concentrated on the cohomology of (certain natural) line bundles.
It was a continuation of \cite{NNI}, where the Abel map of resolution of normal surface singularities was
introduced and treated. The  article  \cite{NNI}
 creates that new mathematical machinery,  which can handle the subtle analytic invariants of
line bundles. In \cite{NNII} the cycle $Z_{max}$ was already determined for the generic analytic structure
(together with `analytic semigroup' of divisors of analytic functions of $(X,o)$).
In the present note we characterize topologically the  base points of  $\calO_{\tX}(-Z_{max})$ as well.
In this
topological characterization we use  the Riemann--Roch expression $\chi(l)$, (defined for cycles $l$
 supported on the exceptional curve). For the definition of $\chi$ see \ref{ss:topol}.

For $\tX$ generic, and $(X,o)$ non--rational,
 $Z_{max}$ is determined as follows (\cite{NNII}, or Theorem \ref{th:OLD} below).
Set $\calm=\{ Z\,:\, \chi(Z)=\min _{l\in L}\chi(l)\}$.
Then
the unique maximal element of $\calm$ is the maximal ideal cycle of\, $\tX$.

 The next theorem provides the structure of base points (for  more general versions see
 Theorems \ref{th:NEW1} and \ref{th:NEW2} or Proposition \ref{prop:GENERAL} below).

  \begin{theorem}\label{th:NEW1Intr}
Consider a resolution $\tX\to X$  with generic analytic structure. Let $E$ be the exceptional curve $\cup_{v\in\calv}E_v$.
We say that the irreducible component  $E_v$  satisfies the property $(*_v)$ if
\begin{equation*}
(*_v) \hspace{1cm} \min_{l\geq E_v}\{\chi(Z_{max}+l)\}=\chi(Z_{max})+1.\hspace{2cm}
\end{equation*}
Then the following facts hold.

$\bullet$ \ If $p$ is a base point of $\calL:=\calO_{\tX}(-Z_{max})$ then $p$ is a regular point of $E$.

$\bullet$ \ If $p\in E_v$ is a base point of $\calL$  then $E_v$ satisfies $(Z_{max}, E_v)<0$ and  the property $(*_v)$.

$\bullet$ \ If $(Z_{max},E_v)<0$ and $E_v$ satisfy $(*_v)$ then $\calL$
has  exactly $-(Z_{max},E_v)$ base points on $E_v$. 

In this case define  $t(v):=m_v^+-m_v$, where $m_v$ is the
$E_v$--coefficient of $Z_{max}$ and
$$
m_v^+=\max\{\,E_v\mbox{--coefficient of } Z_{max}+l \, :\, \mbox{where} \ l\geq E_v, \ l\in L, \ \chi(Z_{max}+l)=\chi(Z_{max})+1\}.
$$
Then the ideal of each base point $p$ on $E_v$ has the uniform local type $(x^{t(v)},y)$, where $(x,y)$ are local coordinates at $p$
and  $\{x=0\}$ is the local equation of $(E,p)$.
%
%
%
\noindent In particular,
\begin{equation}
{\rm mult} (X,o)=-Z_{max}^2 -\sum_v t(v)\cdot  (Z_{max},E_v),
\end{equation}
where the sum is over all $v\in \calv$ with $(Z_{max},E_v)<0$ and
$\min_{l\geq E_v}\chi(Z_{max}+l)=\chi(Z_{max})+1$.
\end{theorem}

\subsection{} Note that if we blow up the resolution graph $\Gamma$ we get a new graph, which determines the same
topological type of $(X,o)$.
If we associate generic analytic structures to both graphs then the structure of the base points can
be  identified isomorphically.  (However,  if we blow up a base point of a generic analytic
structure, then we modify  the type of the  base point, but the analytic structure obtained by blow up  the base point
will be not generic on its
supporting topological type.) For details see Remark \ref{rem:*}(b).

\subsection{} In fact, our results are more general. In order to be able to run an induction in the proof,
we need to consider a relative case of resolutions $\tX\subset \tX_{top}$, where $\tX_{top}$ is a fixed resolution
space with dual graph $\Gamma_{top}$,
and $\tX$ is a convenient
small neighbourhood of exceptional curves given by a subgraph $\Gamma$ of $\Gamma_{top}$. Furthermore, we
consider several line bundles as well:  the restrictions of the `natural line bundles' from
$\tX_{top}$  level  (with some positivity restriction regarding their Chern classes).
\subsection{}
The idea of the proof of our main theorem and of some other inductive procedures  have their origin in \cite{R}, where
the first author extends   the results of \cite{NNII} (valid for generic germs)
for the `relatively generic singularities'.
In this setup  one  fixes two resolution graphs $\Gamma_1 \subset \Gamma$ and a singularity $\tX_1$ with graph
$\Gamma_1$ and one studies a generic surface singularity $\tX$ with graph $\Gamma$ such that it has the sub--singularity $\tX_1$.
Although we follow intuitively the ideas and techniques developed in \cite{R},
 we try to be self-contained and use only a few results from \cite{R} without proof.

\subsection{} The structure of the article is the following.
In section \ref{s:prel} we collect preliminary definitions and  lemmas, and we recall the definition of (restricted)
natural line bundles. In section \ref{s:AnGen} we review the definition of the generic analytic structure
(based on the work of Laufer) and several results from \cite{NNII} regarding invariants for generic analytic structures.
Here we state the new results regarding the structure of base points
as well (Theorems \ref{th:NEW1} and \ref{th:NEW2}) formulated
in the general case of natural line bundles. Both theorems  are divided into five steps (geometric statements)
{\it (1')--(5')}.
The proof of {\it (1')} is already in this section.
Section \ref{s:ECD} contains  a review of  certain needed  material regarding the Abel maps from \cite{NNI}.
 Part {\it (2')}  is proved in section \ref{s:Proof2},
{\it (3')--(4')}  in section \ref{s:Proof3}, while
{\it (5')} in section \ref{s:Proof5}, after a review in section \ref {s:Prel5} of  the description of the
Abel map in terms of differential forms via Laufer's duality.
Section \ref{s:examples} contains some examples, which support the
theory. Section \ref{s:GEN-GEN} shows that the statements of the main results (formulated for
natural line bundles of generic singularities) remain valid for generic line bundles of arbitrary singularities  as well
(modulo a necessary assumption).
Here we explain also the expected relationship between natural line bundles of generic singularities and
the generic line bundles of arbitrary singularities.

\section{Preliminaries}\label{s:prel}

\subsection{The resolution}\label{ss:notation}
Let $(X,o)$ be the germ of a complex analytic normal surface singularity,
 and let us fix  a good resolution  $\phi:\widetilde{X}\to X$ of $(X,o)$.
We denote the exceptional curve $\phi^{-1}(0)$ by $E$, and let $\{E_v\}_{v\in\calv}$ be
its irreducible components. Set also $E_I:=\sum_{v\in I}E_v$ for any subset $I\subset \calv$.
For the cycle $l=\sum n_vE_v$ let its support be $|l|=\cup_{n_v\not=0}E_v$.
For more details see \cite{trieste,NCL,Nfive}.
\subsection{Topological invariants}\label{ss:topol}
Let $\Gamma$ be the dual resolution graph
associated with $\phi$;  it  is a connected graph.
Then $M:=\partial \widetilde{X}$, as a smooth oriented 3--manifold,
 can be identified with the link of $(X,o)$, it is also
an oriented  plumbed 3--manifold associated with $\Gamma$.
{\it We will assume  (for any singularity we will deal with) that the link
 $M$ is a rational homology sphere,}
or, equivalently,  $\Gamma$ is a tree with all genus
decorations  zero. We use the same
notation $\mathcal{V}$ for the set of vertices.

The lattice $L:=H_2(\widetilde{X},\mathbb{Z})$ is  endowed
with a negative definite intersection form  $I=(\,,\,)$. It is
freely generated by the classes of 2--spheres $\{E_v\}_{v\in\mathcal{V}}$.
 The dual lattice $L':=H^2(\widetilde{X},\mathbb{Z})$ is generated
by the (anti)dual classes $\{E^*_v\}_{v\in\mathcal{V}}$ defined
by $(E^{*}_{v},E_{w})=-\delta_{vw}$, the opposite of the Kronecker symbol.
The intersection form embeds $L$ into $L'$. Then $H_1(M,\mathbb{Z})\simeq L'/L$, abridged by $H$.
Usually one also identifies $L'$ with those rational cycles $l'\in L\otimes \Q$ for which
$(l',L)\in\Z$ (or, $L'={\rm Hom}_\Z(L,\Z)\simeq H^2(\tX,\mathbb{Z})$), where the intersection form extends naturally.

All the $E_v$--coordinates of any $E^*_u$ are strict positive.
We define the Lipman cone as $\calS':=\{l'\in L'\,:\, (l', E_v)\leq 0 \ \mbox{for all $v$}\}$.
It is generated over $\bZ_{\geq 0}$ by $\{E^*_v\}_v$.
Hence, if $l'\in \calS\setminus \{0\}$ then all its $E_v$--coefficients are strict positive.
We also write $\calS:=\calS'\cap L$.

There is a natural partial ordering of $L'$ and $L$: we write $l_1'\geq l_2'$ if
$l_1'-l_2'=\sum _v r_vE_v$ with all $r_v\geq 0$. We set $L_{\geq 0}=\{l\in L\,:\, l\geq 0\}$ and
$L_{>0}=L_{\geq 0}\setminus \{0\}$.
We will write $Z_{min}\in L$ for the  {\it minimal} (or fundamental, or Artin) cycle, which is
the minimal non--zero cycle of $\calS$ \cite{Artin62,Artin66}.

We define the
  (anti)canonical cycle $Z_K\in L'$ via the {\it adjunction formulae}
$(-Z_K+E_v,E_v)+2=0$ for all $v\in \mathcal{V}$.
(In fact,  $Z_K=-c_1(\Omega^2_{\widetilde{X}})$, cf. (\ref{eq:PIC})).
In a minimal resolution $Z_K\in \calS'$.

Finally we consider the Riemann--Roch expression
 $\chi(l')=-(l',l'-Z_K)/2$ defined for any $l'\in L'$.

\subsection{Some analytic invariants. The Picard groups.}\label{ss:analinv}
The group ${\rm Pic}(\widetilde{X})$
of  isomorphism classes of analytic line bundles on $\widetilde{X}$ appears in the (exponential) exact sequence
\begin{equation}\label{eq:PIC}
0\to {\rm Pic}^0(\widetilde{X})\to {\rm Pic}(\widetilde{X})\stackrel{c_1}
{\longrightarrow} L'\to 0, \end{equation}
where  $c_1$ denotes the first Chern class. Here
$ {\rm Pic}^0(\widetilde{X})=H^1(\widetilde{X},\calO_{\widetilde{X}})\simeq
\C^{p_g}$, where $p_g$ is the {\it geometric genus} of
$(X,o)$. $(X,o)$ is called {\it rational} if $p_g(X,o)=0$.
 Artin in \cite{Artin62,Artin66} characterized rationality topologically
via the graphs; such graphs are called `rational'. By this criterion, $\Gamma$
is rational if and only if $\chi(l)\geq 1$ for any effective non--zero cycle $l\in L_{>0}$.

Similarly, if $Z\in L_{>0}$ is a non--zero effective integral cycle such that its support is $|Z| =E$,
and $\calO_Z^*$ denotes
the sheaf of units of $\calO_Z$, then ${\rm Pic}(Z)=H^1(Z,\calO_Z^*)$ is  the group of isomorphism classes
of invertible sheaves on $Z$. It appears in the exact sequence
  \begin{equation}\label{eq:PICZ}
0\to {\rm Pic}^0(Z)\to {\rm Pic}(Z)\stackrel{c_1}
{\longrightarrow} L'\to 0, \end{equation}
where ${\rm Pic}^0(Z)=H^1(Z,\calO_Z)$.
If $Z_2\geq Z_1$ then there are natural restriction maps,
${\rm Pic}(\widetilde{X})\to {\rm Pic}(Z_2)\to {\rm Pic}(Z_1)$.
Similar restrictions are defined at  ${\rm Pic}^0$ level too.
These restrictions are homomorphisms of the exact sequences  (\ref{eq:PIC}) and (\ref{eq:PICZ}).

\bekezdes \label{bek:BASE} {\bf Fixed components and  base points of line bundles.}
Fix some $Z\in L_{>0}$ with $|Z|=E$ and $\calL\in {\rm Pic}(Z)$. We say that $E_v$ is a fixed component
of $\calL$ if the natural inclusion $H^0(Z-E_v,\calL(-E_v))\hookrightarrow  H^0(Z,\calL)$ is an isomorphism.
In particular, $\calL$ has {\it no fixed components} at all if
\begin{equation}\label{eq:H_0}
H^0(Z,\calL)_{\reg}:=H^0(Z,\calL)\setminus \bigcup_v H^0(Z-E_v, \calL(-E_v))
\end{equation}
is non--empty. Let us use the same notation $\calL$ for the sheaf of sections of the
 line bunle $\calL$. If $\calL$ has no fixed components then there exists a sheaf of ideals $\cali_{\calL}$
 of $\calO_{\tX}$ such that $H^0(\tX,\calL)\cdot \calO_{\tX}=\calL\cdot \cali_{\calL}$, and
 $\cali_{\calL}$ is supported at  finitely many points of $E$. These are the base points of $\calL$.

We will refer to the next elementary lemma many times.

 \begin{lemma}\label{lem:BASE}
Assume  that $\calL\in {\rm Pic}(\tX)$ has no fixed components and $p\in E$ is a base point.
Let $b:\tX^{new}\to \tX$ be the blow up of $\tX$ at $p$  and set $E^{new}=b^{-1}(p)$.  Then

(a) if $p\in E_v$ then $(c_1(\calL),E_v)>0$,

(b) $H^0(\tX,\calL)=H^0(\tX^{new},b^*\calL)=H^0(\tX^{new},b^*\calL(-E^{new}))$,

(c) $h^1(\tX,\calL)=h^1(\tX^{new},b^*\calL)=h^1(\tX^{new},b^*\calL(-E^{new}))-1$.
 \end{lemma}
\begin{proof} Let $\m_{\tX,p}$ denote the maximal ideal of the local algebra $\calO_{\tX,p}$.
{\it (a)} If   $(c_1(\calL),E_v)\leq 0$ then comparison of the exact sequence
$0\to H^0(\m_p\calL)\to H^0(\calL)\to \C_p$ with $0\to H^0(\calL(-E_v))\to H^0(\calL)\to
H^0(\calL|_{E_v})$ would imply that $E_v$ is a fixed component.
For {\it (b)}--{\it (c)} notice that $R^0b_*(b^*\calL)=\calL$ and $R^1b_*(b^*\calL)=0$,
hence by Leray spectral sequence $H^*(\tX^{new}, b^*\calL)=H^*(\tX,\calL)$. Then identify
$0\to H^0(\tX,\m_{\tX,p}\calL)\to H^0(\tX,\calL)\to \C_p$ with
$0\to H^0(\tX^{new},b^*\calL(-E^{new}))\to H^0(\tX^{new}, b^*\calL)\to \C$.
\end{proof}

\begin{definition}\label{def:BASE}
A base point $p$ of $\calL$ is called of  $A_t$--type (for some integer $t\geq 1$)  if  $p$
is a regular point of $E$  and $\cali_{\calL,p}$  in the
local ring $\calO_{\tX,p}$ is $(x^t,y)$, where $x,y$ are some local coordinates of $(\tX,p)$
at $p$  with $\{x=0\}=E$ (locally).
We say that $p$ is of $A$--type
if it is $A_t$--type  for some $t\geq 1$. In such cases we write $t=t(p)$.
Note that  $A_1$--type means   $\cali_{\calL,p}= \m_{\tX,p}$.
\end{definition}
One verifies that  a base point $p$ is of $A$--type  if and only if $\cali_{\calL,p}\not \subset \m_{\tX,p}^2$.
A  base point of $A_t$--type  has the following geometric picture.
If $s\in H^0(\tX,\calL)$ is a generic global section then its divisor $D$ in $(\tX,p)$ is  reduced, smooth and
transversal to $E$. Moreover,
if we blow up $\tX$ at $p$ then (via  the notations  of  Lemma \ref{lem:BASE})
$b^*\calL(-E^{new})$ has no fixed components,  and
on  $E^{new}$ it has  no base points in the $t=1$ case. If $t>1$ then
it has exactly one base point, namely at the intersection of $E^{new}$  with
the strict transform of $D$.  This base points is of $A_{t-1}$--type.

In particular, in order to eliminate a base point of type $A_t$  we need exactly $t$ successive blow ups.
At  all these steps Lemma \ref{lem:BASE} {\it (b)}--{\it (c)} applies.

We warn the reader that if a base point can be eliminated by $t$ successive blow ups then it is not
necessarily of $A_t$--type. (Take e.g. the ideal $\cali_{\calL,p}=\m_{\tX,p}^2$, which can be eliminated by
one blow up.)

\bekezdes\label{bek:NATLINEB} {\bf Natural line bundles.} \
The epimorphism  $c_1$ in (\ref{eq:PIC})
admits a unique group homomorphism section $L'\ni l'\mapsto s(l')\in {\rm Pic}(\widetilde{X})$,
 which extends the natural
section $l\mapsto \calO_{\widetilde{X}}(l)$ valid for integral cycles $l\in L$, and
such that $c_1(s(l'))=l'$  \cite{trieste,OkumaRat}.
We call $s(l')$ the  {\it natural line bundles} on $\widetilde{X}$ with Chern class $l'$.
By  the very  definition, $\calL$ is natural if and only if some power $\calL^{ n}$
of it has the form $\calO_{\tX}(l)$ for some $l\in L$.
We will use the uniform notation $\cO_{\tX}(l'):=s(l')$ for any $l'\in L'$.

The following fact will be used  several times:

\begin{lemma}\label{lem:BUnatline}
Consider the natural line bundle $\calO_{\tX}(l')\in{\rm Pic}(\tX)$ for $l'\in L'$.
Let $b:\tX^{new}\to \tX$ be the blow up of a point $p\in E$. Then
$b^*(\calO_{\tX}(l'))\in {\rm Pic}(\tX^{new})$ is natural,
in fact, it is $\calO_{\tX^{new}}(b^*(l'))$.
\end{lemma}
Indeed, it is enough to verify the statement for $l\in L$ in which case it is immediate.

If $Z\in L_{>0}$  with $|Z|=E$, then  we can define a similar section of (\ref{eq:PICZ}) by
$s_Z(l'):=  {\mathcal O}_{\widetilde{X}}(l')|_{Z}$.
These bundles  satisfy $c_1\circ s_Z={\rm id}_{L'}$. We write  ${\mathcal O}_{Z}(l')$ for $s_Z(l')$, and we call them
 {\it natural line bundles } on $Z$.

We also use the notations ${\rm Pic}^{l'}(\widetilde{X}):=c_1^{-1}(l')
\subset {\rm Pic}(\widetilde{X})$ and
${\rm Pic}^{l'}(Z):=c_1^{-1}(l')\subset{\rm Pic}(Z)$
respectively. Multiplication by $\calO_{\widetilde{X}}(-l')$, or by
$\calO_Z(-l')$, provides natural affine--space isomorphisms
${\rm Pic}^{l'}(\widetilde{X})\to {\rm Pic}^0(\widetilde{X})$ and
${\rm Pic}^{l'}(Z)\to {\rm Pic}^0(Z)$. (But, of course, multiplication by
 any other line bundle with the right Chern class might also realize the isomorphisms, the previous ones are
 `canonical'.)

\bekezdes \label{bek:ansemgr} {\bf The analytic semigroups associated with $\tX$.} \
By definition, the analytic semigroup (monoid)  associated with the resolution $\tX\to X$ is
\begin{equation}\label{eq:ansemgr}
\calS'_{an}:= \{l'\in L' \,:\,\calO_{\tX}(-l')\ \mbox{has no  fixed components}\}.
\end{equation}
It is a subsemigroup of $\calS'$. One also sets $\calS_{an}:=\calS_{an}'\cap L$, a subsemigroup
of $\calS$. In fact, $\calS_{an}$
consists of the restrictions   ${\rm div}_E(f)$ of the divisors
${\rm div}(f\circ \phi)$ to $E$, where $f$ runs over $\calO_{X,o}$. Therefore, if $s_1, s_2\in \calS_{an}$, then
${\rm min}\{s_1,s_2\}\in \calS_{an}$ as well (take the generic linear combination of the corresponding functions).
In particular,  for any $l\in L$, there exists a {\it unique} minimal
$s\in \calS_{an}$ with $s\geq l$.

Similarly, for any $h\in H=L'/L$ set $\calS'_{an,h}:\{l'\in \calS'_{an}\,:\, [l']=h\}$.
Then for any  $s'_1, s'_2\in \calS'_{an,h}$ one has
${\rm min}\{s'_1,s'_2\}\in \calS'_{an,h}$, and
for any $l'\in L'$   there exists a unique minimal
$s'\in \calS'_{an,[l']}$ with $s'\geq l'$.

\bekezdes \label{bek:HSF}  {\bf The Hilbert--Samuel function.}
 S. S.-T. Yau's {\it maximal ideal cycle}
$Z_{max}\in L$ can be defined either as the unique minimal element of $\calS_{an}\setminus \{0\}$
(or, as the unique minimal element of $\calS_{an}$ which is  $\geq E$, cf. \ref{bek:ansemgr}), or, as
 the  divisorial part of the pullback of the maximal ideal $\m_{X,o}\subset \calO_{X,o}$, i.e.
 $\phi^*{\m_{X,o}}\cdot \calO_{\widetilde{X}}=\calO_{\widetilde{X}}(-Z_{max})\cdot \cali$,
where $\cali$ is an ideal sheaf with 0--dimensional support \cite{Yau1}.
 In general, $Z_{min}\leq Z_{max}$ (but they can be different).
 By the base points of $\m_{X,o}$ associated with $\phi$
 we understand the base points of $\calO_{\tX}(-Z_{max})$. They  are described by  $\cali$.


The {\it Hilbert--Samuel function } is defined  as 
 $f^{HS}(k):= \dim_\bC(\cO_{X,o}/\m_{X,o}^{k})$ for any $k\geq 1$. 
The {\it Hilbert--Samuel polynomial}
is the unique polynomial $P^{HS}(k)=a_2k^2/2+a_1k+a_0$  such that
$P^{HS}(k)=f^{HS}(k)$ for  $k$ sufficiently large.
The coefficient $a_2$ is the {\it multiplicity of $(X,o)$}, ${\rm mult}(X,o)$. Geometrically, it is the degree of the
{\it generic} map $(X,o)\to (\bC^2,0)$.
By \cite{Wa70} ${\rm mult}(X,o)\geq -Z_{max}^2$, and equality holds
exactly in those cases when  $\m_{X,o}$ has no base points with respect to
$\phi$. Moreover,
if all the base points of $\calO_{\tX}(-Z_{max})$  are of $A$--type  then
\begin{equation}\label{eq:multGEN}
{\rm mult} (X,o)=-Z_{max}^2+\sum_{p}t(p).
\end{equation}

If for a certain resolution  the line bundle
$\calO_{\tX}(-Z_{max})$ has base points, then they  can be eliminated by
a convenient sequence of additional blow ups (infinitely close to the base points).
However, from the topological data,  in general,  it is not possible to
identify those resolutions for which $\calO_{\tX}(-Z_{max})$
has no base points (or, the structure of ideal sheaves $\cali$ of $\cO_{\tX}$ in the presence of base points).

\bekezdes\label{bek:restrnlb} {\bf Restricted natural line bundles.}
Regarding natural line bundles the following warning is appropriate.
Note that if $\tX_1$ is a connected small convenient  neighbourhood
of the union of some of the exceptional divisors (hence $\tX_1$ also stays as the resolution
of the singularity obtained by contraction of that union of exceptional  curves) then one can repeat the definition of
natural line bundles at the level of $\tX_1$ as well (as a splitting of (\ref{eq:PIC}) applied for
$\tX_1$). However,  the restriction to
$\tX_1$ of a natural line bundle of $\tX$ (even of type
$\calO_{\tX}(l)$ with $l$ integral cycle supported on $E$)  usually is not natural on $\tX_1$:
$\calO_{\tX}(l')|_{\tX_1}\not= \calO_{\tX_1}(R(l'))$
 (where $R:H^2(\tX,\Z)\to H^2(\tX_1,\Z)$ is the natural cohomological
 restriction), though their Chern classes coincide.

Therefore, in inductive procedures when such restriction is needed,
 we will deal with the family of {\it restricted natural line bundles}. This means the following.
We fix a resolution space $\tX_{top}$ with dual graph $\Gamma_{top}$. Then for any $\tX$,
a convenient small neighbourhood of the
exceptional curves indexed by the graph $\Gamma$ (a connected subgraph of $\Gamma_{top}$)
 the `restricted natural line bundles' in ${\rm Pic}(\tX)$ are
the restrictions to $\tX$ of the natural line bundles from ${\rm Pic}(\tX_{top})$. In this way,
for any $\tX_1$ ($\tX_1\subset \tX$, defined similarly as $\tX$) the restriction of these line bundles
from $\tX$ to $\tX_1$ are basically the restriction of natural line bundles from $\tX_{top}$, hence any induction based on restriction preserves the family stably. The same is valid when we consider instead of $\tX$ an
effective cycle $Z$ with connected support $|Z|\subset E$.

This basically means that we fix $\tX_{top}$
and we  consider the
tower of singularities (resolutions)  $\{\tX\}_{\tX\subset \tX_{top}}$,
or $\{\calO_Z\}_{|Z|\subset E_{top}}$,  and all the restricted
natural line bundles are restrictions  from the top level $\tX_{top}$.
We use the notations $\calO_{\tX}(l'_{top}):=\calO_{\tX_{top}}(l'_{top})|_{\tX}$ and $\calO_Z(l'_{top}):=\calO_{\tX_{top}}(l'_{top})|_Z$ respectively, where  $l'_{top}\in L'(\tX_{top})$.

If for some reason we need a blow up $b:\tX^{new}\to \tX$  at some $p\in E\subset \tX$, then the pull back bundle
$b^*(\calO_{\tX_{top}}(l'_{top})|_{\tX})$ is again a `restricted natural line bundle', namely
$\calO_{\tX_{top}^{new}}(b_{top}^*(l'_{top}))|_{\tX^{new}}
$, where $b_{top}:\tX_{top}^{new}\to \tX_{top}$
 is the blow up of $\tX_{top}$ at $p$ (cf. Lemma \ref{lem:BUnatline}).

 In particular, we obtain a compatible family of line bundles, well--defined and indexed by the Chern classes,
 which are stable with respect to blow up and restrictions (in the towers as above).


Though the next statement is elementary, it is a key ingredient in several arguments.

 The line bundle $\calO_{\tX_{top}}(l'_{top})\in {\rm Pic}(\tX_{top})$
depends on its Chern class $l'_{top}$ (as combinatorial data) but definitely also
 on the analytic type of $\tX_{top}$. When we restrict it to $\tX$, and we vary the analytic structure of $\tX_{top}$ with
the analytic structure of $\tX$ fixed, the bundle
 $\calO_{\tX_{top}}(l'_{top})|_{\tX}\in \pic(\tX)$ might vary in the fixed
 ${\rm Pic}(\tX)$. The next lemma aims to reduce the dependence of
 $\calO_{\tX_{top}}(l'_{top})|_{\tX}$  on the analytic structure of $\tX_{top}$ to the analytic type of
 the pair $(\tX,\tX\cap E_{top})$.

\begin{lemma}\label{lem:resNat}
The restriction
$\calO_{\tX_{top}}(l'_{top})|_{\tX}\in \pic(\tX)$ depends only on the  Chern class $l'_{top}$,
on the analytic type of $\tX$,  and on the analytic type of the
non--compact divisor $E_{top}\cap \tX$ of $\tX$.
\end{lemma}
\begin{proof}
Since $\pic(\tX)$ has no torsion, it is enough to argue  for $l'_{top}\in L(\tX_{top})$
(identified with an integral cycle
supported on $E_{top}$), in which case the statement follows from the definitions.
\end{proof}

\section{Analytic invariants of generic analytic type}\label{s:AnGen}

\subsection{}
Let us comment first the definition  of `generic' analytic type.
The point is that for  a fixed topological type the moduli space  of
all analytic structures supported by that fixed topological type (of a singularity), is not yet described in
the literature. Similarly,
for a fixed resolution graph $\Gamma$, the moduli space of all analytic structures (or resolution spaces $\tX$)
having dual graph $\Gamma$ is again unknown.
 Hence, we cannot define our generic structure as a generic point of such moduli  spaces.
However,  Laufer in \cite{LaDef1}
 defined {\it local complete deformations}  of resolution of singularities.
 For a given resolution $\tX\to X$ with dual graph $\Gamma$, the base space of this deformation space
 parametrizes all the possible (local) deformations of the analytic structure of $\tX$ (with fixed topological type $\Gamma$).
 This parameter space is  the basic  tool in our `working definition', cf. \cite{NNII} and \ref{ss:genan} below.

\bekezdes\label{ss:genan} {\bf The working definition of
the `generic analytic type'.} Usually when we have a parameter space
for a family of geometric objects, the `generic object' might depend essentially
on the fact that  what kind of
anomalies we wish to avoid. Accordingly,  we determine a discriminant space
of the non--wished objects, and generic means elements from  its complement.
In the present article, following \cite{NNII},
 all the discrete  analytic invariants
we treat are basically guided by the cohomology groups of the restricted
natural line bundles associated with a resolution.
Hence, the discriminant spaces
(sitting in the base space of complete deformation spaces of Laufer  \cite{LaDef1}, parametrizing deformations
of  a pair $\tX\subset \tX_{top}$ with fixed dual graphs,
are defined as the `jump loci' of the first--cohomology groups of the restricted natural line bundles
at all levels of the tower $\{\tX_1\subset \tX\}_{\tX_1}$, cf. \ref{bek:restrnlb}.
(Usually, guided by a specific geometrical problem --- e.g.
 the maximal ideal and properties of $Z_{max}$ ---,
we have to consider only finitely many Chern classes, hence only finitely many such bundles/discriminants too.)
A {\it generic  analytic structure}  avoids all such discriminants.

In particular, the definition of the generic analytic type is linked with some distinguished resolution
pair $\tX\subset \tX_{top}$.
(However, this distinguished pair  can be replaced by a new one, generic as well,  if this new one is obtained
from the distinguished one e.g. by a  blow up at a {\it generic} point of $E\subset \tX$, see \ref{lem:BLOWUP}.
Furthermore, in the situation  $\tX_1\subset \tX\cap \tX_{1,top}$, $\tX_{1,top}\subset \tX_{top}$ (cf.
 \ref{bek:restrnlb}), when  $\tX\subset \tX_{top}$ is generic,
  then $\tX_1\subset \tX_{1, top}$ is automatically generic as well.)

The consideration  of a {\it pair} is motivated by the fact  that the notions associated with pairs
behave properly in inductive steps.
(As was explained in \ref{bek:restrnlb},
 even if we start with $\tX_{top}=\tX$ and natural line bundles of $\tX$, if we need to
 restrict them
to some $\tX_1\subset \tX$, we face the situation of restricted bundles associated with the pair $\tX_1\subset \tX$.)
However, once the theorem is proved by induction based on the relative setup, a posteriori,
 in most concrete applications we choose $\tX_{top}=\tX$.  In this latter case
we speak about the generic analytic structure of $\tX$
 with fixed dual graph $\Gamma$ (and about properties of genuine natural line bundles on $\tX$).
For more see \cite{NNII}.

In a slightly simplified language we can regard  the generic analytic structure in the following way as well.
Fix a graph $\Gamma$. For each $E_v$ ($v\in\calv$) the disc bundle with Euler number $E_v^2$ is taut:
it has no analytic moduli. The generic $\tX$ is obtained by gluing `generically' these bundles according to the edges of $\Gamma$ as an analytic plumbing.

\subsection{Review of some results of \cite{NNII}}\label{ss:ReviewNNII}

 The list of analytic invariants, associated with a generic analytic type
  (with respect to a fixed resolution graph),
 which in \cite{NNII} are described topologically,    include  the following ones:
 $h^1(\calO_Z)$, $h^1(\calO_Z(l'))$ (with certain restriction on the Chern class $l'$),
  --- this last one applied  for $Z\gg 0$ provides  $h^1(\calO_{\tX})$
 and  $h^1(\calO_{\tX}(l')$) too ---,
the multivariable Hilbert function
 $L\ni l\mapsto \mathfrak{h}(l)$,  
the analytic semigroup, 
and the maximal ideal cycle of $\tX$.
See above or \cite{CDGPs,CDGEq,Lipman,Nfive,NPS,NCL,Ok,MR} (or Theorem \ref{th:OLD})
for the definitions and relationships between them.
The topological characterizations use the RR--expression $\chi:L'\to \bQ$.

In the next theorem  the bundles $\calO_{\tX}(-l')$ are the `genuine  natural line bundles'
associated with $\tX$ and $l'\in L'$. (For the general case $\tX\subset \tX_{top}$ see
\ref{bek:TOWER}.)
It says  (like  several other statements
 regarding generic analytic structure and restricted natural line bundles) that these bundles behave cohomologically
as the generic line bundles of ${\rm Pic}^{-l'}(\tX)$ (for more comments  see \cite{NNII}, and also
Theorem \ref{th:B}{\it (II)} here).

\begin{theorem}\cite[Theorem A]{NNII}\label{th:OLD}
 Fix a resolution graph (tree of $\bP^1$'s) and assume that the analytic type of  $\tX$ is generic.
 In parts (a)--(b) we assume that $Z$ is an effective cycle $Z\in L_{>0}$. Then
the following identities hold:\\
(a) For any  $Z\in L_{>0}$ with connected support $|Z|$
\begin{equation*}
h^1(\calO_Z) = 1-\min_{0< l \leq Z,l\in L}\{\chi(l)\}.
\end{equation*}
(b) If $l'=\sum_{v\in \calv}l'_vE_v \in L'$ satisfies
$l'_v >0$ for any $E_v$ in the support of $Z$ then
\begin{equation*}
h^1(Z,\calO_Z(-l'))=\chi(l')-\min _{0\leq l\leq Z, l\in L}\, \{\chi(l'+l)\}.
\end{equation*}
(c) If $p_g(X,o)=h^1(\tX,\calO_{\tX})$ is the geometric genus of $(X,o)$ then
\begin{equation*}
p_g(X,o)= 1-\min_{l\in L_{>0}}\{\chi(l)\} =-\min_{l\in L}\{\chi(l)\}+\begin{cases}
1 & \mbox{if $(X,o)$ is not rational}, \\
0 & \mbox{else}.
\end{cases}
\end{equation*}
(d) More generally, for any $l'\in L'$
\begin{equation*}
h^1(\tX,\calO_{\tX}(-l'))=\chi(l')-\min _{l\in L_{\geq 0}}\, \{\chi(l'+l)\}+
\begin{cases}
1 & \mbox{if \ $l'\in L_{\leq 0}$ and $(X,o)$ is not rational}, \\
0 & \mbox{else}.
\end{cases}
\end{equation*}
(e) For $l\in L$ set $\mathfrak{h}(l)=\dim ( H^0(\tX, \calO_{\tX})/ H^0(\tX, \calO_{\tX}(-l)))$.
Then $\mathfrak{h}(0)=0$ and for $l_0>0$ one has
\begin{equation*}
\mathfrak{h}(l_0)=\min_{l\in L_{\geq 0}} \{\chi(l_0+l)\}-\min_{l\in L_{\geq 0}} \{\chi(l)\}+
\begin{cases}
1 & \mbox{if $(X,o)$ is not rational}, \\
0 & \mbox{else}.
\end{cases}
\end{equation*}
(f) \ 
$\calS'_{an}= \{l'\,:\, \chi(l')<
\chi(l' +l) \ \mbox{for any $l\in L_{>0}$}\}\cup\{0\}$.\\
(g)
 Assume that $\Gamma$ is a non--rational graph  and set
$\calm=\{ Z\in L_{>0}\,:\, \chi(Z)=\min _{l\in L}\chi(l)\}$.
Then 
the unique maximal element of $\calm$ is the maximal ideal cycle of\, $\tX$.

(Note that in the above formulae one also has $\min _{l\in L_{\geq 0}}\{\chi(l)\}=\min _{l\in L}\{\chi(l)\}$.)
\end{theorem}

\begin{remark}\label{rem:GENAN}
By part {\it (g)} of Theorem \ref{th:OLD}  for a generic analytic structure $\tX$  one has $\chi(Z_{max})=\min_{l\in L}\{\chi(l)\}$. Note that $\min_{l\in L}\{\chi(l)\}$ is independent on the choice of the resolution graph, it is a
topological invariant of the singularity (denoted in the sequel by $\min\chi$).

Let us assume that $\calO_{\tX}(-Z_{max})$ of a  generic analytic structure $\tX$ has a base point $p\in E_v$,
where $p$ is a regular point of $E$.
Then, if we blow up $\tX$ at $p$ we get a new resolution, say $\tX^{new}$, with dual graph $\Gamma^{new}$. Write the blow up as $b:\tX^{new}\to \tX$, $b^{-1}(p)=E^{new}$. Then
$(b\circ \phi)^*\m_{X,o}\cdot \calO_{\tX^{new}}=\calO_{\tX^{new}}(-b^*Z_{max}-kE^{new})\cdot \cali^{new}$ for some $k\in\Z_{\geq 1}$. Hence, the maximal ideal cycle of $\tX^{new}$ is $Z_{max}^{new}=b^*Z_{max}+kE^{new}$. However,
$\chi(b^*Z_{max}+kE^{new})=\chi(Z_{max})+k(k+1)/2>\min\chi$. In particular, $\tX^{new}$ and $Z_{max}^{new}$ do
 not satisfy {\it (g)} (and several other properties of Theorem \ref{th:OLD}).
This is compatible with the fact that  $\tX^{new}$ is {\it not generic with
respect to the new graph $\Gamma^{new}$}.
(Recall that the center of the blow up was a special point, a base point associated with $\tX$.)

On the other hand, if we take a generic structure, say $\tX^{new}_{gen}$  supported on $\Gamma^{new}$, then
$E^{new}$ can be contracted in this case too, and one gets a resolution $\tX^{new}_{gen}/E^{new}$. In this case the
point $p$ (the image of $E^{new}$) cannot be a base point (since {\it (g)} is valid for $\tX^{new}_{gen}$ as well),
in fact it is a generic point of $E_v$.
(As $\tX^{new}_{gen}$ is constructed via a  generic analytic plumbing,
the gluing point $E_v\cap E^{new}$ is
 also generic on $E_v$.) For further references we highlight this statement.
\end{remark}
\begin{lemma}\label{lem:BLOWUP}
If the pair   $\tX\subset \tX_{top}$ is  generic (with respect to $\Gamma\subset \Gamma_{top}$),
and $p$ is a generic point of $E$,
then the blow up $\tX^{new}\subset \tX_{top}^{new}$ of $\tX\subset \tX_{top}$
at $p$ produces a generic pair. 
\end{lemma}

\subsection{The new results. The structure of base points. } \label{ss:NEW}
If $\tX$ is generic and $l'\in\calS'_{an}\setminus\{0\}$ then we have
$\min_{l>0}\{\chi(l'+l)\}>\chi(l')$ (cf. Theorem \ref{th:OLD}{\it (f)}).

We say that $l'$ and $E_v$ satisfy the property $(*_v)$ if
\begin{equation*}
(*_v)\hspace{2cm} \min_{l\geq E_v, \ l\in L}\{\chi(l'+l)\}=\chi(l')+1.\hspace{3cm}
\end{equation*}
\begin{theorem}\label{th:NEW1}
Consider a resolution $\tX\to X$  with generic analytic structure  as in \ref{ss:genan} and fix $l'\in \calS' _{an}\setminus \{0\}$ and write $\calL:=\calO_{\tX}(-l')$.
Then the following facts hold.

(1) If $p$ is a base point of $\calL$ then $p$ is a regular point of $E$.

(2) All the base points of $\calL$ are of $A$--type.

(3) If $p\in E_v$ is a base point of $\calL$  then $l'$ and $E_v$ satisfy the property $(*_v)$.

(4) If $(l',E_v)<0$ and $l'$ and $E_v$ satisfy $(*_v)$ then $\calL$
has  exactly $-(l',E_v)$ base points on $E_v$. 

(5) Under the assumptions of (4), any base point on $E_v$ is uniformly of $A_{t(v)}$--type, where
$t(v)=m_v^+-m_v$, $m_v$ is the $E_v$--coefficient of $l'$ and
\begin{equation*}
m_v^+=\max\{\,E_v\mbox{--coefficient of } l'+l \, :\, \mbox{where} \ l\geq E_v, \ l\in L, \ \chi(l'+l)=\chi(l')+1\}.
\end{equation*}
%
%
%
%
\end{theorem}
\begin{corollary}\label{cor:NEW1} Assume that $\tX$ is generic and let $Z_{max}$ be its
maximal ideal cycle. Theorem \ref{th:NEW1} applied for $l'=Z_{max}$ and (\ref{eq:multGEN})
imply:
\begin{equation}\label{eq:multGEN2}
{\rm mult} (X,o)=-Z_{max}^2 -\sum_v t(v)\cdot (Z_{max},E_v),
\end{equation}
where the sum is over all $v\in \calv$ with $(Z_{max},E_v)<0$ and
$\min_{l\geq E_v}\chi(Z_{max}+l)=\chi(Z_{max})+1$.

Since all the involved invariants (in the case $\tX$ generic)  are
computable from the dual graph $\Gamma$ of $\tX$ (cf. Theorem \ref{th:OLD}), (\ref{eq:multGEN2}) is a topological/combinatorial expression for ${\rm mult}(X,o)$.
\end{corollary}

\begin{remark}\label{rem:*} (a)
The long cohomological exact sequence associated with $0\to \calO_{\tX}(-l'-E_v)\to \calO_{\tX}(-l')\to \calO_{E_v}(-l')\to 0$ and Theorem \ref{th:OLD}{\it (d)-(f)}  show that for $\tX$ generic and
 $l'\in\calS'_{an}\setminus \{0\}$
 $$\mbox{if}\
 V_v:=\frac{H^0(\tX, \calO_{\tX}(-l'))}{H^0(\tX, \calO_{\tX}(-l'-E_v))}\ \ \ \mbox{then} \
 \ \dim(V_v)= \min_{l\geq E_v}\{\chi(l'+l)\}-\chi(l').$$
 In general $\dim(V_v)\geq 1$. One the other hand,  $(*_v)$ reads as $\dim(V_v)=1$.

Equivalently, $\dim(V_v)=1$  means that
 $\dim\, {\rm im} \big(H^0(\tX, \calO_{\tX}(-l'))\to H^0(E_v, \calO_{\tX}(-l'))\big)=1$. If this happens, and $(l',E_v)<0$,
 then  (even for not necessarily generic $\tX$)
  the line bundle necessarily has base points at  the intersection points of the divisor
   of the generic section with $E_v$.
   Parts {\it (4)--(5)} of Theorem \ref{th:NEW1}
   say that these base points share  uniformly the same type of ideal.
 The geometric meaning of  part {\it (3)} is that if $\dim(V_v)\geq 2$ then there exist
 two generic sections without common zeroes along $E_v$.

 (b) If we blow up a {\it generic} point of $E$ in the generic $\tX$, then $\tX^{new}$ is also generic (cf. \ref{rem:GENAN}), and furthermore, the base points and their structures at level $\tX$ and
 $\tX^{new}$ can be identified. Hence, for ${\rm mult}(X,o)$ the very same type of formula holds with the very same
 correction term given by  the base points. In particular, for any resolution graph $\Gamma'$
 (say,  obtained from $\Gamma$ by several blow ups), the associated generic analytic resolution $\tX'$ will have the very same type of base points. Hence, the structure of base points is independent of the choice of the {\it generic} resolution. (However, if we blow up a base point, then we might eliminate the base points, but
 on those resolutions the formulae valid for generic resolutions do not work,
 and we lose the topological control  as well.)
\end{remark}

\bekezdes\label{bek:TOWER}
Theorem \ref{th:NEW1} is a consequence of the more general {\bf Technical Theorem} \ref{th:NEW2} below,
which is formulated in such a way that that a certain induction runs properly.
More precisely, it is stated for pairs $\tX\subset \tX_{top} $ with generic analytic structure
and  the bundles  are the `restricted natural line bundles' from the level of $\tX_{top}$.

Before we state the new version we note that Theorem \ref{th:OLD} was also
 proved  in \cite{NNII} for  the more general relative version, that is, the line bundles
$\calO_{\tX}(-l')$ from Theorem \ref{th:OLD} can be replaced by
 `restricted natural line bundles'
associated with some generic pair $\tX\subset \tX_{top}$,
 under some negativity assumption regarding $l'_{top}$. 
In this version part {\it (f)} of Theorem \ref{th:OLD}  reads as follows.

Assume that $\tX\subset\tX_{top}$ is a generic pair, and fix $l'_{top}\in L'(\tX_{top})$.
We will assume that its $E_v$--coordinates satisfies $l'_{top,v}>0$ for all $v\in \calv$.
Let $-l':=R(-l'_{top})=c_1(\calO_{\tX_{top}}(-l'_{top})|_{\tX}) \in L'(\tX)$ be its  cohomological
restriction, and assume that $l'\in\calS'\setminus\{0\}$ (compare also with Theorem
\ref{th:B}).
Then,  the fact that $\calO_{\tX_{top}}(-l'_{top})|_{\tX}\in {\rm Pic}^{-l'}(\tX)$ has no fixed components
can be characterised topological, it depends only on the Chern class $l'$
and it happens exactly when the generic element of ${\rm Pic}^{-l'}(\tX)$ has no fixed component.
The topological characterization is (like  for the genuine natural line bundles):
 $\chi(l')<\chi(l'+l)$ for any $l\in L_{>0}$.
In particular, the fact that $\calO_{\tX_{top}}(-l'_{top})|_{\tX}\in {\rm Pic}(\tX)$ has no fixed components
is independent of the top level $\tX_{top}$, and it depends only on the cohomological restriction $l'$ .

In the next statement
$\Gamma$, $E$, $\calv$, etc. denote the invariants at level $\tX$.

\begin{theorem}\label{th:NEW2}
Consider a generic analytic pair  $\tX\subset \tX_{top}$.
Choose  $l'_{top}\in L'(\Gamma_{top})$  such that its $E_v$--coordinate $l'_{top,v}>0$ for any $v\in \calv$.
Let $l':=R(l'_{top})\in L'(\Gamma)$ be its cohomological restriction, and we assume that
$l'\in \calS' _{an}(\tX)\setminus \{0\}$.  Write $\calL:=\calO_{\tX}(-l'_{top})$ for the restricted natural line bundle
$\calO_{\tX_{top}}(-l'_{top})|_{\tX}$ as above.
Then the following facts hold.

\vspace{2mm}

(1') If $p$ is a base point of $\calL$ then $p$ is a regular point of $E$.

(2')  $\calL$ has a global section whose divisor is smooth and intersects $E$ transversally
(along the  regular part of $E$).

(3') If for a certain $v\in\calv$ one has  $(l', E_v)<0$ and $\min_{l\geq E_v}\{\chi(l'+l)\}-\chi(l')\geq 2$ then
 $\calL$ admits   two generic sections without common zeroes along $E_v$.

(4') If $(l',E_v)<0$ and $l'$ and $E_v$ satisfy $(*_v)$ then $\calL$
has  exactly $-(l',E_v)$ base points on $E_v$.

(5')
  In the situation of (4') assume additionally that $l'_{top,v}\geq 0$ for any
  $v\in\calv_{top}$. Let $s'$ be the unique minimal element of
$\calS_{an,[l']}$ with $s'\geq l'+E_v$. Write $s'$ as $l'+l$.
 Then the generic sections
of $\calL$ and $\calL(-l)$ have no common zeroes along  $E_v$.

Furthermore, in numerical terms,
if  $m_v$ (resp. $m_v^+$) denote the multiplicity of $l'$ (resp. of $s'$)
along $E_v$, then $t(p)=m_v^+-m_v$ for any base point $p\in E_v$.

(For further discussion regarding $s'$ and   $m^+_v$ see Remark \ref{rem:M+}.)
\end{theorem}

\begin{remark}\label{rem:M+}  Fix a resolution $\tX\to X$ with generic analytic structure.

(a) For any $n\in \Z$ and $h\in H$ assume that
$L'_{n,h}:=\{l'\in L'\,:\, [l']=h, \ \chi(l')=n\}$ is non--empty. Let $M$ be a maximal element of it,
and assume that there exists no $l\in L_{>0}$ such that $\chi(M+l)<\chi(M)$.
Then $M\in\calS'_{an}$.
Indeed, if  $M\not\in\calS'_{an}$, then by Theorem \ref{th:OLD}{\it (f)} there exists $l\in L_{>0}$
with $\chi(M+l)=\chi(M)$. This contradicts the maximality of $M$ in $L'_{n,h}$.

(b) Note that the assumptions of Theorem \ref{th:NEW1}{\it (5)} (namely,  $(l',E_v)<0$ and
$l'$ and $E_v$ satisfy $(*_v)$) imply   that  $L_{l',v}:=\{l'+l\,:\,
l\geq E_v, \ l\in L, \ \chi(l'+l)=\chi(l')+1\}$ is non--empty.

We claim that $L_{l',v}$ has a unique maximal element, which is exactly $s'$ from {\it (5')} (namely,
the minimal element of $\calS'_{an,[l']}$  with $s'\geq l'+E_v$). Indeed, let $\sigma'$ be a maximal element of
$L_{l',v}$.
Since $l'\in\calS'_{an}$, $\chi(l'+l)>\chi(l')$ for any $l\in L_{>0}$, hence $\chi(l'+l)\geq \chi(\sigma')$.
By part (a) $\sigma'\in \calS'_{an,[l']}$. By the minimality of $s'$ we have $s'\leq \sigma'$.
Assume that $\sigma'-s'=l>0$. Then
$\chi(l')<\chi(s')< \chi(s'+l)=\chi(\sigma')=\chi(l')+1$, a contradiction. ($l'\in \calS'_{an}$
and Theorem \ref{th:OLD}{\it (f)} imply
the  first inequality, and similarly, $s'\in \calS'_{an}$ the second one.) Hence $\sigma'=s'$. This is true for any choice of $\sigma'$, hence $L_{l',v}$ has a unique maximal element, namely, $s'$.
In particular,
in {\it (5')}  $m_v^+$ equals (compare with $(*_v)$)
\begin{equation}\label{eq:M+}
m_v^+=\max\{\,E_v\mbox{--coefficient of } l'+l \, :\, \mbox{where} \ l\geq E_v, \ l\in L, \ \chi(l'+l)=\chi(l')+1\}.
\end{equation}
(c) Not that in the context of Theorem \ref{th:NEW1}, when $\calv=\calv_{top}$,
the assumptions of Theorem \ref{th:NEW2} regarding $l'_{top}$ and $l'$ are satisfied.
Indeed, if $l'\in \calS'\setminus \{0\}$ then $l'_v>0$ for every $v$.
\end{remark}

\bekezdes The proof of Theorem \ref{th:NEW2} runs over several section. At the end of this section
we prove part {\it (1')} and all the statements for $\tX$ rational (as a starting point  of an induction).

\subsection{The proof of Theorem \ref{th:NEW2}{\it (1')}}\label{ss:proofpart1}
We will use the following fact, cf. \ref{bek:TOWER}:
\begin{equation}\label{eq:semigr}
\calO_{\tX}(-l'_{top}) \ \mbox{has no fixed components} \ \Leftrightarrow \ \chi(l')<\chi(l'+l) \ \mbox{for any $l\in L_{>0}$}.
\end{equation}
Fix a singular point  $p=E_u\cap E_v$ of $E$. Let
$b:\tX^{new}\to \tX$ be the blow up at $p$ and $E^{new}=b^{-1}(p)$.  One sees that $\tX^{new}$ is also generic
with respect to its dual graph $\Gamma^{new}$. (E.g., the starting $\tX$ can be chosen to be
obtained from a generic structure on $\Gamma^{new}$ by blowing down $E^{new}$.)
 This means that the {\it equivalence} (\ref{eq:semigr}) is  valid for both
 $\calO_{\tX}(-l'_{top})$ and $\calO_{\tX^{new}}(- b_{top}^*(l'_{top}))$ (cf. \ref{bek:restrnlb}).

By assumption $l'\in \calS'_{an}$.  Hence,  by the comments from \ref{bek:TOWER},
the left hand side of (\ref{eq:semigr}) holds for $\calO_{\tX}(-l'_{top})$ too. Thus,
by (\ref{eq:semigr}),  both sides are satisfied in the case of  $\calO_{\tX}(-l'_{top})$, $l'=R(l'_{top})$.

Using this  we show that the right  hand side of (\ref{eq:semigr}) is valid for $\calO_{\tX^{new}}(- b_{top}^*(l'_{top}))$ too.

For this we have to verify that
\begin{equation}\label{eq:b}
\chi(b^*(l'))< \chi(b^*(l')+l^{new})\ \mbox{for any $l^{new}\in L(\Gamma^{new})$,  $l^{new}>0$}.
\end{equation}
Write $l^{new}=b^*(l)+kE^{new}$ with some $l\in L$ and $k\in  \Z$. Then $\chi(b^*(l'))=\chi(l')$ and $\chi(b^*(l')+l^{new})=
\chi(l'+l)+k(k+1)/2$. If $l>0$ then $\chi(l'+l)>\chi(l')$. If $l=0$ then $l^{new}=kE^{new}$, hence $k\geq 1$
and $k(k+1)/2>0$. Hence (\ref{eq:b}) holds.

In particular, the left hand side of (\ref{eq:semigr}) should hold
for $\calO_{\tX^{new}}(- b_{top}^*(l'_{top}))$, i.e. this bundle has no fixed components.
 But then $p$ cannot be a base point of $\calO_{\tX}(-l'_{top})$, since in that case
  $E^{new}$ would be  a fixed component by Lemma \ref{lem:BASE}{\it (b)}.

\subsection{The proof of Theorem \ref{th:NEW2} for $\tX$ rational}\label{ss:proofrat}
From (\ref{eq:PIC}) we obtain that any line bundle with Chern class $-l'$ is isomorphic to
$\calO_{\tX}(-l'_{top})$.
 Therefore, any noncompact curve (cut) $C$ in $\tX$, which makes $l'+C$ numerically trivial
(that is, $(C+l',E_v)=0$ for any $v\in\calv$) is the divisor of a possible global section of $\calO_{\tX}(-l'_{top})$. Since the
position of such curves $C$  can be moved generically, one obtains that   $\calO_{\tX}(-l'_{top})$ has no base points at all
(see also  \cite{Artin66}).
Hence, to finish the proof, we need to verify that if ($\dag$) $(l',E_v)<0$ then ($*_v$) cannot happen.
  Indeed,  $\chi(l'+l)-\chi(l')=\chi(l)-(l',l)$. But for $l\geq E_v$ one has
   $\chi(l)\geq 1$  by Artin's criterion of rationality
  \cite{Artin62,Artin66}, and $(l',l)\leq (l',E_v)\leq -1$ since $l'\in\calS'$ and ($\dag$).

\section{Effective Cartier divisors and Abel maps}\label{s:ECD}

Some parts of the proof of Theorem \ref{th:NEW2} are based on the properties of Abel maps associated with
normal surface singularities.  In this section we review some needed material.
We follow \cite{NNI}, see also \cite[\S 3]{Kl} and \cite{Groth62}.
In the sequel  we fix a good resolution $\phi:\tX\to X$ of a normal surface singularity,
whose link is a rational homology sphere. The notations of section \ref{s:prel} will also be adopted.

Regarding notations  the next observation is appropriate. In the previous sections (and in the sequent ones also)
it was natural to use the notation
$\calO(-l)$ for bundles with  $l\in\calS$ (since these are related with the ideal sheaf of section with vanishing
order $\geq l$). Here $c_1(\calO(-l))=-l$.  On the other hand, in this section we discuss the space of Cartier divisors and Picard groups with fixed Chern classes, and here it is not natural to carry  this sign
in all expressions. So, we will use the notation
$\calO_{\tX}(l')$ for bundles with $l'\in -\calS'$. This explains some sign  differences in certain formulae.

\subsection{The Abel map} \label{ss:4.1}
Let us fix an effective integral cycle  $Z\in L$, $Z\geq E$.
Let $\eca(Z)$  be the space of effective Cartier  divisors supported on  $Z$.
Note that they have zero--dimensional supports in $E$.
Taking the class of a Cartier divisor provides  a map
$c:\eca(Z)\to \pic(Z)$, called the {\it Abel map}.
Let  $\eca^{l'}(Z)$ be the set of effective Cartier divisors with
Chern class $l'\in L'$, that is,
$\eca^{l'}(Z):=c^{-1}(\pic^{l'}(Z))$.
We consider the restriction of $c$, $c^{l'}(Z) :\eca^{l'}(Z)
\to \pic^{l'}(Z)$ too, sometimes still denoted by $c$.
The bundle  $\calL\in \pic^{l'}(Z)$ is in the image ${\rm im}(c)$ of the Abel map  if and only if
it has no fixed components, that is,  if and only if
$H^0(Z,\calL)_{\reg}\not=\emptyset$,  cf. (\ref{eq:H_0}).
%

One verifies that $\eca^{l'}(Z)\not=\emptyset$ if and only if $-l'\in \calS'\setminus \{0\}$. Therefore, it is convenient to modify the definition of $\eca$ in the case $l'=0$: we (re)define $\eca^0(Z)=\{\emptyset\}$,
as the one--element set consisting of the `empty divisor'. We also take $c^0(Z)(\emptyset):=\calO_Z$. Then we have
\begin{equation}\label{eq:empty}
\eca^{l'}(Z)\not =\emptyset \ \ \Leftrightarrow \ \ l'\in -\calS'.
\end{equation}
If $l'\in -\calS'$  then
  $\eca^{l'}(Z)$ is a smooth complex irreducible quasi--projective
  variety of dimension $(l',Z)$ (see \cite[Th. 3.1.10]{NNI}). Moreover, cf.  \cite[Lemma 3.1.7]{NNI},
if $\calL\in \im (c^{l'}(Z))$
then  the fiber $c^{-1}(\calL)$
 is a smooth, irreducible quasiprojective variety of  dimension
 \begin{equation}\label{eq:dimfiber}
\dim(c^{-1}(\calL))= h^0(Z,\calL)-h^0(\calO_Z)=
 (l',Z)+h^1(Z,\calL)-h^1(\calO_Z).
 \end{equation}
The Abel map can be defined for any effective integral cycle $Z$ (even without $Z\geq E$). However, in this note
in all our applications all the $E_v$--coefficients of $Z$ will be very large, denoted by $Z\gg 0$.
In this way $Z$ will be a `finite  model' for $\tX$. (Note that `$\eca(\tX)$' is `undefined,  infinite dimensional'.)
Additionally we will also have $h^1(Z,\calL)=h^1(\tX,\calL)$ for $\calL\in {\rm Pic}(\tX)$
by Formal Function Theorem.

\bekezdes \label{bek:I}
Consider again  a Chern class $l'\in-\calS'$ as above.
The $E^*$--support $I(l')\subset \calv$ of $l'$ is defined via the identity  $l'=\sum_{v\in I(l')}a_vE^*_v$ with all
$\{a_v\}_{v\in I}$ nonzero. Its role is the following:

Besides the Abel map $c^{l'}(Z)$ one can consider its `multiples' $\{c^{nl'}(Z)\}_{n\geq 1}$ as well.
 It turns out (cf. \cite[\S 6]{NNI}), that $n\mapsto \dim \im (c^{nl'}(Z))$
is a non-decreasing sequence, and   $\im (c^{nl'}(Z))$ is an affine subspace
for $n\gg 1$, whose dimension $e_Z(l')$ is independent of $n\gg 1$, and essentially it depends only
on $I(l')$. Moreover, by \cite[Theorem 6.1.9]{NNI},
\begin{equation}\label{eq:ezl}
e_Z(l')=h^1(\calO_Z)-h^1(\calO_{Z|_{\calv\setminus I(l')}}),
\end{equation}
where $Z|_{\calv\setminus I(l')}$ is the restriction of the cycle $Z$ to its $\{E_v\}_{v\in \calv\setminus I(l')}$
coordinates.
For $Z\gg 0$  this gives
\begin{equation}\label{eq:ezlb}
e_Z(l')=h^1(\calO_{\tX})-h^1(\calO_{\tX(\calv\setminus I(l'))}),
\end{equation}
where $\tX(\calv\setminus I(l'))$ is a convenient small neighbourhood of $\cup_{v\in \calv\setminus I(l')}E_v$.

Let $\Omega _{\tX}(I)$ be the subspace of $H^0(\tX\setminus E, \Omega^2_{\tX})/ H^0(\tX,\Omega_{\tX}^2)$ generated by differential forms which have no poles along $E_I\setminus \cup_{v\not\in I}E_v$.
Then, cf. \cite[\S8]{NNI},
\begin{equation}\label{eq:ezlc}
h^1(\calO_{\tX(\calv\setminus I)})=\dim \Omega_{\tX}(I).
\end{equation}
\bekezdes \label{bek:DomAbel} {\bf $c^{l'}(Z)$ dominant.}
Next, we characterize those cases, when the Abel map $c^{l'}(Z)$ is dominant
(the closure of its image is  $\pic^{l'}(Z)$). By \cite[Theorem 4.1.1]{NNI} one has

\begin{theorem}\label{th:dominant}
 Fix $l'\in -\calS'$,  $Z\geq E$ as above. Then
 $c^{l'}(Z)$ is dominant if and only if $\chi(-l')<\chi(-l'+l)$ for all $0<l\leq Z$, $l\in L$.
If $Z\gg0$, then this last restriction runs over $0<l$, $l\in L$.
In particular, the fact that $c^{l'}(Z)$ is dominant is independent of the analytic structure supported by $\Gamma$ and it can be characterized topologically.

Moreover,
if $c^{l'}(Z)$ is dominant then $h^1(Z,\calL_{gen})=0$ for  generic $\calL_{gen}\in\pic^{l'}(Z)$.
\end{theorem}
More generally, cf. \cite[Prop. 5.6.1]{NNI},
\begin{theorem}\label{th:dominant2}
If $\calL\in {\rm im}(c^{l'}(Z))\subset \pic^{l'}(Z)$ then
$h^1(Z,\calL)\,\geq\, {\rm codim}({\rm im}(c^{l'}(Z)))$. 
\end{theorem}

\bekezdes \label{bek:GenAbel} {\bf The case of generic analytic structure $\tX$.}
We consider  a generic pair $\tX\subset \tX_{top}$ and the corresponding restricted
natural line bundles $\calO_{\tX}(l'_{top})\in {\rm Pic}(\tX)$, restricted from ${\rm Pic}(\tX_{top})$.
Additionally, we will take an integral cycle $Z\geq E$ (this will `replace' $\tX$ whenever $Z\gg 0$).
The corresponding restricted natural line bundles will be denoted by  $\calO_Z(l'_{top})\in {\rm Pic}(Z)$.

The main feature of the generic analytic structures is that
 a  restricted  natural line bundle  $\calO_Z(l'_{top})$  cohomologically behave like the generic line bundle
$\calL_{gen}\in \pic^{l'}(Z)$. The  precise statement is formulated as follows.
(This is Theorem 5.1.1 from \cite{NNII}; here we use the notation
$\tX\subset \tX_{top}$ for the pair $\tX(|Z|)\subset \tX$ of \cite{NNII}.)   Below, $\calv$, $\calS'$, $E$
are invariants of the dual graph of  $\tX$.

\begin{theorem}\label{th:B}\cite{NNII}
 Take $\tX\subset \tX_{top}$  generic and $Z\geq E$ as above.
Assume that   $l'_{top}=\sum_{v\in \calv_{top}}l'_{top,v}E_v$
satisfies $l'_{top,v} <0$ for any $v\in\calv$ and   $l': = R(l'_{top}) \in -\calS'$.

(I) \  The following facts are equivalent:

(a) $\calO_Z(l'_{top})\in \im(c^{l'}(Z))$, that is,   $H^0(Z,\calO_Z(l'_{top}))_{\reg}\not=\emptyset $;

(b) $c^{l'}(Z)$ is dominant, or equivalently,  $\calL_{gen}\in \im(c^{l'}(Z))$, that is,
 $H^0(Z,\calL_{gen})_{\reg}\not=\emptyset $,
 for a generic line bundle $\calL_{gen}\in \pic^{l'}(Z)$;

 (c) $\calO_Z(l'_{top})\in \im(c^{l'}(Z))$,
 and for any $D\in (c^{l'}(Z))^{-1}(\calO_Z(l'_{top}))$ the tangent map
 $T_Dc^{l'}(Z): T_D\eca^{l'}(Z)\to T_{\calO_Z(l'_{top})}\pic^{l'}(Z)$ is surjective.

 \vspace{1mm}

 \noindent (II)  We have $h^i(Z,\calO_Z(l'_{top}))=h^i(Z,\calL_{gen})$  for $i=0,1$ and a generic line
 bundle $\calL_{gen}\in \pic^{l'}(Z)$.

\end{theorem}

\subsection{The Abel map in the relative setup.}
We consider a resolution $\tX$ with resolution graph $\Gamma$ and an integral cycle $Z\geq E$  as in \ref{ss:4.1}.
Moreover, we take another integral cycle (maybe with smaller support) $Z_1 \leq Z$, and set
 $|Z_1| = \calv_1$ and the full subgraph $\Gamma_1$ associated with $|Z_1|$.

We have the restriction map $r: \pic(Z) \to \pic(Z_1)$ and one has also the (cohomological) restriction operator
  $R_1 : L'(\Gamma) \to L_1':=L'(\Gamma_1)$
(defined as $R_1(E^*_v(\Gamma))=E^*_v(\Gamma_1)$ if $v\in \calv_1$, and
$R_1(E^*_v(\Gamma))=0$ otherwise).
For any $\calL\in \pic(Z)$ they satisfy
$c_1(r(\calL))=R_1(c_1(\calL))$ (where $r$ is the restriction, see the diagram below).
In particular,
we have the following commutative diagram as well:

\begin{equation*}  
\begin{picture}(200,40)(30,0)
\put(50,37){\makebox(0,0)[l]{$
\ \ \eca^{l'}(Z)\ \ \ \ \ \ \ \ \ \stackrel{c^{l'}(Z)}{\longrightarrow} \ \ \ \pic^{l'}(Z)$}}
\put(50,8){\makebox(0,0)[l]{$
\eca^{R_1(l')}(Z_1)\ \ \stackrel{c^{R_1(l')}(Z_1)}{\longrightarrow} \  \pic^{R_1(l')}(Z_1)$}}
\put(182,22){\makebox(0,0){$\downarrow \, $\tiny{$r$}}}
\put(78,22){\makebox(0,0){$\downarrow \, $\tiny{$\fr$}}}
\end{picture}
\end{equation*}

By the `relative case' we mean that instead of the `total' Abel map
$c^{l'}(Z)$ we study its restriction above a fixed fiber of $r$.
That is, we fix some  $\mfl\in \pic^{R_1(l')}(Z_1)$,
we set the  subvariety $\eca^{l', \mfl}:=(r\circ c^{l'}(Z))^{-1}(\mfl)
=(c^{R_1(l')}(Z_1) \circ \fr)^{-1}(\mfl) \subset \eca^{l'}(Z)$, and  we study
the restriction  $\eca^{l', \mfl}\to r^{-1}(\mfl)$ of $c^{l'}(Z)$.
Note that it might happen that  $\eca^{l', \mfl}$ \ is empty.
However, if it is non--empty then by \cite[Corollary 5.1.4]{R}
 it is smooth and irreducible (similarly as any $\eca^{l'}(Z)$).


\section{Proof of Theorem \ref{th:NEW2}{\it (2')}}\label{s:Proof2}

\subsection{}\label{ss:proofpart2}
We will prove {\it (2')} by induction on $h^1(\calO_{\tX})$. If $h^1(\calO_{\tX})=0$ then
{\it (2')} follows from  \ref{ss:proofrat}.
Assume that it is true for any pair $\tX\subset \tX_{top}$  with
 $h^1(\calO_{\tX})<p_g$ (for some integer $p_g>0$) and consider the new situation of
a certain $(\tX\subset \tX_{top},l'_{top})$ with $h^1(\calO_{\tX})=p_g$.
We fix also some $Z\in L$, $Z\gg 0$.

Though $\tX_{top}$ is an important ingredient, in some  discussions below
(in order to simplify the notations) we will neglect it tacitly; however,
in the key situations we will provide the needed information regarding $\tX_{top}$ as well
 (the completions at other parts are  rather immediate).

\bekezdes \label{bek:5.1.1}
By Laufer's duality (see e.g. \cite[7.1]{NNI}), $H^1(\calO_{\tX})^*\simeq H^0(\tX\setminus E, \Omega^2_{\tX})/
H^0(\tX, \Omega^2_{\tX})$, hence there exist $u\in\calv$ and a
 form $\omega \in H^0(\tX\setminus E, \Omega^2_{\tX})$ such that $\omega$ has a non--trivial pole
 along $E_u$. Let $t+1\geq 1$ be the largest such pole for some $u$.
 We claim that there exists $\omega$ and $E_u$ such that $t\geq 1$. Indeed, otherwise
$ H^0(\tX\setminus E, \Omega^2_{\tX})/H^0(\tX, \Omega^2_{\tX})=
H^0(\tX, \Omega^2_{\tX}(E))/H^0(\tX, \Omega^2_{\tX})$. But this last space, by Laufer's  duality
 (see \cite[7.1.3]{NNI}) is $H^1(\calO_E)^*$. Hence $p_g=h^1(\calO_E)=0$, a contradiction.

Hence, we assume that $t\geq 1$ and we blow up
 $E_u$ in a generic point $q_1$ and we get a new exceptional divisor $F_1$, then we blow up $F_1$ in a generic point
$q_2$   and we get $F_2$.
 We repeat this procedure $t$ times. Let $\tX_b$ (resp. $\tX_b^-$)
 denote a small neighbourhood of the union of the strict   transform of $E$
 (still denoted by $E$) with  $\cup_{i=1}^t F_i$ (resp. of $E\cup\cup_{i=1}^{t-1}F_i$).
 The dual graphs are denoted by $\Gamma_b$ and $\Gamma_b^-$.
 Let $b:\tX_b\to \tX$  denote  the modification and
 $R$ the cohomological restriction $L'(\Gamma_b)\to L'(\Gamma_b^-)$.

In parallel, we can consider the same blow ups at the very same points, and we get $\tX_{top,b}$.

 Then one has the following facts (for the notation see the statement of Theorem \ref{th:NEW2}):

\vspace{2mm}

 (i) $\tX_b\subset \tX_{top,b}$ and $\tX_b^-\subset \tX_{top,b}$ are generic pairs
 (with respect to their dual graphs).

(ii) $\calL_b:=b^*\calL$ and $\calL_b^-:=
b^*\calL|_{\tX_b^-}$ are restricted natural line bundles  (from $\tX_{top,b}$).

 (iii) $l'_b:= b^*(l')\in L'(\Gamma_b)$ satisfies $(l'_b, F_i)=0$ ($1\leq i\leq t$),
 hence $\calL_b$ cannot have base points along $F_i$.
Similarly,  $l'^{,-}_b:= R( l'_b)\in L'(\Gamma_b^-)$ satisfies $(l'^{,-}_b, F_i)=0$ ($1\leq i\leq t-1$),
hence $\calL^-_b$ cannot have base points along such $F_i$.

(iv) $l'_b\in \calS'_{an}(\tX_b)\setminus \{0\}$, \  $l'^{,-}_b\in \calS'_{an}(\tX_b^-)\setminus \{0\}$.

(v) $h^1(\tX_b, \calO_{\tX_b})=h^1(\tX,\calO_{\tX})=p_g$ and  $h^1(\tX_b,\calO_{\tX_b})>h^1(\tX_b^-,\calO_{\tX_b^-})$.

(vi) The maximum of  pole orders of  differential forms $\omega\in H^0(\tX_b\setminus E(\Gamma_b), \Omega^2_{\tX_b})$ along $F_t$ is  one.

\vspace{2mm}

 For (i) use  \ref{ss:genan} and Lemma \ref{lem:BLOWUP}.
 For (ii) see (\ref{bek:restrnlb}), for (iii)--(iv) use the projection formula.
 The first part of (v) follows from Leray spectral sequence argument.
For (vi) use the fact  that if a form has pole order $k$ along $F_i$ (with $F_0:=E_u$) then its pull--back
via the  blow up at an arbitrary point of $F_i$ has pole order at most $k-1$, and its pull--back
via the  blow up at the  {\it generic} point (with respect to that form)  has
pole order  $k-1$ along $F_{i+1}$. This also shows that there exists at least  one form
with non--trivial pole along $F_t$. Indeed, if we fix a form with pole order $t$ along $E_u$ and the centers of blow up
$q_i$ are generic with respect to this form, then the pull--back of this form has this property.
 This fact together with
(\ref{eq:ezlb})--(\ref{eq:ezlc}) applied for $E_I=F_t$ shows the second part of (v) as well.

Note that $H^0(\tX,\calL)$ is naturally isomorphic to $H^0(\tX_b, \calL_b)$, hence {\it (2')}
for $(\tX\subset \tX_{top};\calL)$ or for $(\tX_b\subset \tX_{top,b};\calL_b)$ are equivalent.
Hence it is enough  to prove it for the second one.

Furthermore, the inductive step applies for  $(\tX_b^-\subset \tX_{top,b}; \calL_b^-)$, hence {\it (2')} is true for this case.

However, in general, the restriction map $H^0(\tX_b, \calL_b)\to H^0(\tX_b^-, \calL_b^-)$ is not surjective, hence a section $s_b^-\in
H^0(\tX_b^-, \calL_b^-)$, which satisfies {\it (2')} does not necessarily lift to $H^0(\tX_b, \calL_b)$.  But,
if it lifts, then it automatically satisfies {\it (2')} since the lift will have no divisor along $F_t$ by (iii).

In order to establish the existence of such a lift we will perturb the analytic structure of the pair
$\tX_b\subset \tX_{top,b}$ by preserving the type of
$\tX^-_b$. Hence,  $\calL_b$ (being the restriction of a natural bundle of $\tX_{top,b}$)
  will also be perturbed by the corresponding restriction natural line bundle
associated with Chern class $l'_{b,top} = b^*(l'_{top})$. However, the construction will guarantee
that the pair $(\tX_b^-; \calL_b^-)$ will stay stable.
Then we show that for a  generic element of the perturbation the lifting is possible. (On the other hand,  since
the original $(\tX_b\subset \tX_{top,b};\calL_b)$ was generic, it has the very same properties as any
 small perturbation of it,
hence the lifting follows for the original $(\tX_b\subset \tX_{top,b};\calL_b)$ too.)

The analytic structure of $\tX_b\subset \tX_{top,b}$ will be perturbed via the following additional construction.

\bekezdes \label{bek:5.1.2}
First, we fix $n$ generic points $\{p_i\}_{i=1}^n$ on $F_t$  and we blow up $\tX_b$ at these
points. This modification is denoted by $B:\tX_B\to \tX_b$, respectively
$\tX_{top,B}\to \tX_{top,b}$. 

 The strict transforms of
$\{E_v\}_{v\in \calv}$ and $\{F_i\}_{i=1}^{t-1}$ are denoted by the same symbols, while the
strict transform of $F_t$ by $F_{t,B}$. Let $\Gamma_B$ be the dual graph of $\tX_B$,
and let $\tX_B^-$ be a small convenient neighbourhood of $\sum_vE_v\cup \cup_{i=1}^{t-1}F_i \cup F_{t,B}$
in $\tX_B$ with dual graph $\Gamma_B^-$. (Note that $F_{t,B}^2\not=F_t^2$ and  $\Gamma_B^-\not=\Gamma_b$ , though their shapes are the same.) Additionally, set $\calL_B:=B^*\calL_b=B^*b^*\calL$ and $\calL_B^-:=
\calL_B|_{\tX_B^-}$. They have Chern classes $l'_{B}:= B^*b^*l'\in L' (\Gamma_B)$
 and its cohomological restriction $l'^{,-}_B$  into $L'(\Gamma_B^-)$, respectively.

Write also $Z_B:=B^*b^*Z$ and $Z_B^-:= Z_B|_{L(\Gamma_B^-)}$ (projection to the exceptional curves from $\Gamma_B^-$).

Then the analogues of (i)--(vi) from \ref{bek:5.1.1} are the following:

\vspace{2mm}

 (i) $\tX_B\subset \tX_{top,B}$ and $\tX_B^-\subset \tX_{top,B}$ are generic (with respect to their dual graphs).

(ii) $\calL_B$ and $\calL_B^-$ are restricted natural line bundles (from $\tX_{top,B}$).

 (iii) $(B^*b^*(l'), E_i^{new})=0$ for $1\leq i\leq n$,
 where  $\{E_i^{new}\}_i$  are the exceptional curves of $B$,

(iv) $l'_B\in \calS'_{an}(\tX_B)\setminus \{0\}$, \  $l'^{,-}_B\in \calS'_{an}(\tX_B^-)\setminus \{0\}$.

(v) $p_g(\tX_B)=p_g(\tX_b)$.

(vi) $p_g(\tX_B)=p_g(\tX_B^-)$ and  the restriction realizes an isomorphism
$\pic^{-l'_B}(Z_B)\stackrel{\simeq}{\longrightarrow}\pic^{-l'^{,-}_B}(Z_B^-)$.

\vspace{2mm}

Part (vi)
 follows again from statements from \ref{bek:I}, since along $E^{new}_i$ none of the differential forms have got a pole
 (by the same reason as in the proof of (vi) from \ref{bek:5.1.1}).

\bekezdes\label{bek:5.1.3} \ $\tX_b^-$ embeds naturally into $\tX_B^-$ and $\calL_B^-|_{\tX^-_b}=\calL^-_b$.
Hence we have the following commutative diagram:
\begin{equation*}  
\begin{picture}(200,40)(30,0)
\put(50,37){\makebox(0,0)[l]{$
\ \ \eca^{-l'^{,-}_B}(Z^-_B)\ \ \ \ \ \stackrel{c^-_B }{\longrightarrow} \ \ \ \pic^{-l'^{,-}_B}(Z^-_B)$}}
\put(50,8){\makebox(0,0)[l]{$
\ \ \eca^{-l'^{,-}_b}(Z^-_b)\ \ \ \ \ \stackrel{c^-_b}{\longrightarrow} \ \ \ \pic^{-l'^{,-}_b}(Z^-_b)$}}
\put(200,22){\makebox(0,0){$\downarrow \, $\tiny{$r$}}}
\put(90,22){\makebox(0,0){$\downarrow \, $\tiny{$\fr$}}}
\end{picture}
\end{equation*}
Above, $r$ is an affine projection associated with the surjective linear projection
$H^1(\calO_{\tX^-_B})\to H^1(\calO_{\tX_b^-})$. Since $H^1(\calO_{\tX^-_B})\simeq H^1(\calO_{\tX_B})
\simeq H^1(\calO_{\tX})$ (cf. (vi) of
\ref{bek:5.1.2}) the fiber has dimension $p_g-h^1(\calO_{\tX^-_b})>0$.

In $\pic^{-l'^{,-}_b}(Z^-_b)$ we fix $\calL^-_b=b^*\calL|_{\tX^-_b}$. Recall that for the
system  $(\tX^-_b\subset \tX_{top,b};\calL^-_b)$ the statement of the induction holds. Then we study the relative Abel map, the restriction of $c^-_B$
\begin{equation}\label{eq:RelAb}
\eca_{rel}:=  \eca^{-l'^{,-}_B, \calL^-_b}(Z^-_B)\ \stackrel{c_{rel}}{\longrightarrow} \  r^{-1}(\calL^-_b).
\end{equation}
Recall that $\eca_{rel}$ consists of effecive Cartier divisors over $Z^-_B$ with Chern class $-l'^{,-}_B$ whose line bundle restricted to $\tX^-_b$ is exactly $\calL^-_b$.

\bekezdes \label{bek:5.1.4} {\it We claim that  $c_{rel}$ is dominant.} (This in \cite{R} is developed as
`relative dominancy'.)

Indeed, since   $l'^{,-}_B\in \calS'_{an}(\tX_B^-)\setminus \{0\}$ (cf. (iv) of \ref{bek:5.1.2})
there exists $D\in  \eca^{-l'^{,-}_B}(Z^-_B)$, $D\not=\emptyset$,
so that $c^-_B(D) = \calL^-_B\in r^{-1} (\calL^-_b)$.
Since $\tX^-_B$ is generic (cf. (i) of \ref{bek:5.1.2}),
by Theorem \ref{th:B} $T_Dc^-_B$ is surjective, hence $c^-_B$ is a local submersion at $D$. In particular, there exists an analytic open set $V \subset \pic^{-l'^{,-}_B}$ so that $\calL^-_B\in V\subset \im(c^-_B)$. Then
$V\cap r^{-1}(\calL^-_b)$ is an analytic open set in $r^{-1}(\calL^-_b)$ and it is in the image of $c_{rel}$. But $c_{rel}$ is an algebraic map, hence it is necessarily dominant.

\bekezdes \label{bek:5.1.5}
Next, we compare  $\eca^{-l'^{,-}_B}(Z^-_B)$ and $\eca^{-l'^{,-}_b}(Z^-_b)$. Since $(l'^{,-}_B, F_{t,B})=(b^*(l'),F_t)=0$,
in the first space no divisor is allowed, which has support along $F_{t,B}$.

However, in the second space divisors with support $q_t=F_{t-1}\cap F_t$ are allowed
(they might appear if $t=1$.)
Let
$\eca^{-l'^{,-}_b}(Z^-_b)_{q_t} $  be the Zariski open set of $\eca^{-l'^{,-}_b}(Z^-_b)$ consisting of those
divisors whose support does not contain $q_t$. Then $\eca^{-l'^{,-}_b}(Z^-_b)_{q_t} $ and
$\eca^{-l'^{,-}_B}(Z^-_B)$ can be identified.
Hence $D$ can be transported into $\eca^{-l'^{,-}_b}(Z^-_b)_{q_t} $ as well.
Furthermore,
consider ${\rm div}:H^0(Z^-_b,\calL^-_b)_{\reg}\to \eca^{-l'^{,-}_b}(Z^-_b)$, which associates with a section its divisor.
It is surjective onto $(c^-_b)^{-1}(\calL^-_b)$.
Let $H^0(Z^-_b,\calL^-_b)_{\reg,q_t}$ be ${\rm div}^{-1}(\eca^{-l'^{,-}_b}(Z^-_b)_{q_t}) $. It consists
of section, which do not vanish at $q_t$.

We claim that  $H^0(Z^-_b,\calL^-_b)_{\reg,q_t}$ is a {\it non--empty} Zariski open set in $H^0(Z^-_b,\calL^-_b)_{\reg}$.
Indeed, if all the  sections vanish at $q_t$, since $q_t$ was chosen generically (cf. \ref{bek:5.1.1}),
we get that all the sections vanish along
$F_{t-1}$, hence at $q_{t-1}$ too. Since $q_{t-1}$ is also generic, we get vanishing along $F_{t-2}$ and at $q_{t-2}$.
By induction  we get vanishing along $E_u$, a contradiction, since $E_u$ is not a fixed component.

In this way we obtain a {\it surjective} map
\begin{equation}\label{eq:div}
{\rm div}: H^0(Z^-_b,\calL^-_b)_{\reg,q_t}\to \eca^{-l'^{,-}_B,\calL^-_b}(Z^-_B).
\end{equation}
\bekezdes \label{bek:5.1.6}
Now, we apply the induction for the pair $(\tX^-_b\subset \tX_{top,b}; \calL^-_b)$. By this, there
exists a section of $\calL^-_b$ which satisfies {\it (2')}. Let $U$ be the non--empty  Zariski open set
in  $H^0(Z^-_b,\calL^-_b)_{\reg,q_t}$  consisting of sections with property {\it (2')}. Since both ${\rm div}$ and $c_{rel}$ are dominant, $c_{rel}({\rm div}(U))\subset r^{-1}(\calL^-_b)$ contains a non--empty Zariski open set $U_{\pic}$.
Any  bundles from $U_{\pic}$ has the property that its restriction to $\tX^-_b$ is $\calL^-_b$, and it  has
a section which satisfies {\it (2')}.

We will show that (under the initial genericity assumption) the natural line bundle $\calL_B^-$ is in
$U_{\pic}$.

\bekezdes \label{bek:5.1.7} Now we concentrate  on the position of $\calL^-_B$ in $r^{-1}(\calL^-_b)$.

We show that by a conveniently constructed family of perturbations of $(\tX_B\subset \tX_{top,B};\calL_B)$,
the perturbed $\calL_B^-=\calL_B|_{\tX_B^-}$
will move in a small  analytic  open set of $r^{-1}(\calL_b^-)$, hence it necessarily
will intersect $U_{\pic}$.  Since $\tX_B$ itself is generic, we can assume that $\calL_B^-$ itself is an element of $U_{\pic}$.

Let  $T_i $ be a tubular neighbourhood of a  $(-1)$--curve $\tilde{E}_i$ in a smooth surface ($i=1, \ldots, ,n$).
Note that $\tX_B$ is obtained from $\tX_B^-$ by an  analytic plumbing: we glue $\tX_{B}^-$ with
the spaces $T_i$ such that $\tilde{E}_i$ is identified with  $E_i^{new}$, hence $\tilde{E}_i\cap F_{t,B}=p_i$.
In the construction of the flat deformation we glue $T_i$ with $\tX_B^-$ such that $\tilde{E}_i\cap F_{t,B}$
moves  in a small neighbourhood of $p_i\in F_{t,B}$.
 Hence we get a flat family over the parameter germ-space
 $(F_{t,B}^n, (p_1,\ldots, p_n))$ with fibers
 $\tX_{B,\lambda}$ ($\lambda\in (F_{t,B}^n, \{p_i\}_i)$).
  It is convenient to rename each $\tilde{E}_i$ by
$E^{new}_{i,\lambda}$.
  If we blow down the $\{E^{new}_{i,\lambda}\}_i$ and the
 $\{F_j\}_j$ curves then we get a flat deformation of the structure of $\tX$ (in the sense of
 Laufer \cite{LaDef1}, cf. \cite{NNII}).
 For the precise description of these deformations/gluings see \ref{bek:5.1.def}.
 Furthermore, by the very same deformation (regluings)  we obtain a flat family
 $\{\tX_{top,B,\lambda}\}_{\lambda}$ too, hence pairs
 $\tX_{B,\lambda}\subset \tX_{top,B,\lambda}$. $\tX_{top,B,\lambda}$
 is the level  where the natural line bundles are defined, and
 their restrictions are the corresponding `restricted natural line bundles'  in $\pic(\tX_{B,\lambda})$ and
 $\pic(\tX^-_{B,\lambda})$.  Now, for any $\lambda$, one can consider all the data defined in the
 previous subsections for $\tX_B^-$ and $\tX_b^-$.

 It is crucial to notice that $\tX^-_B$ embeds naturally into each $\tX_{B,\lambda}$, hence provides a constant family of subspaces over the parameter space.
 The next key observation  follows from Lemma \ref{lem:resNat}:

 \begin{lemma}\label{lem:CONSTANT}
 $\calL^-_{b,\lambda}:=\calO_{\tX_{top,B,\lambda}}(-l'_{top,B})|_{\tX^-_b}\in \pic(\tX^-_b)$ is independent of $\lambda$, it is exactly
 $\calL^-_b$.
 \end{lemma}
 Since  $\tX_B^-$ and $\calL^-_{b,\lambda}$ are  constant with respect to $\lambda$,
 all the objects considered in the subsections \ref{bek:5.1.3}--\ref{bek:5.1.6} stay stably,
 except
 $\calL^-_{B,\lambda}:=\calO_{\tX_{top,B,\lambda}}(-l'_{top,B})|_{\tX^-_B}\in r^{-1}(\calL^-_b)\subset
 \pic(\tX^-_B)=\pic(Z^-_B)$ (and this is exactly the point, since we wished to `move' the position of $\calL^-_B\in r^{-1}(\calL^-_b)$).

 \bekezdes \label{bek:5.1.8} We claim that for $n\gg0$ and for $\lambda\in (F_{t,B}^n, \{p_i\}_i)$
 the bundle  $\calL^-_{B,\lambda}$ moves in an analytic open set of $ r^{-1}(\calL^-_b)$.
  Here the definition of the natural line bundles will play a role. Indeed, it is enough to verify
  the statement for any  multiple of $\calL^-_{B,\lambda}$. Set $N\gg 1$ so that
  $N\cdot l'_{top,B}=N\cdot B^*b^*(l'_{top})$ can be written as
  $l+m\sum _iE^{new}_{i,\lambda}$, where $l\in L(\tX^-_B)$ and $m\in \Z$.
Note that $m/N$ is the $E_u$--multiplicity  $l'_{top,u}$ of $l'_{top}$, which is positive by the assumption of
Theorem \ref{th:NEW2}. Hence $m>0$ too. This shows that $(\calL^-_{B,\lambda})^{\otimes N}=\calO_{Z_{B}^-}(-l)
\otimes \calO_{Z_{B}^-}(-\sum_i E^{new}_{i,\lambda}\cap \tX_B^-)^{\otimes m}$, where
$\calO_{Z_{B}^-}(-l)$ is again $\lambda$--independent.
Hence,  it is enough to determine the dimension of the space filled by the second contribution when
$\lambda$ moves in its parameter space.

\bekezdes \label{bek:5.1.9} Note that $-\sum_i E^{new}_{i,\lambda}\cap \tX_B^- $ consists of
$n$ generic transversal
divisors in $\eca^{-nF^*_{t,B}}(Z^-_B)$
and we are interested in the dimension of the image of the Abel map
$c^{-nF^*_{t,B}}(Z^-_B): \eca^{-nF^*_{t,B}}(Z^-_B)\to \pic^{-nF^*_{t,B}}(Z^-_B)$.
 This by the results of \ref{bek:I} (see also \cite{NNI}), for $n$ sufficiently large,  is  $h^1(\calO_{\tX_B^-})-h^1(\calO_{\tX^-_b})$,   hence it equals
$\dim (r^{-1}(\calL^-_b))$ too.
Note that the line bundles  $ \calO_{Z_{B}^-}(-\sum_i E^{new}_{i,\lambda}\cap \tX_B^-)$
depend only on the position of the points $\{p_i\}_i$ on $F_{t,B}$.
This follows from the fact that  all the differential forms along $F_{t,B}$ have pole order $\leq 1$
 (and from the explicit  description of the Abel map via integration, cf. \cite[7.2]{NNI}).

In particular, when we move $\lambda$ in its parameter germ, the bundle
$\calL^-_{B,\lambda}$ covers an open subset of $ r^{-1}(\calL^-_b)$.
In particular, for some $\lambda$ it is in $U_{\pic}$. Since $\tX$ was already generic, $\tX_\lambda$ can be replaced by $\tX$, hence we can assume in the sequel that $\calL_B^- \in U_{\pic}$.
This means that $\calL^-_B$ has a section whose divisors are smooth and intersect the exceptional curve transversally.

 \bekezdes \label{bek:5.1.10} Next we lift this property to the level of $\tX_B$. Consider the diagram
\begin{equation*}  
\begin{picture}(200,45)(30,0)
\put(50,37){\makebox(0,0)[l]{$
\ \ \eca^{-l'_B}(Z_B)\ \ \ \ \ \stackrel{c_B }{\longrightarrow} \ \ \ \pic^{-l'_B}(Z_B) \ \ni \, \calL_B$}}
\put(50,8){\makebox(0,0)[l]{$
\ \ \eca^{-l'^{,-}_B}(Z^-_B)\ \ \  \stackrel{c_B^-}{\longrightarrow} \ \ \ \pic^{-l'^{,-}_B}(Z^-_B)\ \ni \, \calL_B^-$}}
\put(200,22){\makebox(0,0){$\downarrow \, $\tiny{$r_B$}}}
\put(90,22){\makebox(0,0){$\downarrow \, $\tiny{$\fr_B$}}}
\end{picture}
\end{equation*}
Then  $r_B$ is an isomorphism by \ref{bek:5.1.2}(vi), and $r(\calL_B)=\calL_B^-$. Moreover, $\fr_B$ is
bijection (identity) too by \ref{bek:5.1.1}(iii) and \ref{bek:5.1.2}(iii).
By the previous paragraph, there exists $D^-\in
\eca^{-l'^{,-}_B}(Z^-_B)$ with $c^-_B(D^-)=\calL^-_B$ and property {\it (2')}, hence
$D:=\fr_B^{-1}(D^-)$ satisfies  $c_B(D)=\calL_B$ and property {\it (2')} too.

On the other hand,  Theorem \ref{th:NEW2}{\it (2')} for $(\tX;\calL)$ and $(\tX_B;\calL_B)$ are equivalent by
blow up.

\bekezdes \label{bek:5.1.def}  Finally, we describe the  deformation
of a fixed resolution, which was used in \ref{bek:5.1.7}.

We choose any good resolution $\phi:(\tX,E)\to (X,o)$, and write
 $\cup_v E_v=E=\phi^{-1}(o)$ as above.
Since each $E_v$ is rational, a small tubular neighborhood of $E_v$ in $\tX$ can be identified with
the disc-bundle associated with the total space $T(e_v)$ of $\calO_{\bP^1}(e_v)$, where $e_v=E_v^2$.
(We will abridge $e:=e_v$.)
Recall that $T(e)$ is obtained by gluing $\C_{u_0}\times \C_{v_0}$ with $\C_{u_1}\times \C_{v_1}$
via identification  $\C_{u_0}^*\times \C_{v_0}\sim\C_{u_1}^*\times \C_{v_1}$,
$u_1=u_0^{-1}$, $v_1=v_0u_0^{-e}$, where $\C_w$ is the affine line with coordinate $w$, and $\C^*_w=\C_w\setminus \{0\}$.

Next, fix any curve $E_w$ of $\phi^{-1}(o)$  and also a {\it generic} point $P_w\in E_w$.
There exists an identification  of the tubular neighbourhood of $E_w$ via $T(e)$ such that
 $u_1=v_1=0$ is $P_w$.
By blowing up $P_w\in \tX$ we get a second resolution $\psi:\tX'\to \tX$;  the strict transforms of $\{E_v\}$'s will be denoted by $E_v'$, and the new
exceptional $(-1)$ curve by $E^{new}$. If we contract  $E_w'\cup E^{new}$ we get a cyclic quotient singularity, which is taut, hence the
 tubular neighbourhood of $E_w'\cup E^{new}$ can be identified with the tubular neighbourhood
 of the union of the zero sections in $T(e-1)\cup T(-1)$. Here we represent $T(e-1)$ as the gluing
 of $\C_{u_0'}\times \C_{v_0'}$ with $\C_{u_1'}\times \C_{v_1'}$ by
$u_1'=u_0'^{-1}$, $v_1'=v_0'u_0'^{-e+1}$. Similarly, $T(-1)$ as
$\C_{\beta}\times \C_{\alpha}$ with $\C_{\delta}\times \C_{\gamma}$ by
$\delta=\beta^{-1}$, $\gamma=\alpha\beta$. Then $T(e-1)$ and $T(-1)$ are glued
along $\C_{u_1'}\times \C_{v_1'}\sim \C_{\beta}\times \C_{\alpha}$
by
$u_1'=\alpha$, $v_1'=\beta$ providing a neighborhood of $E_w'\cup E^{new}$ in $\tX'$.
Then the neighbourhood  $\tX'$ of  $\cup_vE_v'\cup E^{new}$ will be modified by the following 1--parameter family of spaces: the neighbourhood of $\cup_vE_v' $ will stay unmodified, however
$T(-1)$, the neighbourhood of $E^{new}$ will be glued along $\C_{u_1'}\times \C_{v_1'}\sim \C_{\beta}\times \C_{\alpha}$  by
$u_1'+\lambda =\alpha$, $v_1'=\beta$, where $\lambda \in (\C,0)$ is a small holomorphic parameter.

\section{Proof of Theorem \ref{th:NEW2}{\it (3')-(4')}}\label{s:Proof3}

\subsection{} Fix a vertex $v\in\calv$,  which satisfies the assumptions of  {\it (3')}.
Additionally we keep all the constructions and notation of
section \ref{s:Proof2} (proof of part {\it (2')}) as well.

\bekezdes \label{bek:6.1}
Let $o$ be a {\it generic} point of $E_v$ and $\pi_o:\tX_{top,B,o}\to \tX_{top, B}$ be the blow up  at $o$,
$\pi_o : \tX^-_{B,o}\to \tX^-_B$ its restriction over $\tX^-_B$,
and $E_o$  the created exceptional curve.
  Let $\Gamma^-_{B,o}$ be the dual graph of $\tX^-_{B,o}$ and  $l'_o:=\pi^*_o(l'^{,-}_B)
\in L'(\Gamma^-_{B,o}) $.  Since $l'^{,-}_B\in \calS'_{an}(\tX^-_B)\setminus \{0\}$, cf. \ref{bek:5.1.2}(iv),
using
the  pullback of the generic section we get $l'_o\in  \calS'_{an}(\tX^-_{B,o})\setminus \{0\}$ too.
However,  we claim that
 under the assumption of {\it (3')} one also  has
\begin{equation}\label{eq:x}
l'_o+E_o\in \calS'_{an}(\tX^-_{B,o})\setminus \{0\}.
\end{equation}
We give two proofs (a combinatorial one and a geometric one).

\bekezdes \label{bek:6.2} Since $\tX^-_{B,o}\subset \tX_{top, B,o}$ is generic (cf. Lemma \ref{lem:BLOWUP})
and $\pi_o^*(\calL^-_B)$ is a restricted natural line bundle (cf. \ref{bek:restrnlb}),
$l'_o+E_o\in \calS_{an}(\tX^-_{B,o})\setminus \{0\}$
(that is, $\pi_o^*(\calL^-_B)(-E_o)$ is in the image of the  corresponding Abel map) if and only if that Abel map is dominant (cf.
Theorem \ref{th:B}), and this fact happens if and only if (cf. Theorem \ref{th:dominant})
$\chi(l'_o+E_o+\tilde{l})>\chi(l'_o+E_o)$
for any $\tilde{l}\in L(\Gamma^-_{B,o})$, $\tilde{l}>0$. This rewritten is:
$\chi(\tilde{l})\geq (l'_o+E_o,\tilde{l})+1$.
Write $\tilde{l}$ as $\pi_o^*(l)+kE_o$, $l\in L(\Gamma^-_B)$, $k\in\Z$. Then 
the needed inequality at  $\Gamma^-_B$--level  reads as
\begin{equation}\label{eq:chi6b}
\chi(l)+ k(k+1)/2\geq (l'^{,-}_B,l)-k+1.
\end{equation}
If $l>0$ then the assumption of {\it (3')} gives $\chi(l)\geq (l'^{,-}_B,l)+2$, hence
(\ref{eq:chi6b}) follows. If $l=0$ then necessarily $k\geq 1$, hence (\ref{eq:chi6b}) follows again.

\bekezdes \label{bek:6.3} The second  proof is more geometrical (it constructs the needed section).
By assumption, $\dim \im \big(H^0(\tX_B,\calL_B)\to H^0(E_v,\calL_B)\big)\geq 2$. By restriction we get the same
property at the level of
$(\tX^-_B,\calL^-_B)$ too. Hence, there exists two section $s_1, s_2\in H^0(\tX^-_B,\calL^-_B)$ such that
their divisors restricted to $E_v$ (that is, in $\eca^{-l'}(E_v)$)  do not agree. Such elements of $\eca^{-l'}(E_v)$
can be reinterpreted as the set of  roots of a polynomial of degree $-(l',E_v)$. Then one verifies
that for any generic $o\in E_v$ there exists constants $\lambda_1,\lambda_2$ such that
$s_o:=\lambda_1s_1+\lambda_2s_2$ restricted to $E_v$ has a simple root at $o$. Then the pull--back of $s_o$  to $\tX^-_{B,o}$ realizes the divisor  $l'_o+E_o$.

\bekezdes \label{bek:6.4} Next, for $o\in E_v$ generic,  we consider the Abel map with Chern class $-l'_o-E_o$
$$ \eca^{-l'_o-E_o}(\pi_o^*(Z^-_B)-E_o)\to \pic ^{-l'_o-E_o}(\pi_o^*(Z^-_B)-E_o).$$
(The modification of $\pi_o^*(Z^-_B)$ into $\pi_o^*(Z^-_B)-E_o$  will be explained/motivated  in \ref{bek:6.5}.)

Using (\ref{eq:x}) and Theorem \ref{th:B} we get that this Abel map is dominant; even more, for any divisor $D_o$
of $\pi^*_o(\calL^-_B)(-E_o)$, the corresponding tangent map at $D_o$ is surjective.

\bekezdes \label{bek:6.5} Next we make the following identification. For $o$ a generic point of $E_v$,
let $\eca^{-l'^{,-}_B}(Z^-_B)_o$ be the subspace of $\eca^{-l'^{,-}_B}(Z^-_B)$ consisting of those
divisors $D$ whose support contains $o$, and $D$ localized at $o$ is `smooth and transversal to $E_v$'.

We wish to compare this space with the space from  \ref{bek:6.4}. Note that $(-l'_o-E_o,E_o)=1$, hence
any divisor from  $\eca^{-l'_o-E_o}(\pi_o^*(Z^-_B)-E_o)$ intersects $E_o$ with multiplicity one.
Let  $\eca^{-l'_o-E_o}(\pi_o^*(Z^-_B)-E_o)_o $ be the Zariski open set of $\eca^{-l'_o-E_o}(\pi_o^*(Z^-_B)-E_o)$
consisting of those divisors whose support does not contain $E_o\cap E_v$.
We claim that there exists an isomorphism of spaces
\begin{equation}\label{eq:IDENTIFICATION}
\eca^{-l'^{,-}_B}(Z^-_B)_o\to \eca^{-l'_o-E_o}(\pi_o^*(Z^-_B)-E_o)_o.\end{equation}
Indeed, the isomorphism is induced by pull--back of Cartier divisors via $\pi_o^*$.
Let us present  the verification in the relevant  local chart  (for more details
see  the proof of \cite[Theorem 3.1.10]{NNI}).
Fix local coordinates  $(x,y)$ in a neighbourhood of $o$
when  $E_v=\{x=0\}$ and let the multiplicity of $Z$ along $E_v$ be $N$.
Then  the component of a divisor $D$
from  $\eca^{-l'^{,-}_B}(Z^-_B)_o$ with support $o$ (after we eliminate the equivalence via a
multiplication by $\C^*$) can be given by the equation $f=y+P_0(x)+ym_0$ (modulo $x^N$),
where $P_0(x)=\sum_{i\geq 1}a_ix^i$ and $m_0$ belongs to the maximal ideal
$\m_o$ of $\C\{x,y\}$. The equivalence $\sim$ is multiplication  by elements from $1+\m_o$ (modulo $x^N$).
If we multiply $f$ by $(1+m_0)^{-1}$ and we group the $\{x^i\}_{i\geq 1}$ terms we get
$f\sim y+P_1(x)+xym_1$ ($m_1\in\m_o$).
Multiplication by $(1+xm_1)^{-1}$ gives $f\sim y+P_2(x)+x^2ym_1$. By induction
$f\sim y+P_N(x)$ (modulo $x^N$).
Hence a smooth chart  of $\eca^{-l'^{,-}_B}(Z^-_B)_o$)
(up to other product--factors given by other components of $D$ with support
disjoint from $o$, and which are transferred by $\pi^*_o$ trivially)
can be parametrized as $\{a_i\}_{i=1}^{N-1}\mapsto
\{\mbox{the class of } \ y+\sum_{i=1}^{N-1}a_ix^i\}$. This lifts
by $\pi_o= (x=\alpha\beta,y=\beta)$ to the divisor
  $\beta+\sum_{i=0}^{N-2}a_i\alpha^{i}$ (modulo $(\alpha^{N-1})$).

In fact, this product--factor in the chart of  $\eca^{-l'^{,-}_B}(Z^-_B)_o$
extends naturally to  $\{a_i\}_{i=0}^{N-1}\mapsto
\{\mbox{the class of } \ y+\sum_{i=0}^{N-1}a_ix^i\}$, providing a chart for
 $\eca^{-l'^{,-}_B}(Z^-_B)$, and showing
that $\eca^{-l'^{,-}_B}(Z^-_B)_o$ is a smooth,   constructible and irreducible   subspace
 $\eca^{-l'^{,-}_B}(Z^-_B)$ of codimension one. 

 Note that (since $Z\gg 0$) the dimension of $\pic ^{-l'_o-E_o}(\pi_o^*(Z^-_B)-E_o)$ and
$ \pic ^{-l'^{,-}_B}(Z^-_B)$ are the same, they equal $p_g$.

 \bekezdes \label{bek:6.6} Then \ref{bek:6.4} and \ref{bek:6.5} combined give
 that the restriction of $c^-_B$
 $$c^-_{B,o}: \eca^{-l'^{,-}_B}(Z^-_B)_o\to \pic^{-l'^{,-}_B}(Z^-_B)$$
 is dominant and for any divisor  $D_o\in \eca^{-l'^{,-}_B}(Z^-_B)_o$
of $\pi^*_o(\calL^-_B)$, the corresponding tangent map at $D_o$ is surjective.
Then we repeat the constructions and arguments of paragraphs \ref{bek:5.1.3}--\ref{bek:5.1.5} from section
from the proof of part {\it (2')}.
Set
$$ \eca^{-l'^{,-}_B, \calL^-_b}(Z^-_B)_o=  \eca^{-l'^{,-}_B, \calL^-_b}(Z^-_B)\,\cap\,
 \eca^{-l'^{,-}_B}(Z^-_B)_o.$$
 Then similarly as in \ref{bek:5.1.4} one proves that
\begin{equation}\label{eq:RelAb6}
c_{rel,o}:  \eca^{-l'^{,-}_B, \calL^-_b}(Z^-_B)_o\ \longrightarrow \  r^{-1}(\calL^-_b)
\end{equation}
is dominant.

Note also that the space $\eca^{-l'^{,-}_B, \calL^-_b}(Z^-_B)_o$, by a similar
identification as in (\ref{eq:IDENTIFICATION}) (i.e., its relative version)
is isomorphic with a `relative $\eca$'--space, hence it is   is irreducible  for every  generic $o \in E_{v}$.
(This can also be proved by fixing an irreducible Zariski open set in it, cf. \cite{NNI,R} or \ref{bek:6.5}.)
Furthermore, by a similar argument as at the end of \ref{bek:6.5},  $\eca^{-l'^{,-}_B, \calL^-_b}(Z^-_B)_o$ is smooth
as well for any generic $o$.

\bekezdes Consider again the dominant relative Abel map $c_{rel}:\eca^{-l'^{,-}_B, \calL^-_b}(Z^-_B)\to
r^{-1}(\calL^-_b)$, cf. (\ref{eq:RelAb}) and \ref{bek:5.1.4}.
Let us denote   by $\eca^{-l'^{,-}_B, \calL^-_b}(Z^-_B)_{\reg}$
the Zariski open subset of $\eca^{-l'^{,-}_B, \calL^-_b}(Z^-_B)$ consisting of classes of divisors, which have smooth transversal cuts
along the exceptional divisor $E_v$ and also the tangent map of $c_{rel}$ is a submersion.
Moreover, set  $\eca^{-l'^{,-}_B, \calL^-_b}(Z^-_B)_{\reg,o} = \eca^{-l'^{,-}_B}(Z^-_B)_o
\cap \eca^{-l'^{,-}_B, \calL^-_b}(Z^-_B)_{\reg}$.


We denote the restriction of the dominant  map $c_{rel,o}$ from (\ref{eq:RelAb6})
to $\eca^{-l'^{,-}_B, \calL^-_b}(Z^-_B)_{\reg,o}$ with the same symbol. Obviously
$c_{rel,o}: \eca^{-l'^{,-}_B, \calL^-_b}(Z^-_B)_{\reg,o} \to   r^{-1}(\calL^-_b)$ is dominant for generic $o \in E_v$.

Finally, we consider the
 incidence space $$ \frII = \{\,   (p, D) \in E_v \times \eca^{-l'^{,-}_B, \calL^-_b}(Z^-_B)_{\reg} \ :\  p \in |D|\, \}$$
 together with the two canonical projections
$\pi_1: \frII \to E_v$ and $\pi_2: \frII \to \eca^{-l'^{,-}_B, \calL^-_b}(Z^-_B)_{\reg}$, where $ \pi_1((p, D)) = p$ and $\pi_2((p, D)) = D$.
Note, that the map $\pi_2$ is finite and surjective, and for any generic point $o$ of the image of
$\pi_1$ one has $\pi_1^{-1}(o) = \eca^{-l'^{,-}_B, \calL^-_b}(Z^-_B)_{\reg,o}$.
We can replace $\frII$ by a smaller Zariski open set of it, denoted by the same symbol $\frII$,
 such that for any point $o$ of the image of
$\pi_1$ one has $\pi_1^{-1}(o) = \eca^{-l'^{,-}_B, \calL^-_b}(Z^-_B)_{\reg,o}$. Note that
${\rm im}(\pi_1)$ is a Zariski open in $E_v$.

Consider next  the map $c_{rel} \circ \pi_2 : \frII
\to r^{-1}(\calL^-_b)$.

Since for a generic point $o$ the map $c_{rel,o}: \eca^{-l'^{,-}_B, \calL^-_b}(Z^-_B)_{\reg,o}  \to
 r^{-1}(\calL^-_b)$ is dominant, we get from the irreducibility
of $\eca^{-l'^{,-}_B, \calL^-_b}(Z^-_B)_{\reg,o}$ that for a generic point
$D \in \eca^{-l'^{,-}_B, \calL^-_b}(Z^-_B)_{\reg,o}$ the tangent map $T_Dc_{rel}:
T_D(\eca^{-l'^{,-}_B, \calL^-_b}(Z^-_B)_{\reg,o})\to T_{c_{rel}(D)}r^{-1}(\calL^-_b)$ is surjective.

Fix $D$ generic, $|D|\cap E_v=\{p_1, \ldots, p_d\}$, where $d=-(l', E_v)$. Then a neighbourhood
of $p_1$ of $\eca^{-l'^{,-}_B, \calL^-_b}(Z^-_B)_{\reg,p_1}$ embeds naturally into a neighbourhood
of $x:=(p_1,D)$ in $\frII$ as a one--codimensional subspace (such that $p_1$ belongs to the $\pi_1$--image of that neighbourhood).
In particular, $T_1:=T_D(\eca^{-l'^{,-}_B, \calL^-_b}(Z^-_B)_{\reg,p_1})$ embeds  into
$T_x\frII$ as a codimension one sub--vectorspace. Furthermore, the restriction of the tangent map
$T_x(c_{rel}\circ \pi_2)$ to $T_1$ is surjective.
If we denote the tangent space of the $\pi_2$--fiber
$(c_{rel}\circ \pi_2)^{-1}(c_{rel}(D))$  at $x$ by $T_2$, then
the last statement means that $T_1$ and $T_2$ are transversal in $T_x\frII$. Since $T_1$ has codimension one,
we get that $T_2\not\subset T_1$. Hence the $\pi_2$--fiber  $(c_{rel}\circ \pi_2)^{-1}(c_{rel}(D))$
cannot be contained in  $\eca^{-l'^{,-}_B, \calL^-_b}(Z^-_B)_{\reg,p_1}$.

The same is true for all the points $p_1, \ldots, p_d$.
Hence the line bundle $c_{rel}(D) \in r^{-1}(\calL_b^-)$ is base point free.

Since the map $c_{rel}$ is dominant, we obtain that
 {\it the generic bundle of $r^{-1}(\calL_b^-)$ has no base point}.

%
%
%
%

 \bekezdes \label{bek:6.8} Hence, we proved that there exists a Zariski open set
 $U_{\pic, t}\subset U_{\pic}\subset r^{-1}(\calL^-_B)$ such that its elements have no base points.
 Then we continue as in   parts \ref{bek:5.1.7}--\ref{bek:5.1.9} in the proof of part {\it (2')}:
 by a very same type of deformation we can move $\calL_B^-$ into $U_{\pic, t}$.
 Finally, we end the proof with similar argument as \ref{bek:5.1.10}. This ends the proof of
 part {\it (3')}.

For Part {\it (4')} notice, that if ($*_v$) holds, then by Remark \ref{rem:*}(a)
$\dim  \im\big( H^0(\calO_{\tX}(-l'_{top})) \to H^0(\calO_{E_v}(-l'_{top}) )\big) = 1$,
hence  the line bundle $\calO_{\tX}(-l'_{top})$ necessarily
has base points on $E_v$.
Furthermore, by part {\it (2')} we know, that there is a section in $H^0(\calO_{\tX}(-l'_{top})) $ whose divisor consists of
$-(l'_{top}, E_v)$ disjoint smooth transversal cuts.

In particular, the line bundle $\calO_{\tX}(-l'_{top})$ has
 $-(l'_{top}, E_v)$ disjoint base points on $E_v$, all of them regular points of $E$.
This proves part {\it (4')}.

\section{Preliminaries for the proof of Theorem \ref{th:NEW2}{\it (5')}}\label{s:Prel5}

\subsection{Laufer's Duality}\label{ss:LauferD}

Let us fix a good resolution $\tX\to X$  as above. We start with the well--known perfect pairing
(cf. \cite{Laufer72,Laufer77,NNI})
\begin{equation}\label{eq:LD}
\langle\,,\,\rangle :H^1(\tX,\cO_{\tX})\otimes  \big(
H^0(\tX\setminus E,\Omega^2_{\tX})/ H^0(\tX,\Omega^2_{\tX})\big) \longrightarrow\ \C.\end{equation}
Here  $H^0(\tX\setminus E,\Omega^2_{\tX})$ can be replaced by
$H^0(\tX,\Omega^2_{\tX}(Z))$ for $Z\gg 0$ (e.g. for any $Z$ with $Z\geq \lfloor Z_K \rfloor$),
cf. \cite[7.1.3]{NNI},  and one also has  $H^1(Z,\cO_{Z})\simeq H^1(\tX,\cO_{\tX})$. Hence we get a perfect pairing
\begin{equation}\label{eq:LD2}
\langle\,,\,\rangle:  H^1(Z,\cO_{Z})\otimes
\big(H^0(\tX,\Omega^2_{\tX}(Z))/ H^0(\tX,\Omega^2_{\tX})\big)\longrightarrow\ \C.\end{equation}
In particular, a basis $[\omega_1],\ldots, [\omega_{p_g}]$ of $H^0(\Omega^2_{\tX}(Z))/H^0(\Omega^2_{\tX})$
provides $p_g$ affine   coordinates in $H^1(Z,\calO_Z)$. These dualities are given by integrations.
The integration formula  can be lifted from the level of line bundles generated by divisors to the level of the space of
Cartier divisors, cf.  \cite[\S 7]{NNI}. This will be reviewed next.

\bekezdes {\bf The Laufer integration.} \label{bek:transport}
Consider the following situation.
We fix a smooth point $p$ on $E$, a local bidisc $B\ni p$ with local coordinates $(x,y)$ such that $B\cap E=\{x=0\}$.
We assume that a certain form   $\omega\in H^0(\tX,\Omega^2_{\tX}(Z))$ has  local equation
$\omega=\sum_{ i\in\Z,j\geq 0}a_{i,j} x^iy^jdx\wedge dy$ in $B$.

In the same time,
we fix a divisor $\widetilde{D}$ on $\tX$, whose unique component $\widetilde{D}_1$
in $B$ has local equation $y^n$, $n\geq 1$.
 Let $\widetilde{D}_t$ be another divisor, which
is the same as $\widetilde{D}$ in the complement of $B$ and its component $\widetilde{D}_{1,t}$
in $B$ has  local equation
$(y+td(t,x,y))^n$.

Next  we identify
$H^1(\tX, \calO_{\tX})$ with $\pic^0(\tX)$
by the exponential map and we consider the composition
$t\mapsto \widetilde{D}_t-\widetilde{D}
\mapsto \calO_{\tX}(\widetilde{D}_t-\widetilde{D})
\mapsto \exp^{-1} \calO_{\tX}(\widetilde{D}_t-\widetilde{D})
\mapsto \langle\exp^{-1} \cO_{\tX}(\widetilde{D}_t-\widetilde{D}),\omega\rangle$.
The next  formula makes this expression  explicit.
(Here  $B=\{|x|,\, |y|<\epsilon\}$ for a small $\epsilon$, and $|t|\ll \epsilon$.)
\begin{equation}\label{eq:Tomega11}
\langle\langle \widetilde{D}_t,\omega\rangle\rangle:= \langle\exp^{-1} \cO_{\tX}(\widetilde{D}_t-\widetilde{D}),\omega\rangle=
n\int_{|x|=\epsilon \atop |y|=\epsilon} \log\Big(1+t\frac{d(t,x,y)}{y}\Big)\cdot
 \sum_{ i\in\Z,j\geq 0}a_{i,j} x^iy^jdx\wedge dy.\end{equation}
This restricted to any cycle $Z\gg0 $ can be reinterpreted as `$\omega$--coordinate'
of the Abel map
restricted to the path $t\mapsto D_t:=\widetilde{D}_t|_Z$
(and shifted by the image of $D:=\widetilde{D}|_Z$).
If $\omega$ has no pole along the divisor $\{x=0\}$ then
$\langle\langle \widetilde{D}_t,\omega\rangle\rangle=0$
for any path $\widetilde{D}_t$.
Furthermore, the tangent application of the above  composition is
  the `$\omega$--coordinate' of the tangent application of the Abel map restricted to $D_t$.

E.g., if $\widetilde{D}_{1,t}$ is given by $(y+tx^{o-1})^n$  for some $o\geq 1$
and $\omega=\sum_{ i\in\Z,j\geq 0}a_{i,j} x^iy^jdx\wedge dy$,
 then
\begin{equation}\label{eq:Tomega}\begin{split}
\frac{d}{dt}\Big|_{t=0}\,\langle\langle \widetilde{D}_t,\omega\rangle\rangle=
\frac{d}{dt}\Big|_{t=0}\, \Big[n\int_{|x|=\epsilon \atop |y|=\epsilon} \log\Big(1+t\frac{x^{o-1}}{y}\Big)\cdot
\omega \, \Big]
= -4\pi^2\cdot  n\cdot a_{-o,0}\cdot
\end{split}
\end{equation}
If more components of $\widetilde{D}$ are perturbed then
$\frac{d}{dt}\big|_{t=0}\,\langle\langle \widetilde{D}_t,\omega\rangle\rangle$ is the  sum of such contributions.

\begin{definition}\label{not:residue}
Consider the above situation  and assume that
$\widetilde{D}_1$ has  local equation $y$ (i.e. $n=1$ in \ref{bek:transport}).
 Then, by definition,   the {\it Leray residue }
of $\omega$ along $\widetilde{D}_1$ is  the 1--form (with possible poles at $\widetilde{D}_1\cap E$) defined by
$(\omega/dy)|_{y=0}=\sum_{i\in \Z} a_{i,0}x^i dx$. We denote it by ${\rm Res}_{\widetilde{D}_1}(\omega)$.
\end{definition}
Note that
the right hand side of (\ref{eq:Tomega}) tests exactly the non--regular part  of 
 ${\rm Res}_{\widetilde{D}_1}(\omega)$.

\bekezdes\label{ss:TA}  {\bf The tangent of the Abel map.} Fix any integral cycle $Z\in L$, $Z\geq E$.
 Consider again $l'\in \calS'$ and a divisor
  $D \in \eca^{-l'}(Z) $, which is a union of $-(l', E)$ disjoint divisors  $\{D_i\}_i $,
 each of them $\cO_Z$--reduction of divisors $\{\widetilde{D}_i\}_i $ from  $\eca^{-l'}(\tX)$
 intersecting  $E$  transversally.
Set  $\widetilde{D}=\cup_i\widetilde{D}_i $.  

Note that the duality (\ref{eq:LD2}) is true for any such $Z$. It  is enhanced by the following statement.

We introduce a subsheaf $\Omega_{\tX}^2(Z)^{{\rm regRes}_{\widetilde{D}}}$
of  $\Omega_{\tX}^2(Z)$ consisting of those forms $\omega$,
which have the property that for every  $i$ the residue ${\rm Res}_{\widetilde{D}_i}(\omega)$
has no pole along $\widetilde{D}_i$.
For more see \cite[10.1]{NNI}.

\begin{theorem}\label{th:Formsres} \cite[Th. 10.1.1]{NNI}
In  the above situation  one has the following facts.

(a) The sheaves $\Omega_{\tX}^2(Z)^{{\rm regRes}_{\widetilde{D}}}/\Omega_{\tX}^2$ and $\calO_{Z}(K_{\tX}+Z-D)$ are isomorphic.

(b) $H^0(\tX,\Omega_{\tX}^2(Z)^{{\rm regRes}_{\widetilde{D}}})/
H^0(\tX,\Omega_{\tX}^2)\simeq H^0(Z,\calO_Z(K_{\tX}+Z-D))\simeq H^1(Z,\calO_Z(D))^*$.

(c) The image
of the tangent map
at $D$ of
$\eca^{-l'}(Z)\to \pic^{-l'}(Z)$,
after an identification of $T_{c^{-l'}(D)}\pic^{-l'}(Z)$ with $\pic^0(Z)=H^1(Z,\calO_Z)$,
is the intersection of kernels of
linear maps $\langle \cdot, \, [\omega]\rangle:  H^1(Z,\calO_Z)\to\bC$, where $\omega$ runs in
$H^0(\tX,\Omega_{\tX}^2(Z)^{{\rm regRes}_{\widetilde{D}}})$.
\end{theorem}

\subsection{Characterization of base points of line bundles using differential forms}\label{ss:BaseLB}
Next we formulate a characterization of the existence of base points for
restricted natural line bundles of {\it generic singularities}.

Fix a generic pair $\tX\subset \tX_{top}$, $Z\gg 0$, $Z\in L$ as above and a line bundle
$\calL:=\calO_Z(-l'_{top})$, $-l'=c_1(\calL)\in L'$.
We assume that $\calL$ has a section $s\in H^0(Z, \calL)$ without fixed components
such that its divisor $D:={\rm div}(s)$ is the restriction to $Z$  of a reduced  smooth divisor $\widetilde{D}$ of $\tX$,
{\it which meets $E$ transversally}. Write $\cup_i\widetilde{D}_i$ for the irreducible components of $\widetilde{D}$,
and set $p:=E\cap \widetilde{D}_1=E_v\cap \widetilde{D}_1$.

Let $b:\tX^{new}\to \tX$ be the blow up of $\tX$ at $p$  and set $E^{new}=b^{-1}(p)$.
Then $b^*\calL(-E^{new})\in \pic^{-b^*(l')-E^{new}}$ has no fixed components, and the divisor $D_p$ of
$b^*s\in H^0( b^*\calL(-E^{new})|_{b^*Z-E^{new}})$ is the
restriction of a smooth divisor  $\cup_{i>1}\widetilde{D}_i\cup \widetilde{D}_p$
(the strict transform of $\widetilde{D}$)
of $\tX^{new}$, which intersects the exceptional curve transversally.

Note that by  Theorem \ref{th:B}   $c^{-l'}(Z)$
is dominant (equivalently,  $T_Dc^{-l'}(Z) $ is surjective),
by Theorem \ref{th:dominant} $\min_{0<l\leq Z}\chi(l'+l)> \chi(l')$,   and using either Theorem \ref{th:dominant}
or Theorem \ref{th:Formsres}
\begin{equation}\label{eq:h^1=0}
h^1(Z, \calL)=0.
\end{equation}
\begin{proposition}\label{prop:BASE}
The following facts are equivalent:

(a) \ $p=E_v\cap \widetilde{D}_1$ \,is a base point of $\calL$;

(b) \ property ($*_v$) for $l'$ and $E_v$: \ $\min_{l\geq E_v, \ l\in L}\{\chi(l'+l)\}=\chi(l')+1$;

(c) \  $h^1(b^*Z-E^{new},  b^*\calL(-E^{new})|_{b^*Z-E^{new}})=1$;

(d) \ ${\rm codim} ({\rm im}( c^{-b^*l'-E^{new}}))=1$;

(e) \  there exists a form $\omega_p\in H^0(\tX^{new}\setminus E\cup E^{new},
\Omega^2_{\tX^{new}})$, with a nontrivial pole,
 such that its Leray residues along $\cup_{i>1}\widetilde{D}_i$ and $\widetilde{D}_p$
are zero.

If (a)--(e) are satisfied then  the form
$\omega_p$ is unique modulo forms without poles and up to multiplication by  a non--zero constant. Moreover,
 $\ker\, \langle \cdot, \, [\omega_p]\rangle= {\rm im}\, T_{D_p}c^{-b^*l'-E^{new}}$.

(For a reformulation of the properties of $\omega_p$  in terms of a form $\omega\in H^0(\tX\setminus E, \Omega^2_{\tX})$ see
\ref{bek:Omega}.)
\end{proposition}
\begin{proof}
$(a)\Leftrightarrow(b)$ follows from parts {\it (1')--(4')} of Theorem \ref{th:NEW2} already proved.
For $(a)\Leftrightarrow(c)$ use  (\ref{eq:h^1=0}) and (the proof of ) Lemma \ref{lem:BASE}.
Next we prove $(c)\Rightarrow(d)$.

Assume that in {\it (b)} the minimum is realized for $l_0$.
Since $l_0\geq E_v$ we have $b^*(l_0)=E^{new}+\widetilde{l}_0$ for some $\widetilde{l}_0\geq E_v$.
Hence, with the abbreviation $\widetilde{l}':= b^*(l')+E^{new}$, via {\it (b)},  we have
$\chi(\widetilde{l}'+\widetilde{l}_0)=\chi(\widetilde{l}')$. Since $\widetilde{l}'\in \calS'(\tX^{new})$,
Theorem \ref{th:dominant} implies that $c^{-b^*l'-E^{new}}$ is not dominant. This fact together with
{\it (c)}  and Theorem \ref{th:dominant2} imply {\it (d)}.

$(d)\Rightarrow(c)$. Let $h^1$ be the left hand side of {\it (c)}. From Theorem \ref{th:dominant2}
$h^1\geq 1$. But $h^1(b^*\calL)=h^1(\calL)=0$, cf.
\ref{eq:h^1=0}),  hence $h^1\leq 1$ from the exact sequence associated with $b^*\calL(-E^{new})\hookrightarrow b^*\calL$.

{\it (d)}$\Leftrightarrow${\it (e)}.
Note that (similarly as in \ref{bek:6.5}) the germ $\big(\eca^{-b^*l'-E^{new}}(b^*Z-E^{new}), D_p\big)$ can be regarded
as a smooth codimension one subgerm of  $(\eca^{-l'}(Z),D)$, and the composition $C^{-l'}$
$$\big(\eca^{-b^*l'-E^{new}}(b^*Z-E^{new}), D_p\big)\hookrightarrow (\eca^{-l'}(Z),D)\stackrel {c^{-l'}}{\longrightarrow}
(\pic^{-l'}(Z),\calL)$$
identifies with $c^{-b^*l'-E^{new}}: \big(\eca^{-b^*l'-E^{new}}(b^*Z-E^{new}), D_p\big)\to
\big(\pic^{-b^*l'-E^{new}}(b^*Z-E^{new}),\calL\big)$.

Note that $T_{D_p}C^{-l'}$ cannot be surjective, since in that case $C^{-l'}$ would be a local submersion, hence
locally surjective, a fact which would contradict {\it (d)}. This shows that
${\rm im}\, T_{D_p}c^{-b^*l'-E^{new}}$ in $T_{\calL}\pic^{-b^*l'-E^{new}}(b^*Z-E^{new})$
has codimension one. Hence {\it (e)} (and the statement after it) follows from
Theorem \ref{th:Formsres}{\it (c)}.
Conversely,
if such a form exists, then by Theorem \ref{th:Formsres}{\it (b)--(c)} $h^1\geq 1$.
Then continue as in the proof of {\it (d)}$\Rightarrow${\it (c)} to conclude $h^1=1$, i.e. the validity of {\it (c)}.
%
\end{proof}
\bekezdes {\bf $\omega_p$ replaced by $\omega$.}\label{bek:Omega}
Since the restriction of $b:\tX^{new}\setminus E\cup E^{new}\to \tX\setminus E$ is an isomorphism, $\omega_p=b^*\omega$
for some $\omega\in H^0(\tX\setminus E, \Omega^2_{\tX})$. Clearly, $\omega$ has no Leray residue along the components of
$\cup_{i>1}\widetilde{D}_i$.
We claim that $\omega $ has a nontrivial pole along $E_v$. Indeed, otherwise would be nonzero in
$H^0(\tX,\Omega_{\tX}^2(Z)^{{\rm regRes}_{\widetilde{D}}})/
H^0(\tX,\Omega_{\tX}^2)\simeq H^1(Z,\calL)^*$, a fact which contradicts
(\ref{eq:h^1=0}).
Next we analyse its  local properties near $p=\widetilde{D}_1\cap E_v\in\tX$.  Let us fix some local coordinates $(x,y)$ of $(\tX,p)$ such that
$E_v$ is $\{x=0\}$ and $\widetilde{D}_1$ is $\{y=0\}$. Assume that  $\omega $ has the form
$(\varphi_0(x)+y\varphi_1(x)+\cdots)dx\wedge dy/x^o$ for some $o\geq 1$ and $\varphi_0(0)+y\varphi_1(0)+\cdots\not=0 $.
 Then by the blow up
$x=\alpha$, $y=\alpha \beta$   the residue along $\widetilde{D}_p=\{\beta=0\}$ is $\varphi_0(\alpha)d\alpha/\alpha ^{o-1}$.
In particular, $x^{o-1}|\varphi_0(x)$.  Hence $\omega $ near $p$ has the form
$$\omega =\big( x^{o-1} \widetilde{\varphi}_0(x)+y\widetilde{\varphi}_1(x,y)\big) \cdot dx\wedge dy / x^o, \ \ \mbox{where} \ \
x\not| \,  x^{o-1} \widetilde{\varphi}_0(x)+y\widetilde{\varphi}_1(x,y), \ \ o\geq 1.$$
Assume that $\widetilde{\varphi}_0(0)=0$. This would imply that ${\rm Res}_{\widetilde{D}_1}\omega=0$.
Since all the other Leray residues are zero, this fact together with (\ref{eq:h^1=0})  and
Theorem \ref{th:Formsres}{\it (b)} would imply that $\omega$ has no pole along $E$, but this is not the case
since $\omega _p$ has a nontrivial pole. Therefore $\widetilde{\varphi}_0(0)\not =0$.

Consider next an arbitrary deformation $\widetilde{D}_{1,t}=\{y+td(t,x,y)=0\}$ of $\widetilde{D}_1=\{y=0\}$, and also
arbitrary deformations  $\widetilde{D}_{i,t}$ of $\widetilde{D}_i$ for $i>1$. Then in $\frac{d}{dt}\big|_{t=0}
\langle\langle \widetilde{D}_{t},\omega\rangle\rangle$ the terms for $i>1$ have no contributions.
(This follows either from the definition of $\omega$, or from local computation via (\ref{eq:Tomega11}) using the fact that
${\rm Res}_{\widetilde{D}_i}(\omega)=0$.)  Computing the contribution of $\widetilde{D}_{1,t}$ by
 (\ref{eq:Tomega11}) we get
$$\frac{d}{dt}\Big|_{t=0}\,\langle\langle \widetilde{D}_{t},\omega\rangle\rangle=
(-4\pi^2)\cdot d(0,0,0)\cdot \widetilde{\varphi}_0(0) \hspace{1cm}(\widetilde{D}_t:=\cup_i \widetilde{D}_{i,t}).$$
Let $(0,y(t))$ be the intersection point $\widetilde{D}_{1,t}\cap E_v$.  Then, by taking the derivative at $t=0$
of the identity $y(t)+td(t,o,y(t))\equiv 0$,  we obtain
$y'(0)=-d(0,0,0)\in T_pE_v$.  This is the tangent vector  at $p$ of the path   $t\mapsto \widetilde{D}_{1,t}\cap E_v$ in
$(E_v,p)$.
Recall also that $\omega$ is well--defined up to a non--zero constant, let us make this choice in such a way that
$4\pi^2 \widetilde{\varphi}_0(0)=1$. Hence,
\begin{equation}\label{eq:ident}
\frac{d}{dt}\Big|_{t=0}\,\langle\langle \widetilde{D}_{t},\omega\rangle\rangle=
\frac{d}{dt}\Big|_{t=0}y(t).\end{equation}
Therefore,  $T_\calL\omega$ identifies
the tangent vector at $\calL$  of
the path of line bundles $\calL_t:=\calO_{\tX}(\widetilde{D}_t)|_Z$
(shifted  by the constant $\calL^{-1}$)
with the tangent vector of the intersection point
$\widetilde{D}_{1,t}\in E_v$ at $p$
\begin{equation}\label{eq:ident2}
T_\calL\omega\Big( \frac{d}{dt}\Big|_{t=0}\, \calL_t \Big)=
\frac{d}{dt}\Big|_{t=0}y(t).\end{equation}

\bekezdes\label{ss:pole} {\bf The pole of $\omega$.}
Assume that the conditions {\it (a)--(e)} of Proposition \ref{prop:BASE} are satisfied (for $Z\gg 0$).
We claim that there exists a {\it unique minimal} $m\geq E_v$ for which {\it (b) } holds. [Indeed, if
$A$ and $B$ realize  equality in {\it (b)},  and $m:=\min\{A,B\}$, $M:=\max\{A,B\}$, then
$2(\chi(l')+1)=\chi(l'+A)+\chi(l'+B)\geq \chi(l'+m)+\chi(l'+M)\geq 2(\chi(l')+1)$.
Hence  equality holds everywhere and  $\min\{A,B\}$ realizes {\it (b)} with equality too.]
Then by Theorem \ref{th:OLD}{\it (d)}, $m$ is the smallest cycle $l\geq  E_v$ such that $h^1(\tX, \calL(-l))=0$.
Using the notation of Proposition \ref{prop:BASE} consider the exact sequence
$$0\to b^*\calL(-b^*m)\to  b^*\calL(-E^{new})\to b^*\calL(-E^{new})|_{b^*m-E^{new}}\to 0.$$
Since $h^1(b^*\calL(-b^*m))=h^1(\calL(-m))=0$,  and $h^1(b^*\calL(-E^{new}))=1$ (cf. \ref{lem:BASE} and
(\ref{eq:h^1=0})), we get that  $h^1(b^*m-E^{new},  b^*\calL(-E^{new})|_{b^*m-E^{new}})=1$.
This means that parts {\it (c)--(e)} of Proposition \ref{prop:BASE} remain valid for $Z=m$ instead of $Z\gg 0$.
In particular, the  form $\omega_p$ of {\it (e)} survives even if we reduce the allowed poles,
 that is, ${\rm pole}(\omega _p)\leq b^*m-E^{new}$. Thus
\begin{equation}\label{eq:pole}
{\rm pole}(\omega )\leq  m.
\end{equation}
(We expect that in (\ref{eq:pole}) equality holds, however, the inequality suffice for our purposes  here.)
\subsection{The `move' of the base point} \label{ss:move}
The final goal of this and next sections is to prove part {\it (5')} of
Theorem \ref{th:NEW2}.  This will follow from the more general Proposition  \ref{prop:GENERAL}, which says that
if the two line bundles $\calL_i=\calO_{\tX}(-l'_{i,top})$ ($i=1,2$) have base points then they cannot have a common base point
for generic analytic structure. In the proof we will follow the following strategy.
We will assume that there  is a common base point for a generic analytic structure, say $p$,
 and then we will perturb the analytic structure.
Then $p$ must survive as a common base point. On the other hand, we will measure the `move' of the
base point $p$ provided by the perturbation, this for $\calL_i$ will be described in terms of $l'_{i,top}$.
Since $p$ moves for both $\calL_i$ in the same way, this will impose strong restrictions regarding the two Chern classes
$l'_{i,top}$, and this will lead us to a contradiction.

In this subsection we will discuss the perturbation of the analytic structure and we will compute
the tangent vector  of the move of the base point along $E_v$.

Fix a generic pair $\tX\subset \tX_{top}$, a line bundle $\calL=\calO_{\tX}(-l'_{top})$, $l'_{top}\in L'(\tX_{top})$, and
its restriction $l'\in L'$. Assume that $\calL$ satisfies all the divisorial properties from
\ref{ss:BaseLB}, and that $p=E_v\cap \widetilde{D}_1$ is a base point (hence all the properties of
Proposition  \ref{prop:BASE} are valid). Let $\omega$ be as in \ref{ss:BaseLB}.

Assume that $\omega$ {\it has a pole of order  one} along a certain exceptional divisor $E_u$, $u\in \calv$.
Let $q$ be a generic  point of $E_u\setminus \cup_{w\not=u }E_w$.
We perform the perturbation of the analytic structure of the pair $\tX\subset \tX_{top}$ as in  \ref{bek:5.1.7} (based on
\ref{bek:5.1.def}).  In short, we blow up $(\tX_{top}, \tX)$ at $q$, we get $(\tX_{top}^{new}, \tX^{new})$,
then we reglue the tubular neighbourhood of $E^{new}$  with the neighbourhood of (the strict transform of) $E\subset
\tX^{new}$ as in \ref{bek:5.1.def} using a local parameter $\lambda\in (\C,0)$ (see also below).
The deformed space will be denoted by $(\tX_{top,\lambda}^{new}, \tX^{new}_\lambda)$. Note that the tubular neighborhood of
$E$ in $\tX^{new}_{\lambda}$ is independent of $\lambda$; it will be denoted by $\tX_b$.

If we blow down $E^{new}_{\lambda}$ in $(\tX_{top,\lambda}^{new}, \tX^{new}_\lambda)$ we get the generic pair
 $(\tX_{top,\lambda}, \tX_\lambda)$, a topological trivial deformation of $(\tX_{top},\tX)$. For the previous $l'_{top}$  we consider the family of line bundles $\calL_{\lambda}:=\calO_{\tX_{top,\lambda}}(-l'_{top})|_{\tX_{\lambda}}$.
Using the discussion from \ref{ss:BaseLB} applied for $\calL_{\lambda}$ we get that $\calL_{\lambda}$ has similar
divisorial  and numerical properties as $\calL$. Let $(0, y(\lambda))=p_\lambda\in E_v$ be the base point of $\calL_\lambda$
(with $p_0=p$). Our goal is the computation of $\frac{d}{d \lambda}\big|_{\lambda=0}y(\lambda)$.

Let $(x,y)$ be some local coordinates of some point on $E_u$  such that $E_u=\{x=0\}$. Let the local equation of $\omega$  be
$(\psi(y)+x\tau(x,y))dx\wedge dy/x$, with $\psi(y) \not\equiv 0$. Consider a generic point $q$ in this interval of $E_u$
and  let the blow up $b$ at $q $ be
$\{x=\alpha\beta, \ y-q=\beta\}$.
 We will use the same local coordinated for the regluing as well:
$E^{new}_\lambda$ will have local equation $\beta+\lambda=0$.
\begin{lemma}\label{lem:derivative} With the above notations:
$$\frac{d}{d \lambda}\Big|_{\lambda=0}y(\lambda)= 4\pi^2\cdot \psi(q)\cdot \{\mbox{$E_u$--multiplicity of $l'_{top}$}\}.$$
\end{lemma}
\begin{proof}
Let $N\in \Z_{>0}$ such that $Nl'_{top}\in L(\tX_{top})$ is an integral cycle supported in the exceptional curve of
$\tX_{top}$. The point is that the line bundle $\calL_\lambda ^{\otimes N}=\calO_{\tX}(N\cdot
\widetilde{D}_\lambda)$ can also be represented, via its very definition of a restricted natural line bundles, as
 $\calO_{\tX_{top,\lambda}}(-N\cdot l'_{top})\big|_{\tX_\lambda}$ (with its `genuine' definition of $\calO(-l)$ with integral $l$). Therefore, after the blow up, we have
$b^*\calL_\lambda^{\otimes N}= \calO_{\tX^{new}_{top,\lambda}}(-N\cdot
b^*l'_{top})\big|_{\tX^{new}_\lambda}$, cf. Lemma \ref{lem:BUnatline}.
Furthermore, from (\ref{eq:ident2}), one has
\begin{equation}\label{eq:ident3}
\frac{d}{d\lambda}\Big|_{\lambda=0}y(\lambda)=
T_\calL\ \omega \Big( \frac{d}{d\lambda }\Big|_{\lambda=0}\, \calL_\lambda \Big)=
\frac{1}{N}\cdot T_{b^*\calL}\ \omega\Big( \frac{d}{d\lambda }\Big|_{\lambda=0}\,
b^*\calL_\lambda^{\otimes N} \Big).\end{equation}

Next, we wish to compute the right hand side via restriction on $\tX_b$ (the neighbourhood of $E$).
 We start  with two observations.  First,
though the kernel of  $H^1(\calO_{\tX^{new}_{\lambda}})\to H^1(\calO_{\tX_b})$ might be nontrivial,
the class of $\omega$ is not in this kernel. Indeed,
since $\omega$ has pole order one along $E_u$, and $q$ was a generic point of $E_u$, $b^*\omega $
has no pole along $E^{new}$, hence the statement follows from (\ref{eq:ezlc}).
Secondly, $(b^*l'_{top}, E^{new})=0$, hence the divisors of any section of the line bundles are supported in
$\tX_b$, and the integral pairings can be tested in $\tX_b$. In particular, the right hand side of
(\ref{eq:ident3}) equals
\begin{equation}\label{eq:RED}\frac{1}{N}\cdot
T_{b^*\calL|\tX_b}\, \omega \Big(\frac{d}{d\lambda}\Big|_{\lambda=0}\,
\calO_{\tX^{new}_{top,\lambda}}(-N\cdot b^*l'_{top})\big|_{\tX_b}\Big).\end{equation}
In this situation we can apply the construction from \ref{bek:transport}. Indeed, in the local chart $(\alpha,\beta)$  of
$E_u\cap E^{new}\in \tX_b$
the divisor of $\calO_{\tX^{new}_{top,\lambda}}(-N\cdot b^*l'_{top})\big|_{\tX_b}\oplus
\big(\calO_{\tX^{new}_{top}}(-N\cdot b^*l'_{top})\big|_{\tX_b}\big)^{-1}$ is
$-Nl'_{top, u}(E^{new}_\lambda-E^{new})\big|_{\tX_b}$. It has local equation
 $\big( \frac{\beta+\lambda}{\beta}\big)^{-Nl'_{top,u}}$.
Recall that the $E_u$--coefficient $l'_{top,u}$ of $l'_{top}$ is non--zero by the assumptions of Theorem \ref{th:NEW2}.
Therefore, by (\ref{eq:Tomega11}), the expression from (\ref{eq:RED}) before taking
$\frac{\partial}{\partial \lambda}|_{\lambda=0}$ is
$$ -l'_{top,u}\cdot \
\int_{|\alpha|=\epsilon, \ |\beta|=\epsilon} \log\Big(1+\frac{\lambda}{\beta}\Big)\cdot
\big(\psi(\beta+q)+\alpha \beta\tau\big)\cdot \frac{d\alpha \wedge d \beta}{\alpha}.$$
Its derivative at $\lambda=0$  is exactly the right hand side of the identity from Lemma \ref{lem:derivative}.
\end{proof}

\bekezdes {\bf `Rational line bundles'.}
In some arguments  it is convenient to use the
formalism of `rational line bundles', well defined whenever
the Picard group of $\tX$ has no torsion. We follow  \cite{CoNa}.

 \begin{definition}\label{def:RLB} Consider a
normal surface singularity with rational homology sphere link. Fix one of its
good resolutions $\tX$, an effective integral  cycle $Z \in L_{>0}$  and
 $l'' \in L'_{|Z|} \otimes \bQ$.

A rational line bundle on $Z$ 
is an equivalence class of  pairs
 $(N, \calL)\in \Z_{>0}\times\pic^{N \cdot l''}(Z)$   such that $N \cdot l'' \in L'_{|Z|}$.
Two pairs $(N_1, \calL_1)$ and $(N_2, \calL_2)$ are equivalent if $ \calL_1^{N_2} \cong  \calL_2^{N_1}$.
We call $l''$ the Chern class of the rational line bundle $(N,\calL)$, and we denote the set of rational line bundles with Chern class $l''$ by $\pic^{l''}(Z)$.
If $\calL \in \pic^{l''}(Z)$, we write  $c^1(\calL) = l''$.
\end{definition}

Since the Picard group is  torsion free and each ${\rm Pic}^{l'}(Z)$ ($l'\in L'(|Z|)$)
is  isomorphic to $H^1(\calO_Z)$ as affine spaces, we obtain an isomorphism of affine spaces
$\pic^{l''}(Z) \cong H^1(\calO_Z)$
 for any $l'' \in L'_{|Z|} \otimes \bQ$ as well. If $l'\in L'(|Z|)$ then the rational line bundles and (usual)
 line bundles with Chern class
 $l'$ can naturally be identified. Sometimes we abridge $(N,\calL)$ by $\mathfrak{L}$.

If we have classes of two rational line bundles $(N_1,\calL_1) \in \pic^{l''_1}(Z)$ and $(N_2,\calL_2) \in \pic^{l''_2}(Z)$, then we define $(N_1,\calL_1)  \otimes (N_2,\calL_2 ) \in \pic^{l''_1 + l''_2}(Z)$ as the rational line bundle represented by
$(N_1 \cdot N_2, \calL_{ 2}^{N_1} \otimes   \calL_{ 1}^{N_2} )$.
Similarly, for any rational line bundle $(N,\calL)$ with Chern classes $l''$
we can define $(N,\calL)^{-1}$ by $(N, \calL^{-1})$ and $(N,\calL)^{n/m}$ by $(Nm,\calL^n)$ for any
positive  rational number $r=n/m \in \bQ_{>0}$.
They have Chern classes $-l''$ and $ r \cdot l''$ respectively.

If $D \in \eca^{l'}(Z)$ is a divisor  and $r \in \bQ_{>0}$ is a positive
rational number, then the pair $(N, \calO_Z(Nr \cdot D))$ defines a rational line bundle whenever
 $Nr \in \bZ_{>0}$.  We denote it  by $\calO_Z(r \cdot D)$.

In \cite[Lemma 5.0.4]{CoNa} the following fact is proved.

\begin{lemma}\label{rat}
Let $\tX$ be as in Definition \ref{def:RLB}.
Fix an effective integer cycle $Z \geq E$ and  a
vertex $w \in \calv$ such that the $E_w$--coefficient of $Z$ is one.
Fix also  a rational number $a_w \in \bQ_{>0}$.

Consider  a rational line bundle $\mathfrak{L}$ with Chern class $l'' \in L' \otimes \bQ$, such that $l':=l'' + a_w E_w^* \in L'$.
Moreover, we assume that
for a generic point $p_w \in E_w$  and the associated rational divisor  $D =  a_w p_w$
we have  $H^0(Z, \mathfrak{L} \otimes \calO_Z(-D) )_{{\rm reg}} \neq \emptyset$.

Next, we fix an integer $0 \leq k \leq \lfloor a_{w}\rfloor$ and we write
 $d(k) = 0$ if $k=a_{w}$ and $d(k) = 1$ otherwise. Furthermore,
we consider  $k + d(k)$ generic points  $\{q_{j}\}_{1\leq j\leq k+d(k)}$
on the exceptional divisor $E_v$, and $k + d(k)$ positive rational numbers
$r_{j} = 1$ if $1 \leq j \leq k$ and  $r_{k + 1} = a_{w} - k$ whenever  $d(k) = 1$.

Consider the rational divisor  $D' = \sum_ { 1 \leq j \leq k + d(k)} r_{ j} \cdot q_{ j}$
 of $Z$ (supported on $E_w$). Then (i)
 $h^0(Z, \mathfrak{L} \otimes \calO_Z( - D)) = h^0(Z, \mathfrak{L} \otimes \calO_Z( - D'))$ and
(ii) $ H^0(Z, \mathfrak{L} \otimes \calO_Z( - D'))_{{\rm reg}} \neq \emptyset$.
\end{lemma}
\begin{remark}  (a)
The first statement {\it (i)} was proved in \cite{CoNa} only in the case $ k = \lfloor a_w\rfloor $.
The present version, valid for arbitrary $k$,  follows by application of \cite{CoNa} twice. Indeed, set
some other generic points $\bar{q}_j$ for $k+1\leq j\leq \bar{J}:=\lfloor a_w\rfloor +d( \lfloor a_w\rfloor)$, set
$\bar{r}_j =1$ for  $k+1\leq j\leq \lfloor a_w\rfloor $ and $\bar{r}_{\bar{J}}=a_w- \lfloor a_w\rfloor$,
$\bar{D}'=\sum _{j\geq k+1}\bar{r}_j\bar{q}_j$, $\bar{\mathfrak{L}}=\mathfrak{L}\otimes \calO_Z(-\sum_{j\leq k}q_j)$.
Then
$h^0(\mathfrak{L}(-D'))=h^0(\bar{\mathfrak{L}}(-r_{k+d(k)}q_{k+d(k)}))
\stackrel{*}{=} h^0(\bar{\mathfrak{L}}(-\bar{D}'))=h^0(\mathfrak{L}(-\sum_{j\leq k} q_j-\bar{D}'))
\stackrel{*}{=}h^0(\mathfrak{L}(-D))$, where $\stackrel{*}{=}$ is the statement  from \cite{CoNa}.

(b) Part {\it (ii)}  is not written explicitly in \cite{CoNa}, however it is a consequence of {\it (i)}. Indeed,
let $A\geq 0$ be the cycle of fixed components of $ \mathfrak{L} \otimes \calO_Z( - D')$, and assume that
$A>0$. Then $h^0(Z, \mathfrak{L}\otimes \calO_Z (-D'))=
h^0(Z-A, \mathfrak{L}\otimes \calO_{Z-A} (-D'-A))  \leq 
h^0(Z-A, \mathfrak{L}\otimes \calO_{Z-A} (-D-A)) < 
h^0(Z, \mathfrak{L}\otimes \calO_Z (-D))\stackrel{(i)}{=}h^0(Z, \mathfrak{L}\otimes \calO_Z (-D'))$,
which is a contradiction. The first inequality  follows from semicontinuity, the second one  from the assumption
 $H^0(Z, \mathfrak{L} \otimes \calO_Z(-D) )_{{\rm reg}} \neq \emptyset$.
\end{remark}

\section{Proof of Theorem \ref{th:NEW2}{\it (5')}}\label{s:Proof5}

\subsection{} In this subsection besides the `standard' restriction $l'_{top,v}>0$ for any $v\in\calv$,
we will also assume that $l'_{top,v}\geq 0$ for any $v\in\calv_{top}$, cf. part {\it (5')} of Theorem \ref{th:NEW2}.

\bekezdes \label{bek:5..0} {\bf Reduction of $\tX_{top}$.}
Recall that  our main goal is to study certain line bundles $\calL\in \pic(\tX)$:  we impose that
they are restrictions of natural bundles from a `top  level' $\tX_{top}$ (but the subject of the study is the restriction
and not the top level bundle). For any pair $\tX\subset \tX_{top}$ let $\tX_{top}'$ be the tubular neoighbourhood
of the exceptional divisors of $\tX_{top}$ which intersect $E$. We claim that any restricted natural line bundle
restricted from $\tX_{top}$ can be realized as a restricted natural line bundle
restricted from the smaller  $\tX_{top}'$. Indeed, it is enough to verify the claim for integral cycle supported on the
exceptional curves, and
$\calO_{\tX_{top}}(-\sum _{v\in\calv_{top}}n_vE_v)|_{\tX}=\calO_{\tX_{top}'}(-\sum _{v\in\calv_{top}'}n_vE_v)|_{\tX}$.
Moreover, if $\tX\subset \tX_{top}$ is a generic pair then the same is true for $\tX\subset \tX_{top}'$ (cf. \ref{ss:genan}),
and also the positivity of the $l'_{top}$--coefficients is preserved by this replacement. Hence, without loss of
generality,  we can replace $\tX_{top}$ by $\tX_{top}'$. Furthermore, in this new situation, if
$l'_{top,w}=0$ for some $w\in \calv_{top}\setminus \calv$, then we can restrict
$\tX_{top}$ more by taking the tubular neighbourhood of $\cup_{v\in\calv_{top}\setminus w}E_v$.
Hence after this reduction all the coefficients of $l'_{top}$ are strict positive.

\bekezdes
 Theorem \ref{th:NEW2}{\it (5')} follows from the next more general statement
regarding base points of restricted natural line bundles (in its formulation we already applied  the discussions from
\ref{bek:5..0}).

\begin{proposition}\label{prop:GENERAL} Fix a generic pair $\tX\subset \tX_{top}$, such that
$E\cap E_u\not=\emptyset$ for any $u\in \calv_{top}$. Moreover, we fix
$l'_{1, top} = \sum_{u \in \calv_{top}} a_u E_u$ and $l'_{2, top} = \sum_{u \in \calv_{top}} b_u E_u$
in $L'(\tX_{top})$ with
$a_u,\, b_u\in \Q_{>0}$.  
We set  the cohomological restrictions  $l'_i = R(l'_{i, top}) \in L'$ and assume
 that $l'_i \in S'_{an}(\tX)$ ($i=1,2$).
 We also fix a certain vertex $v\in \calv$ for which   $b_v > a_v$.
Under these assumptions,
the restricted natural line bundles $\calO_{\tX}(-l'_{1, top}) $ and  $\calO_{\tX}(-l'_{2, top})$
cannot have a common base point on  $ E_v$.
\end{proposition}

The proof is divided in several steps.

\bekezdes \label{bek:5..1}   {\bf The `standard setup'.}
 Assume the contrary, i.e., the two line bundles have  a common base point
$p$ along $E_v$. Then for both line bundles the  parts {\it (1')--(4')} of Theorem  \ref{th:NEW2}, and also all the statements
{\it (a)--(e)} of Proposition \ref{prop:BASE} are valid (for the same $p\in E_v$). Let $\omega_i$ ($i=1,2$)  be the
two forms in $H^0(\tX\setminus E,\Omega^2_{\tX})$ associated with the bundles $\calO_{\tX}(-l'_{i,top})$
and the base point $p$ as in \ref{ss:BaseLB}. We  will call this the `standard setup'.

From this we wish to obtain a contradiction using the strategy mentioned in \ref{ss:move}.

First,  via certain reductions, we add some other properties to the standard  setup.

\bekezdes \label{bek:5..2} {\bf The `$\omega$--divisorial property'.}
Consider an intersection point $q=E_u\cap E_w$, $u,w\in\calv$ and local coordinates $(x,y) $  at $(\tX,q)$ such that
$\{x=0\}=E_u$ and $\{y=0\}=E_w$. Let the local equation of $\omega_i$  be $\varphi_i(x,y)x^ay^b dx\wedge dy$, with
$a,b\in \Z$, and $x\nmid\varphi_i$ and  $y\nmid\varphi_i$. If $\varphi_i(0)=0 $ then we say that $\{\varphi_i(x,y)=0\}$ is a
divisor of $\omega_i$ at $q$. Otherwise, $\omega_i$ has no divisor at $q$.
Then one proves (similarly as the existence of good embedded resolution of plane curves)
that if we perform a conveniently chosen sequence of blow ups $b:\tX_{top, b}\to\tX_{top}$ at several infinitesimally close points of $q$, and we set $\tX_b$ and $E_b$ for the inverse images of $\tX$ and $E$,  then $b^*\omega_i$
will have no divisor at the singular points of $E_b$. Therefore, if we replace
the system $\tX_{top}, \ \tX, \ l'_{i, top}, \ \omega_i$  by $\tX_{top,b}, \ \tX_b, \ b^*l'_{i, top}, \ b^*\omega_i$,
this new system satisfies the `standard setup', but additionally also the new $\omega$--divisorial property realized by them,
i..e. they have no divisors at the singular points of $E=\cup_{v\in\calv}E_v$.
(The fact that the pair  $(\tX_{top,b},  \tX_b)$ is generic follows by the first paragraph of \ref{ss:proofpart1}, and
the stability of the properties from Proposition \ref{prop:BASE} via $b^*$  can also be verified.)  Hence, in the sequel we will
assume that in our `standard setup' from  \ref{bek:5..1} the `$\omega$--divisorial property' is also realized
for both $\omega _i$.

\bekezdes \label{bek:5..3} {\bf The poles of $\omega_i$.}  Consider a generic point $q$ of one of the exceptional curves $E_u$,
$u\in \calv$, such that at least one of the $\omega_i$'s has a non--trivial pole along $E_u$. Let $o_i$ be the pole order of
$\omega_i$ along $E_u$ (if $\omega_i$ has no pole then $o_i=0$).
Write $o_1\geq o_2$, hence $o_1>0$. Similarly as in \ref{bek:5.1.1}, we blow up $\tX$ at a generic point $q_1$ of $E_u$, and we get the new exceptional divisor $F_1$, then we blow up a generic point $q_2$ of $F_1$, etc. we repeat this $o_1-1$ times,
the last new exceptional divisor is $F_{o_1-1}$. Let denote $b$ the sequence of blow ups. Then $b^*\omega_1$ has a pole of order one along $F_{o_1-1}$. Next, we perform the deformation \ref{bek:5.1.def} at a generic point $q$ of $F_{o_1-1}$.
For the local computation see \ref{ss:move}.

Let $(x,y)$ be some local coordinates of some point on $F_{o_1-1}$  such that $F_{o_1-1}=\{x=0\}$.
Let the local equation of $\omega_i$  be
$(\psi_i(y)+x\tau_i(x,y))dx\wedge dy/x$, with $\psi_1(y) \not\equiv 0$. Consider a generic point $q$ in this interval
and we move the exceptional divisor of the blow up at $q$ as in \ref{bek:5.1.def} and \ref{ss:move}.

Since under this deformation the base points of the two perturbed line bundles must stay common, using  the notation of
Lemma \ref{lem:derivative} we have $y_1(\lambda)=y_2(\lambda)$. Therefore, Lemma \ref{lem:derivative}
applied  for both forms gives $ a_u\cdot \psi_1(q)=b_u\cdot \psi_2(q)$ ($\dagger$), and this holds for any generic point $q$.
Since $\psi_1(q)\not=0$ we get that  $\psi_2(q)\not=0$ as well. In particular, $o_1=o_2$. Furthermore
($\dagger$) says that $a_u\cdot b^*\omega_1-b_u\cdot b^*\omega_2$ has no pole along $F_{o_1-1}$. These facts reinterpreted
in $\tX$ and the exceptional curve  $E_u$ give

\begin{lemma}\label{lem:5..3} (1) If one of the forms $\omega_i$ has a pole along $E_u$ ($u\in\calv$) then both forms
have pole along $E_u$  and the pole orders are the same, say $o_u$.

(2) In the situation of (1), $a_u \omega_1-b_u \omega_2$ has a pole order strict smaller than $o_u$ along $E_u$.

(3) The above properties (1) and (2) remain true if we replace $(\tX_{top},\tX)$ by another pair obtained by
blowing up generic points of $E$ or singular points of $E$.
\end{lemma}

\bekezdes \label{bek:5..4} {\bf Reduction of $\tX$ to the support of $\omega_i$'s.}
Let $E_{\omega}$ be the union of exceptional divisors $E_u$ with $o_u>0$. A priori it can happen  that $E_\omega$ is
smaller than $E$. Let $E_{\omega,v}$ that connected component of $E_\omega$ which contains $E_v$ (cf. Proposition
\ref{prop:BASE}). Let $\tX_\omega$ be the tubular neighbourhood of $E_{\omega,v}$ in $\tX$. Then we can replace the pair
$(\tX_{top},\tX)$ by $(\tX_{top},\tX_\omega)$, and consider the restriction of the bundles $\calO_{\tX}(-l'_{i,top})$
and forms $\omega_i$ to $\tX_\omega$. Note that  the restrictions of the distinguished sections of
 the bundles $\calO_{\tX}(-l'_{i,top})$ (with transversal smooth divisors as in \ref{ss:BaseLB}) will have similar properties,
 hence the setup of Proposition \ref{prop:BASE} will be satisfied  by the  restricted objects over $\tX_\omega$.
 Furthermore, the restrictions of the $\omega_i$'s will satisfy automatically part {\it (e)}
 of Proposition \ref{prop:BASE} too,
 hence the point $p$ survives as a common base point.
 Hence, in this way we get a situation of the two bundles with  such that $E(\tX_\omega)=E_\omega$. Even more,
 using the step \ref{bek:5..1} we can reduce $\tX_{top}$ too to the neighbourhood of $\cup E_u$, where the union is over
 $\{u\in\calv_{top},\ E_u\cap E_\omega\not=\emptyset\}$.

 Hence at this point we have a situation of the `standard setup', which satisfies the
 `$\omega$--divisorial property, Lemma \ref{lem:5..3}, and $E=E_\omega$.

\bekezdes \label{bek:5..5} {\bf The proportionality of $\{a_u\}_u$ and $\{b_u\}_u$, $u\in \calv$.}
Recall that $l'_{1, top} = \sum_{u \in \calv_{top}} a_u E_u$ and $l'_{2, top} = \sum_{u \in \calv_{top}} b_u E_u$
where  $a_u,\, b_u\in \Q_{> 0}$.
 Furthermore, $b_v>a_v$, where $E_v$ is the divisor, which contains the common base point $p$.
 We claim that  $b_u/a_u=b_v/a_v$ for any $u\in \calv$.

 Since $\Gamma(\tX)$ is connected, it is enough to verify that $b_u/a_u=b_w/a_w$
 for any edge $(u,w)$ of $\Gamma$.
 Consider an intersection point $q=E_u\cap E_w$,  and local coordinates $(x,y) $  at $(\tX,q)$ such that
$\{x=0\}=E_u$ and $\{y=0\}=E_w$, as in \ref{bek:5..2}.
 Let the local equation of $\omega_i$  be $\varphi_i(x,y)dx\wedge dy/x^cy^d $, with
$c,d\in \Z_{>0}$, and $x\nmid\varphi_i$ and  $y\nmid\varphi_i$ ($i=1,2$). By the `$\omega$--divisorial property'
$\varphi_i(0,0)\not=0$. Using Lemma \ref{lem:5..3}{\it (2)} for $E_u$ we obtain that
$x| a_u \varphi_1 -b_u\varphi_2$, hence $ a_u \varphi_1(0,0) -b_u\varphi_2(0,0)=0$.
Similarly, for $E_w$,  $ a_w \varphi_1(0,0) -b_w\varphi_2(0,0)=0$. Hence $b_u/a_u=b_w/a_w$.

\bekezdes\label{bek:setu}
 Next, we continue the study of other edges $(u,w)$ of type $u\in\calv$, $w\in\calv_{top}\setminus \calv$.
Fix such an edge and set $q=E_u\cap E_w$. By  local verification it turns out that after several convenient
blow ups  at infinitesimally close points of $q$, in the new situation
$b:\tX_{top, b}\to \tX_{top}$, $u'\in \calv_b$, $v'\in \calv_{top,b}\setminus \calv_b$, we will have
$o_{u'}=1$ and the following local picture: If $(x,y)$ are local coordinates at $q'=E_{u'}\cap E_{w'}$
with local equations $\{x=0\}=E_{u'}$, $\{y=0\}=E_{w'}$ then $b^*\omega_i=\varphi_i (x,y)y^{k_i}dx\wedge dy/x$ for
certain $k_i\geq 0$, and $\varphi_i(0,0)\not=0$. [The divisor $y^{k_i}$ of $b^*\omega_i$ cannot be eliminated:
the shape of the form $y^kdx\wedge dy/x$ stays unstable with respect to blow up.]

Therefore, assume that in our setup this property is also satisfied, and we also return to our
simplified notations: $\calv_{top}$, $\calv$, $\omega_i$, etc.

\bekezdes \label{bek:pole}
Assume that $q$ is such an intersection point $E_u\cap E_w$, $u\in\calv$, $w\in\calv_{top}\setminus\calv$
and $\omega_i=\varphi_i (x,y)y^{k_i}dx\wedge dy/x$ (cf. \ref{bek:setu}),
where at leats one of the $k_i$'s is zero. Then, by Lemma
\ref{lem:5..3} {\it (2)}, $x| a_u \varphi_1(x,y)y^{k_1}-b_u \varphi_2(x,y)y^{k_2}$. This implies that
both $k_i$ are zero and $a_u\varphi_1(0,0)=b_u\varphi_2(0,0)$.

On the other hand, if we perform the deformation \ref{bek:5.1.def}  at $q$ by moving $E_w$, then, by
Lemma \ref{lem:derivative},  the base points $y_i(\lambda)$ of the `moved'  line bundles satisfy
$\frac{d}{d\lambda}\big|_{\lambda=0}y_i(\lambda)=4\pi^2\varphi_i(0,0)\cdot \{E_w$--coefficient of $l'_{top,i}\}$. Since the base points must stay together, we get that $a_w\varphi_1(0,0)= b_w\varphi_2(0,0)$. In particular,
$b_w/a_w=b_u/a_u=b_v/a_v$.  This shows that if for a pair  $u\in\calv$, $w\in\calv_{top}\setminus\calv$
we have $b_w/a_w\not =b_v/a_v$, then $o_u=1$, and both  forms $\omega_i$ necessarily must  have divisors at $q=E_u\cap E_w$.
(Note that in such a case, if we perform the above deformation at $q$, then for both line bundles the
tangent vector of the moving base point is zero.)
Set $$\cale:=\{(u,w)\,:\, u\in\calv, \ w\in\calv_{top}\setminus\calv, \ b_w/a_w>b_v/a_v\}.$$

\bekezdes We claim that (under the assumption of a common base point as above)
we can assume that $\cale\not=\emptyset$. Indeed, otherwise,  $b_u/a_u=b_v/a_v$  for all $u\in \calv$ (cf. \ref{bek:5..5}),
and $b_w/a_w\leq b_v/a_v$ for all $w\in \calv_{top}\setminus \calv$. This reads as
 $l'_{2, top} = l'_{1, top} \cdot c + \sum_{w \in \calv_{top} \setminus \calv} r_w \cdot E_w$, where $c:=b_v/a_v>1$ and
$r_w \leq 0$ for all $w$.
By cohomological restriction we get that $l'_2 = l'_1 \cdot c + l'_3$ for some  $l'_3 \in \calS'$.

On the other hand,
we know that there exists  a cycle $l_2 \in L_{>0}$, such that $l_2 \geq E_v$ and $\chi(l'_2 + l_2) = \chi(l'_2) + 1$,
or, equivalently,  $- (l'_2, l_2) + \chi(l_2) = 1$. Note also that $l'_1\in\calS'$ and $(l'_1,E_v)<0$ (cf.
Lemma \ref{lem:BASE}) imply that $(l'_1, l_2)\leq (l'_1, E_v)<0$. Therefore,
 $-(l'_2, l_2) = - c \cdot (l'_1, l_2) - (l'_3, l_2)$
 yields $ -(l'_2, l_2) >  -(l'_1, l_2)$, which means $\chi(l'_1+ l_2) < \chi(l'_1) + 1$, or $\chi(l'_1 + l_2) \leq \chi(l'_1)$.  But this contradicts the fact, that the Chern class $l'_1$ is dominant.

 \bekezdes\label{bek:alter0}
  Therefore,  if there exists a counterexample to the statement then necessarily $\cale\not=\emptyset$.
 Let us fix such a counterexample $(\tX,\tX_{top}, \omega_i, m_i=l_{min,i}, ...)$ with common base point $p\in E_v$, and
 with all the additional properties determined in the previous paragraphs. We will assume that $\cale$ is minimal.
  We fix  $(u,w)\in\cale$ ($u\in\calv$, $w\in\calv_{top}\setminus\calv$) and write   $q:=E_u\cap E_w$.
From these data  we will  construct another counterexample, which  contradicts a   necessary property of any
counterexample discussed  in \ref{bek:pole}. 
Hence,  we conclude that  counterexamples  do not exist.

(I) \ The  construction starts as follows.
 Let us blow up $q$, let $E_q$ be the new exceptional divisor. We denote the strict transforms of
$\{E_i\}_{i\in \calv_{top}}$ by the same symbols and we set $\calv_q=\calv$, $\calv_{top,q}=(\calv_{top}\setminus w)
\cup \{v_q\}$, where $v_q $ indexes $E_q$. Define  also $\tX_q$ and $\tX_{top,q}$  as the tubular neighborhoods of
$\cup_{j\in\calv_q}E_j$ and $\cup_{j\in\calv_{top,q}}E_j$.
The new Chern classes are obtained from $b^*(l_{top,i})$  by cohomologically restriction to $L'(\tX_{top,q})$. This means that
the coefficients $(a_j,b_j)$ stay unmodified for $j\not=v_q$ and the $E_q$--coefficients become
$(a_q,b_q)=(a_w+a_u, b_w+a_u)$. (This increase will be exploited later.)

 Then the form $\omega_i$  is  replaced by the restriction to $\tX_q$ of $b^*\omega_i$: it has the very same pole,
 type of divisors and local behaviour near $p$  as $\omega_i$. In particular, the equivalent conditions of
 Proposition \ref{prop:BASE} still hold in the new situation and the two bundles will have a common
 base point at $p$.

 Note also  that the self--intersection of $E_u$ decreases by one, hence the RR expression $\chi$ is modified.
 Let the RR expression of $L'(\tX_q)$ be denoted by $\chi_q$  and let us analyse the new minimal cycles
 $m_{ i, q}\in L(\tX_q)$ (their existence is guaranteed by Proposition \ref{prop:BASE}, see also \ref{ss:pole}).

 We claim that $m_{i,q}=m_i$. Indeed, let us denote the $E_u$--multiplicity of $m_{i,q}$ by $k_i$.
  By  (\ref{eq:pole}) $k_i\geq 1$.
 Then we will test $m_{i,q}$ in the context of the old situation:  we compute
 $\chi(l'_{i}+m_{i,q})-\chi(l'_{i})=\chi(m_{i,q})-(l'_{i}, ,m_{i,q})$.
 First, $  (l'_{i}, m_{ i,q})=(b^* l'_{top,i}, m_{ i,q})=(l'_{i,q}, m_{ i,q})$.
 Moreover, $\chi(m_{i,q})=\chi_q(b^*m_{i,q})=
 \chi_q(m_{i,q}+k_iE_q)=\chi_q(m_{i,q})-k_i(k_i-1)/2$.
 Since by Proposition \ref{prop:BASE} (applied for both cases)
 we get  $\chi_q(m_{i,q})-(l'_{i,q},m_{ i,q})= 1\leq  \chi(m_{i,q})-(l'_{i}, m_ {i,q})$, or
 $-k_i(k_i-1)\geq 0$, hence $k_i=1$.

 In particular, by the above computation $\chi(m_{i,q})-(l'_{i}, m_{i,q})=1$ too, hence
 by the minimality of $m_i$ we obtain  $m_i\leq m_{i,q}$, and the $E_u$--coefficient of $m_i$ is also one.
 (This is `compatible' with the fact that the pole of
 $\omega_i$ along $E_u$ is one, cf. (\ref{eq:pole}) and the comment after it.)
 Next, we start with $m_i $,  then a similar computation as above
 (and  using the fact that its $E_u$--coefficient is one) shows that
 $\chi_q(m_i)-(l'_{i,q},m_i)=1$ too, hence by the minimality of $m_{i,q}$ we get $m_{i,q}\leq m_i$  as well.

\vspace{1mm}

(II) \ We denote $E_u\cap E_q$ again by $q$, and we repeat the blow up from step (I)
 $N$ times, where $N\gg 0$ will be identified later.  In the notations we will adopt
 the previous package of notations of (I), with the difference that we write
indices $N$ instead of $q$. In particular, if $E_N$ is the very last exceptional divisor created  in the last step,
neighbour of $E_u$, then the Chern class $E_N$--multiplicities become $(a_N,b_N)=(a_w+Na_u,b_w+Nb_u)$.
We emphasize again  that the poles of the forms  and the minimal cycles $m_{i,N}=m_i$  are $N$--independent.

\vspace{1mm}

(III) \ If $N$ is very large  then $(E_u,E_u)$ becomes very small, hence
by the adjunction formula we get that the $E_u$ coefficient of $\lfloor Z_K\rfloor$ is $\leq 1$.
Recall that   for any $Z'\in L_{>0}$,
 $Z'\geq \lfloor Z_K\rfloor$ we have
 $H^0(\tX\setminus E, \Omega^2_{\tX})=H^0(\tX, \Omega^2(Z'))$ (hence any form has pole order $\leq 1$ along $E_u$)
 and $H^1(\tX,\calO_{\tX})=H^1(Z',\calO_{Z'})$ (cf. \ref{ss:LauferD}). Since the pole of $\omega_{i,N}$ along $E_u$ is one,
 we get that the $E_u$--coefficient of $\lfloor Z_K\rfloor$ is in fact one, and (by taking restriction)
 we can assume that the $E_u$-- coefficient of $Z$ is one too.

(IV) \ Next we perturb the Chern--class contributions $a_NE_N$ (resp. $b_NE_N$) in the restricted line bundles
$\calL_{1,N}=\calO_{\tX_{N}}(-\sum_{j\not=v_N} a_jE_j- a_NE_N)$ and
 $\calL_{2,N}=\calO_{\tX_{N}}(-\sum_{j\not=v_N} b_jE_j- b_NE_N)$
via Lemma  \ref{rat}.
However, we wish to preserve the common base point $p$, hence we apply Lemma \ref{rat} only after
we perform the following steps.

We blow up $B:\tX_{N,p}\to \tX_N $ at $p$, we set $Z_p:=B^*Z$, and
$B^{-1}(p)=E^{new}$.
Note that  $H^0(Z_p, B^*\calL_{i,N}(-E^{new}))_{{\rm reg}}\not=\emptyset$. Furthermore, decompose these line bundles
as a combination of the rational line bundles: $B^*\calL_{1,N}(-E^{new})=\mathfrak{L}'_1\otimes \calO_{Z_p}(-a_NE_N)$ and
 $B^*\calL_{2,N}(-E^{new})=\mathfrak{L}'_2\otimes \calO_{Z_p}(-b_NE_N)$.
 Recall that the analytic structure on $\tX_{N}$ is generic, hence the point $q=E_u\cap E_N$ is generic on $E_u$.
 Next, we choose a  sufficient large number $k$ with $k< a_N$ and  $k<b_N$, and generic points
 $q_1,\ldots, q_{k+1}\in E_u$, and  we apply Lemma \ref{rat} for both $\mathfrak{L}_{i}'$.

Let us write $\calG_{1,N}=\calO_{\tX_N}(-\sum_{j\not=v_N}a_jE_j-\sum_{1\leq j\leq k+1} r_{1,j}q_j)$ and
 $\calG_{2,N}=\calO_{\tX_N}(-\sum_{j\not=v_N}b_jE_j-\sum_{1\leq j\leq k+1} r_{2,j}q_j)$.
 That is, $B^*\calG_{i,N}(-E^{new})=\mathfrak{L}'_i\otimes \calO_{Z_p}(-\sum_{1\leq j\leq k+1} r_{i,j} q_j )$ for $i=1,2$.
 Recall that $r_{1,j}=r_{2,j}=1$ for $j\leq k$.  Then Lemma \ref{rat} says that
 $H^0(Z_p, B^*\calG_{i,N}(-E^{new}))_{{\rm reg}}\not=\emptyset$, $i=1,2$.

 Let us list some properties of the bundles $\calG_{i,N}$. They are restrictions of  natural line bundles from a generic top--level.  Indeed,
 since the $E_u$--multiplicity of $Z_p$ is one, the generic points $q_j$ can be considered as restrictions of transversal
 generic curves meeting $E_u$ at the points $q_j$. Furthermore, one can show that this can be completed to a top level
 graph $\Gamma_{N,top}$, which  admits a certain $l'_{N,top}\in L'(\Gamma_{N,top})$ as extension of this set of multiplicities.
 Next, for each $i=1,2$ independently,
 the Chern class of $\calL_{i,N}$ and $\calG_{i,N} $ agree in $L'(\tX_N)$. Since the Chern class $l'_{i,N}$ satisfies
 property $(*_v)$, the Chern class of  $\calG_{i,N}$ must satisfy too. In particular, $\calG_{i,N}$ satisfies part {\it (4')}
 of the main Theorem \ref{th:NEW2}, that is, it has $-(l'_{i,N}, E_v)$ base points on $E_v$,
 and these base points are exactly at the zeros of a section (basically unique), which has no fixed components.
 Since   $H^0(Z_p, B^*\calG_{i,N}(-E^{new}))_{{\rm reg}}\not=\emptyset$, this section vanishes at $p$, hence
 $p$ is a base point of both $\calG_{i,N}$, cf. \ref{rem:*}(a).
 Hence in this situation we have again a common base point of two
 restricted natural line bundles. In particular, each line bundle (and the base point $p$) determines
  a form $\omega_{\calG,i}$ by Proposition \ref{prop:BASE}. Furthermore,
  for this pair of bundles all the additional properties listed in
 \ref{bek:5..5}--\ref{bek:pole} hold. Note that the ratio of the Chern class multiplicities at the points $q_j$ ($j\leq k$)
 are $r_{2,j}/r_{1,j}=1/1=1$, which differ from $b_u/a_u=b_v/a_v>1$. Hence, by \ref{bek:pole} at each  point
 $q_j $   ($j\leq k$) both forms must have a divisor.

\bekezdes \label{bek:alter}
 Well,  this fact will lead to a contradiction. The point is that the number of divisors is bounded
 independently of $N$, and the number $i$ can be arbitrary large when $N$ is large.
 Indeed, first notice that the minimal cycle $m_i$, which satisfies property $(*_v)$ is the same for
  $\calG_{i,N}$ and $\calL_{i,N}$. Thus,
  (\ref{eq:pole}) shows  that the common pole  $P$ of $\omega _{\calG,i}$ (cf. Lemma \ref{lem:5..3})
  is $\leq m_i$. If $E_u$ is not in the support of $P$, then we can reduce the situation to a smaller resolution space
  as in \ref{bek:5..4}, but in this way we eliminate from $\cale$ the point $q$, hence we decrease $\#\cale$.
  But $\#\cale$ was taken minimal, hence necessarily $E_u$ is in the support of $P$ with multiplicity one (since $E_u$
  has multiplicity one in $m_i$).

   Write $P=E_u+P'$.
If $D_i$ is the union of the divisors of $\omega_{\calG,i}$ along $E_u$, then $(D_i-P-K_{\tX_N}, E_u)=0$.
Or, $(D_i, E_u)=(P'+E_u+K_{\tX_N}, E_u)=(P',E_u)-2\leq (m_i-E_u, E_u)-2$, an
$N$--independent bound. Hence  $i$  can be taken larger than the
possible number
of the divisors of the forms.

This ends the proof of part {\it (5')}.

\begin{remark} In the above proof, in order to get a contradiction,
in \ref{bek:alter}
we used the structure of the divisors  of the forms.
An alternative argument, which replaces \ref{bek:alter},  might run as follows.

Consider the situation from the end of \ref{bek:alter0}, in particular the bundles $\calG_{i,N}\in
{\rm Pic}(\tX_N)$.  Set $W:=Z-E_u$. Note that $W$ is supported on $\cup_{i\in\calv\setminus u}E_i$, and all its coefficients are large.
We claim that if the pair $\{\calG_{i,N}\}_{i=1,2}$ is a counterexample (for any choice of the points of $q_{i,j}$'s)
then the pair  $\{\calG_{i,N}|_W\}_{i=1,2}$ is a counterexample too. However, by taking this restriction, the set $\cale$
decreases, a fact which contradicts its minimality. In order to verify the claim, the verification of the
requirements regarding the Chern classes are immediate. What should be explain is the fact that  $\{\calG_{i,N}|_W\}_{i=1,2}$
have  a common  base point at (the very same) $p$.

Fix $i\in \{1,2\}$. Then  note that the restrictions bundles  $\calG_{i,N}|_W$
 do not  depend on the choice of the generic points $q_{i,j}$ on $E_u$.
Next,  consider the restriction maps $r_i : \pic^{l'_{i, top}}(Z) \to \pic^{l'_{i, top}}(W)$.
If the integer $M$ is large enough, then by \cite[Th. 6.1.9]{NNI}
the image of the Abel map $ c^{-M E_{u}^*}(Z) : \eca^{-M E_{u_2}^*}(Z) \to \pic^{-M E_{u}^*}(Z) $ is an affine space of dimension $h^1(\calO_Z) - h^1(\calO_W)$, the dimension of the fibers of $r_i$. Here we take $M=a_N$ if $i=1$ and
$M=b_N$ if $i=2$.
Thus, when  we move the generic points $q_{i,j}$,  the line bundle $\calG_{i,N}$ cover an open set in $r_i^{-1}(\calG_{i,N}|_W)$. Since $\calG_{i,N}$  (for any choice of $q_j$) has a  base point at $p$, we get that the
generic bundle in $r_i^{-1}(\calG_{i,N}|_W)$ has a base point at $p$.

Now assume that $\calG_{i,N}|_W$ has no base point at $p$. Then the divisor ${\rm div}(s)$ of its   generic section
$s$ is not supported at $p$. Let $D$ be ${\rm div}(s)$ completed by $M'=-(l'_{i,top}, E_u)$ generic points on $E_u$.
Then, when we move $s$ and the $M'$  generic points,  $\calO_Z(D)$ covers an open set in
$r_i^{-1}(\calG_{i,N}|_W)$.
 But, by construction, this
bundle has a section which does not vanish at $p$, a fact which contradicts the previous paragraph.
\end{remark}

\section{Examples}\label{s:examples}

Below $\tX$ is a normal surface singularity whose  link is a rational homology sphere.

\subsection{The case of Chern class  $l'=Z_K$}\label{ss:zk} Let us fix a resolution graph $\Gamma$.
If $l'\in L'$ is `sufficiently negative' (i.e., if each  $(l', E_v)$ is sufficiently negative for all $v\in \calv$)
then {\it for any analytic structure supported by $\Gamma$ any line bundle $\calL\in \pic(\tX)$ with Chern class $l'$}
 is base point free; in particular,
$l'\in \calS_{an}'\setminus \{0\}$ too. For different negativity conditions (imposed by different proofs)
see e.g. \cite[Th. 4.1]{CNP},
\cite[Th. 3.1]{LaW}, \cite[Th. 2, Prop. 4]{S-B}.
The condition $l'\in \calS_{an}'\setminus \{0\}$ (versus base point freeness)
can be guaranteed by weaker assumptions, in general
we require slightly stronger  negativity than being in $Z_K+\calS'$.
 However, none of these  combinatorial assumptions are  satisfied in general
by $Z_K$. In the next paragraphs we analyse with details exactly this case of $l'=Z_K$.

Assume that $\tX$ is minimal, i.e. it contains no $(-1)$--curve. Then, by adjunction formula, $Z_K\in\calS'$.
(Recall also that $Z_K=0$ happens exactly when $\Gamma$ is ADE.) We claim that {\it  if
$\tX$ is generic and $\Gamma$ is not ADE  then $Z_K\in \calS'_{an}\setminus \{0\}$} (that is,
$\calO_{\tX}(-Z_K)$ has no fixed components).

In the proof  we use \ref{th:OLD}{\it (f)}: we need to show that
$\chi(Z_K+l)>\chi(Z_K)$ for any $l>0$. This reads as  $\chi(-l)>0$. Note that from $-l$ there exists
$\chi$--nonincreasing generalized  Laufer computation sequence which connects $-l$ to 0, cf. \cite[\S 7]{NOSZ}
or \cite[4.3.3]{trieste}.
Hence $\chi(-l)\geq 0$ (see also \cite[Prop. 5.7]{NOSZ}).
However, if $\chi(-l)=0$, then  the sequence is necessarily $\chi$--constant, hence at the very last step one has
$\chi(-E_u)=0$ for some $u\in \calv$. But this means $E_u^2=-1$, a contradiction.

Note that for an arbitrary analytic structure it is not true that $Z_K\in \calS'_{an}\setminus \{0\}$,
cf. next example.

\begin{example}\label{ex:TWOBASEDIFF}
Consider  the following  $\Gamma$,
where the $(-2)$--vertices are unmarked.

\begin{picture}(200,50)(-20,0)
\put(170,30){\circle*{4}}\put(190,30){\circle*{4}}\put(210,30){\circle*{4}}\put(230,30){\circle*{4}}
\put(150,30){\circle*{4}}
\put(70,30){\circle*{4}}
\put(210,40){\makebox(0,0){\small{$-3$}}}
\put(230,20){\makebox(0,0){\small{$E_1$}}}
\put(210,20){\makebox(0,0){\small{$E_2$}}}
\put(90,30){\circle*{4}}
\put(110,30){\circle*{4}}
\put(130,30){\circle*{4}}
\put(110,10){\circle*{4}}
\put(70,30){\line(1,0){160}}\put(110,10){\line(0,1){20}}
\end{picture}

It is an elliptic (integral homology sphere) graph  $Z_{min}=E_1^*$ and $Z_K=E_2^*$, $Z_{min}<Z_K$.
The  length of the elliptic sequence is two
(for terminology see e.g.  \cite{Laufer77,weakly,Nfive}), hence $1\leq p_g\leq 2$, and $\Gamma$
 supports  two rather different families  of  analytic structures according to the value of $p_g$.
E.g. $\Gamma$   can be realized even by the hypersurface singularity
$x^2+y^3+z^{11}=0$. In this case $Z_{max}=Z_{min}=E_1^*$, it is the divisor of $z$. In fact, $p_g=2$,
${\rm mult}(X,o)=2$  and
$Z_{max}=Z_{min}$ is true for any Gorenstein structure, cf. \cite{weakly,Nfive}.
However, if $Z_{max}=Z_{min}$ then $Z_K\not\in \calS_{an}$. More precisely, by a topological argument on this $\Gamma$,
(and for any analytic structure supported on this $\Gamma$)
$Z_{min}$ and $Z_K$ cannot by simultaneously elements of $\calS_{an}$. Indeed, if both are realized by some functions,
say $f$ and $g$,  then (since $-(Z_K,Z_{min})=1$) the degree of the map $(f,g):(X,o)\to (\C^2,0)$ is one.
But this  can occur only for smooth germs $(X,o)$, which is not the case.

However, as we already proved in \ref{ss:zk}, for the generic analytic structure
$Z_{max}=Z_K$ (hence $Z_{min}\not\in \calS_{an}$). (In this case $p_g=1$ by Theorem \ref{th:OLD}{\it (c)}
and $(X,o)$ is non--Gorenstein \cite[6.9]{NNII}.)

Since  $(Z_K,E_2)=-1$ and $\chi(Z_K+Z_{min})=1=\chi(Z_K)+1$,
by Theorem \ref{th:NEW2} $\calO_{\tX}(-Z_K)$ has a (unique) base point on $E_2$.
Note that $Z_K+Z_{min}\in\calS_{an}$ is the Chern class $l'+l$ of part {\it (5')} in  Theorem \ref{th:NEW2}.
Furthermore, $(Z_K+Z_{min},E_2)=-1$ and $\chi(2Z_K)=2=\chi(Z_K+Z_{min})+1$, hence
$\calO_{\tX}(-Z_K-Z_{min})$ has a (unique) base point on $E_2$ too. However,
by Theorem \ref{th:NEW2}, the two base points are different.
Note also that ${\rm mult}(X,o)=-Z_{max}^2+1=3$.
\end{example}

\subsection{Base points of $Z_{max}$ and the multiplicity in the generic elliptic case.} \
Assume that $\min\chi=0$. In this case $\Gamma$ is either rational or elliptic (see e.g. \cite{Nfive}).
In the rational case   $Z_{max}=Z_{min}$,
$\calO_{\tX}(-Z_{max})$  has no base points, and  ${\rm mult}(X,o)=-Z_{min}^2$
independently of the analytic structure supported on $\Gamma$ \cite{Artin62,Artin66}
(see also \ref{ss:proofrat} with the compatibility with our criterions).

In the sequel we assume that $\Gamma$ is elliptic.
For the simplicity of the presentation we also assume that $\Gamma$ is numerically Gorenstein (i.e.
$Z_K\in L$), and that $\Gamma$ is the dual graph of a minimal good resolution, which is minimal
(contains no $(-1)$--curves). Let $C$ be the minimally elliptic cycle (for the standard notations and  combinatorial
properties of elliptic graphs see e.g. \cite{Laufer77,weakly,Nfive}).
We claim that following facts hold, whenever $\tX$ is {\it generic}:

{\it (1) $Z_{max}=Z_K$.

(2) $\calO_{\tX}(-Z_{max})$ has a base point if and only if $C^2=-1$. Moreover, if $C^2=-1$ then
$\calO_{\tX}(-Z_{max})$ admits a unique base point (of type $A_1$).
(For the peculiar structure of the graph when $C^2=-1$
and the position of the base point see the discussion below.)}

We sketch the arguments.
For
{\it (1)} we use Theorem \ref{th:OLD}{\it (g)} and we verify that $Z_K=\max{\mathcal M}$.  Indeed,  $\chi(Z_K)=0$ and
$\chi(Z_K+l)=\chi(-l)>0$ for any $l>0$ (cf. \ref{ss:zk}).

For {\it (2)} fix some $E_v$ such that $(Z_K,E_v)<0$  and $\chi(Z_K+E_v+l)=1$  for some $l>0$.
Then $1=\chi(E_v+l)-(Z_K,E_v)-(Z_K,l)$ with $\chi(E_v+l)\geq 0$ (ellipticity), $-(Z_K,E_v)>0$ (assumption),
$-(Z_K,l)\geq 0$ ($Z_K\in\calS$). Hence necessarily (a) $\chi(E_v+l)=0$, (b) $(Z_K,E_v)=-1$, (c) $(Z_K,l)=0$.
From (b) follows that $E_v^2=-3$, from (c) we obtain that $(Z_K,E_w)=0$  for any on $E_w$ from
 the support $|l|$,  hence $|l|$ consists of $(-2)$ curves (in particular $E_v\not\in |l|$), and (a) implies that $(l, E_v)=\chi(l)+1\geq 1$, hence $E_v$ is adjacent with $|l|$.
Since $\chi(E_v+l)=0$, by the definition of $C$, one has $E_v+l\geq C$.
Since $C$ cannot have only $(-2)$--curves ($\dagger$)  $E_v\leq C\leq E_v+l$.
In particular,  $E_v$ is uniquely determined by this property.

Hence, by ($\dagger$),  $C$ itself has the form
$E_v+l_0$, where $E_v\not\in|l_0|$, $E_v$ is adjacent to $|l_0|$, and $|l_0|$ consists of $(-2)$ curves.
Then $l_0$ verifies (a)-(b)-(c), i.e. $\chi(Z_K+E_v+l_0)=1$ (and $l_0$ is minimal with this property).
Since by general theory $C^2=(C,Z_K)$,
$C^2=(C, Z_K)=(E_v+l_0,Z_K)=-1$.

All the possible graphs of elliptic cyles $C$  with $C^2=-1$ are listed in \cite{Laufer77}.


Finally observe that if for a generic singularity with arbitrary graph $\Gamma$, if
 $Z_K=Z_{max}$ then by Theorem \ref{th:OLD}{\it (g)} $\Gamma$ is necessarily elliptic
 (hence {\it (1)} above is an `if and only if' characterization).

%
 %

\begin{example} Consider the following non--elliptic plumbing graph.
It has $\min \chi =-1$. It supports several
analytic structures, the possible values for the geometric genus are $1-\min\chi=2\leq p_g\leq 3$, cf. \cite{NO17}.
If $\tX$ is generic then $p_g=2$ and $Z_{max}=2E_v^*$. 

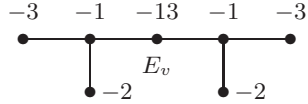
\begin{picture}(300,45)(0,0)
\put(125,25){\circle*{4}}
\put(150,25){\circle*{4}}
\put(175,25){\circle*{4}}
\put(200,25){\circle*{4}}
\put(225,25){\circle*{4}}
\put(150,5){\circle*{4}}
\put(200,5){\circle*{4}}
\put(125,25){\line(1,0){100}}
\put(150,25){\line(0,-1){20}}
\put(200,25){\line(0,-1){20}}
\put(125,35){\makebox(0,0){\small{$-3$}}}
\put(150,35){\makebox(0,0){\small{$-1$}}}
\put(175,35){\makebox(0,0){\small{$-13$}}}
\put(200,35){\makebox(0,0){\small{$-1$}}}
\put(225,35){\makebox(0,0){\small{$-3$}}}
\put(160,5){\makebox(0,0){\small{$-2$}}}
\put(210,5){\makebox(0,0){\small{$-2$}}}
\put(175,15){\makebox(0,0){\small{$E_v$}}}

\end{picture}

There is only one $E_u$ with $(E_u,Z_{max})<0$, namely $E_v$, and $(E_v,Z_{max})=-2$. Moreover, $Z_K\geq E_v+Z_{max}$ and
 $\chi(Z_K)=0=\chi(Z_{max})+1$. Hence $Z_{max}$ has two base points on $E_v$ and
 ${\rm mult}(X,o)=-Z_{max}^2+2=6$.
(This is compatible with \cite{NO17}.)

We wish to emphasize that there exists a Gorenstein (even complete intersection) analytic structure
supported on $\Gamma$, which has the very same $Z_{max}=2E^*_v$, however in that case
$\calO_{\tX}(-Z_{max})$ has no base points, hence ${\rm mult}(X,o)=-Z_{max}^2=4$ (and $p_g=3$).

Furthermore, there exists also a (Kodaira/Kulikov) type analytic structure supported on $\Gamma$ with a
smaller maximal ideal cycle, namely $Z_{max}=Z_{min}=E_v^*$. In this case $Z_{max}^2=-1$,
$\calO_{\tX}(-Z_{max})$ has a unique base point of $A_2$--type on $E_v$, hence ${\rm mult}(X,o)=3$.
(In this case $p_g=3$ too.) For details see \cite{NO17}.
\end{example}

\section{Generic line bundles of arbitrary  singularities.}\label{s:GEN-GEN}

\subsection{} In \cite{NNI} we fixed an analytic type $\tX$ (not necessarily generic)  and
 we determined combinatorially  several cohomological properties of generic line bundles
$\calL_{gen}\in\pic^{-l'}(\tX)$. On the other hand, for a fixed resolution graph $\Gamma$,
the philosophy/aim  of \cite{NNII} was to show that
(restricted) natural line bundles with given Chern class, associated with generic analytic  structures
supported on $\Gamma$, behave cohomologically as the generic line bundles
(of an arbitrary singularity) with the same Chern class.

In the present note, in Theorem \ref{th:NEW2} we establish several properties of (restricted) natural line bundles
of generic singularities.
It is natural to ask whether these properties are valid for generic line bundles of an arbitrary  singularity.
The next theorem answers positively, provided that an additional assumption  is satisfied.
(For the fact that some restriction is needed see Remark \ref{rem:m=m'}.)

Recall the following fact (see Remark \ref{rem:M+} and paragraph \ref{ss:pole}).
\begin{lemma}\label{minimumcycle}
Fix a resolution graph $\Gamma$  and a dominant Chern class $l' \in \calS'$. Assume that
for  $v \in \calv$ the identity  $\min_{l \geq E_v} \chi(l' + l) = \chi(l') + 1$ holds.
Then the set of cycles $x$ which satisfy both
 $ x \geq E_v$  and  $\chi(l' + x) =  \chi(l') + 1$  has  a unique maximal element  $l$
and a unique  minimal element  $m$.
\end{lemma}

%

\begin{theorem}\label{th:GENGEN}
Let  $\tX$ be a resolution of an arbitrary singularity (with rational homology sphere link).
Fix $l'\in \calS'\setminus \{0\}$ such that
$c^{-l'}(Z)$ is dominant for $Z\gg 0$.  Then the properties {\it (1')--(4')}
of Theorem \ref{th:NEW2} hold for
a  generic element $\calL_{gen}$ of  $\pic^{-l'}(\tX)$ (instead of $\calL$ of Theorem \ref{th:NEW2}).

Suppose that a generic line bundle $\calL_{gen}$ of  $\pic^{-l'}(\tX)$ has  a base point on the exceptional divisor $E_v$ and $l$ is the largest cycle such that $l \geq E_v$ and $\chi(l' + l) = \chi(l') + 1$ and $m$ is the minimal cycle such that
$m \geq E_v$ and $\chi(l' + m) = \chi(l') + 1$. (In particular, $l'+l$ is dominant too.)

Then, if the property {\it (5')} fails for the line bundle $\calL_{gen}$ then
$m$ is necessarily the minimal cycle associated with the dominant Chern class $l'+l$ too,
that is, it is the minimal cycle  $m \geq E_v$ satisfying  $\chi(l' + l + m) = \chi(l' + l) + 1$.
\end{theorem}

\begin{proof} We can assume that $(X,o)$ is not rational, otherwise the argument from \ref{ss:proofrat} holds identically.

 Take a generic divisor $\widetilde{D}$ with Chern class $-l'$. Then all components of $\widetilde{D}$ are smooth,
$\widetilde{D} $ intersects $E$ transversally, and $\calL_{gen}=\calO_{\tX}(\widetilde{D})$ satisfies {\it (1')--(2')}.

Next we prove  {\it (3')} and {\it (4')} in a slightly  more general context;
we will need this version  in the proof of {\it (5')} as well. The generalized statement is  the following:

\begin{lemma}\label{lemma:UJ}
Let $Z$ be an arbitrary cycle on $\tX$ and a Chern class $l'\in \calS'\setminus \{0\}$ such that
$c^{-l'}(Z)$ is dominant and let us  consider  a generic line bundle $\calL_{gen}$ of  $\pic^{-l'}(Z)$ as well.

(a) If $v \in |Z|$ and $\min_{Z \geq l \geq E_v} \chi(l' + l)  - \chi(l') \geq 2$, then $\calL_{gen}$ does not have
any  base point along the exceptional divisor $E_v$.

(b)  If $v \in |Z|$ and $\min_{Z \geq l \geq E_v} \chi(l' + l)  - \chi(l') = 1$, then $\calL_{gen}$ has exactly $-(l', E_v)$ base points along the exceptional divisor $E_v$.
\end{lemma}

\begin{proof}\ {\it (a)}\
 If $\calL$ is generic then by Theorems \ref{th:dominant} and \ref{th:OLD}{\it (f)}
$h^1(Z,\calL)=0$, and by Theorem 5.3.1 of \cite{NNI} (see also Theorem \ref{th:OLD}{\it (d)} from above)
$$ h^1(Z - E_v, \calL(-E_v)) = \chi(l'+E_v)-\min_{Z \geq l \geq E_v}\,\chi(l'+ l) .$$
Therefore, from the exact sequence $0\to \calL(-E_v)|_{Z-E_v}\to \calL\to \calL|_{E_v}\to 0$ we get

\begin{equation*}\label{eq:geq2}
\dim\,H^0(Z, \calL)/H^0(Z - E_v, \calL(-E_v)) =-(l',E_v)+1-h^1(Z - E_v, \calL(-E_v))=\min _{Z \geq l\geq E_v}  \, \chi(l'+l)-\chi(l')\geq 2.\end{equation*}
Thus, if $\widetilde{D}$ is a generic divisor of $\tX$ with Chern class $-l'$ and
$\widetilde{D}\cap E_v=\{p_1,\ldots, p_k\}$ ($k=-(l',E_v)$),
then not all the points $p_i$ are base points of $\calL=\calO_Z(\widetilde{D})$.
We wish to show that in fact none of them is a base point.
This basically will follow from the irreducibility of an incidence space.

We consider two incidence spaces
$$ \frII = \{\,   (p, D) \in E_v \times \eca^{-l'}(Z)  :  p \in |D|, \
\mbox{$D=\widetilde{D}|_Z$ and $\widetilde{D}$ intersects $E$ transversally}
\},$$
$$  \frII_b  = \{\,   (p, D)\in \frII  :  p \ \mbox{is a base point of $\calO_Z(D)$}\,\}.$$
Let $\pi_2:\frII\to \eca^{-l'}(Z)$ be the second projection, and let $\pi_{2,b}$ be its restriction
to $\frII_b$. They are morphisms with finite fibers.  If $c^{-l'}\circ \pi_{2,b}$ is not dominant, then for
$\calL \in \pic^{-l'}(Z) $ generic the fiber $(c^{-l'}(Z) \circ \pi_{2,b})^{-1}(\calL)=\emptyset$, hence we are done.
Hence, in the sequel we assume that $c^{-l'}(Z) \circ \pi_{2,b}$ is dominant. Then we can fix a non--empty Zariski open set
$U$ in $\pic^{-l'}(Z)$ such that $c^{-l'}(Z)$  and $c^{-l'}(Z) \circ \pi_2$ are ($C^{\infty}$) fibrations over $U$ and
$\pi_2$ is a regular covering over $U':=(c^{-l'}(Z))^{-1}(U)$. Furthermore, we can assume that the same facts  are true for
the restriction $\pi_{2,b}$ and  for the very same $U$.  We will replace the spaces $\frII$ and $\frII_b$ with their subspaces
sitting over $U$.

We claim that $\frII$ is irreducible. Indeed, $U$ is irreducible, all the fibers of $c^{-l'}(Z)$ are irreducible
(cf. \ref{ss:4.1}), hence $U'$ is irreducible. We need to show that the total space of the regular covering
$\frII\to U'$ is irreducible.
For this fix a divisor $\widetilde{D}$ with $\widetilde{D}\cap E_v=\{p_1,\ldots, p_k\}$
as above. Then, moving along a path the components of the divisor (hence the intersection points
$\{p_i\}_i$)   there exists a (monodromy)
path in $\frII$ such that the starting point corresponds to a fixed order of
$\{p_1,\ldots, p_k\}$ and the ending point any permutation of them.
(Here we need the fact that  the regular part of $E_v$ is also connected, and that any real one--dimensional
path in $\pic^{-l'}(Z)$ can be perturbed to be in $U$.) This shows that the covering $\pi_2$ over $U'$ is irreducible,
hence $\frII$ is irreducible.

On the other hand, the covering $\frII_b$ is a proper subspace of $\frII$, since not all the points $\{p_i\}_i$ are base points. This contradicts the irreducibility of $\frII$.
 This ends the proof of part {\it (a)}.

For part  {\it (b)} notice that from the exact sequence
$0\to \calL(-E_v)|_{Z-E_v}\to \calL\to \calL|_{E_v}\to 0$
we get that the dimension of the image of the map $H^0(Z, \calL) \to H^0(E_v, \calL|_{E_v})$
equals $\dim\,H^0(Z, \calL)/H^0(Z - E_v, \calL(-E_v)) =  \min _{Z \geq l\geq E_v}  \, \chi(l'+l)-\chi(l') = 1$.

This  means that every section of $\calL$ vanishes at the same points of the exceptional divisor $E_v$, so
hence  the line bundle $\calL$ has  $-(l', E_v)$ base points on $E_v$ ($\{p_1,\ldots, p_k\}$ as above).
\end{proof}

 Finally we consider property  {\it (5')}: we have to investigate
whether the line bundles $ \calL_{gen}$ and $\calL_{gen}(-l)$ have a common base point on  the exceptional divisor $E_v$.

Notice that $\calL_{gen}(-l)$  is also a generic line bundle in $\pic^{l' + l}(\tX)$ and $c^{l' + l}$ is dominant
(since $l'+l\in \calS'_{an}$ and then via Theorems \ref{th:OLD}{\it (f)} and \ref{th:dominant}) .

Assume in the following that for a generic line bundle $\calL_{gen} \in \pic^{l'}(\tX)$
the line bundles $ \calL_{gen}$ and $\calL_{gen}(-l)$ have a common base point on  the exceptional divisor $E_v$.

Consider the minimal cycle $m \geq E_v$ such that $E_v \leq m $ and $\chi(l' + m) = \chi(l') + 1$, and also
the minimal cyle $m'$ such that $m' \geq E_v$ and  $\chi(l' + l + m') = \chi(l' + l) + 1$.
 We claim that $m = m'$.

To prove the claim, assume first that $m \ngeq  m'$.
By Lemma \ref{lemma:UJ}{\it (b)}
 a generic line bundle $\calL'_{gen} \in \pic^{-l'}(m)$
  has  $k:=-(l', E_v)$ disjoint base points  $q_1, q_2, \cdots, q_{k}$ on the exceptional divisor $E_v$.
On the other hand, using $m \ngeq  m'$,  Lemma \ref{lemma:UJ}{\it (a)}  and Lemma \ref{minimumcycle},
 the generic line bundle $\calL'_{gen}(-l) \in \pic^{-l' - l}(m)$ has no base points on the exceptional divisor $E_v$.

Now, if we consider  the restriction maps $r_1:  \pic^{-l'}(\tX) \to \pic^{-l'}(m)$ and $r_2:
 \pic^{-l'-l}(\tX) \to \pic^{-l'-l}(m)$ and a generic line bundle $\calL_{gen}$ in $ r_1^{-1}(\calL'_{gen})$,
then $\calL_{gen}$ is a generic line bundle in $\pic^{-l'}(\tX)$ and it has  base points
 $q_1, q_2, \cdots, q_k$.
On the other hand,  the line bundle $\calL_{gen}(-l)$ is a generic line bundle in  $ r_2^{-1}(\calL'_{gen}(-l)) \subset \pic^{-l' - l}(\tX)$ and has no base point at $q_1, q_2, \cdots, q_k$, since the line bundle $\calL'_{gen}(-l) \in \pic^{-l' - l}(m)$ has no base points on the exceptional divisor $E_v$.
This contradicts to the assumption that the line bundles $ \calL_{gen}$ and $\calL_{gen}(-l)$ have  a common base point
on  $E_v$. In particular, it  proves that $m\geq m'$.
The verification of the other case $m'\geq m$  is completely identical. Hence,  indeed,  $m = m'$.
\end{proof}

\begin{remark}\label{rem:m=m'}
Property {\it (5')} does not hold in general in the previous theorem (without the `$m=m'$ assumption').
Indeed, let us  consider  again the graph from  Example \ref{ex:TWOBASEDIFF} with
 the Chern class $l' = Z_K$ and with a Gorenstein analytic structure.

We know that $Z_K$ is dominant,   $(Z_K,E_2)=-1$ and $\chi(Z_K+Z_{min})=1=\chi(Z_K)+1$. Therefore,
the generic line bundle  in $\pic^{-Z_K}(\tX)$ has  a base point on the exceptional divisor $E_2$.

One can verify  that the elliptic cycle $C$ is the minimal cycle $l$
such that $l \geq E_2$ and $\chi(Z_K+ l)=1=\chi(Z_K)+1$,  and $Z_{min}$ is the maximal cycle $l$
such that $l \geq E_2$ and $\chi(Z_K+ l)=1=\chi(Z_K)+1$.

Notice also that the generic line bundle in $\pic^{-Z_K- Z_{min }}(\tX)$ also has a base point on $E_2$ and $C$ is the minimal cycle such that $C \geq E_2$ and $\chi(Z_K + Z_{min} + C)=1=\chi(Z_K + Z_{min} )+1$.

Now, we show that property {\it (5')} does not hold for a generic line bundle $\calL \in \pic^{-Z_K}(\tX)$,
that is,  the line bundles $\calL$ and $\calL(-Z_{min})$ have a common
base point on the exceptional divisor $E_2$.

Indeed, by the previous Lemma \ref{lemma:UJ}
already the line bundles $\calL | C$ and $\calL(-Z_{min}) |C$ have one  base point on the exceptional divisor $E_2$.
On the other hand,  these two restricted line bundle are the same because the `obstruction line bundle'
 $\calO_C(Z_{min})$ vanishes in the  Gorenstein case, cf. \cite{weakly}.
This means that the base points of the line bundles  $\calL | C$ and $\calL(-Z_{min}) |C$ coincide, which is
 the common base point of the line bundles $\calL$ and $\calL(-Z_{min})$ as well.

\end{remark}

\end{document}